\documentclass[12pt]{amsart}

\usepackage[margin=1in]{geometry}
\usepackage[utf8]{inputenc}
\usepackage{bm}
\usepackage{tikz}
\usetikzlibrary{arrows}
\usepackage{tikz-cd}
\usepackage{pgffor}
\usepackage{amssymb,amsmath,amsfonts,mathrsfs,amsthm}
\usepackage{enumitem}
\usepackage{adjustbox}
\usepackage{hyperref}
\usepackage{relsize}
\usepackage{graphicx}
\usepackage[font={small}]{caption}
\usepackage{float}
\usepackage{subcaption}
\usepackage{changepage}
\usepackage{changes}
\usepackage[toc]{appendix}

\newcommand{\R}{\mathbb{R}}

\newcommand{\Z}{\mathbb{Z}}
\newcommand{\C}{\mathbb{C}}
\newcommand{\N}{\mathbb{N}}
\newcommand{\Hom}{\text{Hom}}

\newcommand{\hocolim}{\text{hocolim}}
\newcommand{\id}{\text{id}}

\renewcommand{\l}{\ell}

\renewcommand{\sl}{\mathfrak{sl}}

\newcommand{\B}{\mathscr{B}}
\newcommand{\Cs}{\mathcal{C}}

\renewcommand{\d}{\partial}

\newcommand{\q}{\mathfrak{q}}
\newcommand{\Zq}{\Z[\q,\q^{-1}]}

\newcommand{\Tot}{\text{Tot}}
\newcommand{\BN}{\mathcal{BN}}
\newcommand{\BBN}{\mathcal{BBN}}
\newcommand{\A}{\mathbb{A}}
\newcommand{\ot}{\leftarrow}
\newcommand{\2}{\underline{2}}
\newcommand{\Top}{\text{Top}}
\newcommand{\F}{\mathcal{F}}
\newcommand{\til}[1]{\widetilde{#1}}
\renewcommand{\b}[1]{\overline{#1}}
\newcommand{\h}[1]{\widehat{#1}}
\renewcommand{\c}{\circ}
\newcommand{\Ab}{\mathcal{A}}
\newcommand{\Mod}{\text{Mod}}

\renewcommand{\k}{\Bbbk}
\newcommand{\qdeg}{\text{\emph{q}deg}}
\newcommand{\adeg}{\text{adeg}}
\newcommand{\edge}{e}
\newcommand{\X}{\mathcal{X}} 
\newcommand{\lr}[1]{\vert {#1} \vert}
\newcommand{\LR}[1]{\lVert{#1}\rVert}
\newcommand{\Cscr}{\mathscr{C}}
\newcommand{\U}{U_{\q}(\sl_2)}
\newcommand{\I}{\mathbb{I}}
\newcommand{\Is}{\mathcal{I}}
\newcommand{\Js}{\mathcal{J}}

\newcommand{\lar}[1]{\xleftarrow{#1}}
\newcommand{\rar}[1]{\xrightarrow{#1}}

\usepackage{chngcntr}

\newtheorem{theorem}{Theorem}[section]
\newtheorem{lemma}[theorem]{Lemma}
\newtheorem{proposition}[theorem]{Proposition}
\newtheorem{corollary}[theorem]{Corollary}
\newtheorem{conjecture}[theorem]{Conjecture}

\theoremstyle{remark}
\newtheorem{example}[theorem]{Example}

\theoremstyle{remark}
\newtheorem{remark}[theorem]{Remark}

\theoremstyle{definition}
\newtheorem{definition}{Definition}[section]

\numberwithin{equation}{section}

\begin{document}

\title{Stable Homotopy refinement of Quantum Annular homology}

\author[R. Akhmechet]{Rostislav Akhmechet}
\address{Department of Mathematics, University of Virginia, Charlottesville VA 22904-4137}
\thanks{RA was supported by NSF RTG Grant DMS-1839968}
\email{\href{mailto:ra5aq@virginia.edu}{ra5aq@virginia.edu}}

\author[V. Krushkal]{Vyacheslav Krushkal}
\address{Department of Mathematics, University of Virginia, Charlottesville VA 22904-4137}
\thanks{VK was supported by NSF Grant DMS-1612159.}
\email{\href{mailto:krushkal@virginia.edu}{krushkal@virginia.edu}}

\author[M. Willis]{Michael Willis}
\address{Department of Mathematics, University of California, Los Angeles, CA 90095}
\thanks{MW was supported by NSF FRG Grant DMS-1563615.}
\email{\href{mailto:msw188@ucla.edu}{msw188@ucla.edu}}

\maketitle

\begin{abstract}
We construct a stable homotopy refinement of quantum annular homology, a link homology theory introduced by Beliakova, Putyra and Wehrli. For each $r\geq 2$ we associate to an annular link $L$ a naive $\Z/r\Z$-equivariant spectrum whose cohomology is isomorphic to the quantum annular homology of $L$ as modules over $\Z[\Z/r\Z]$.
The construction relies on an equivariant version of the Burnside category approach of Lawson, Lipshitz and Sarkar. The quotient under the cyclic group action is shown to recover the stable homotopy refinement of annular Khovanov homology. We study spectrum level lifts of structural properties of quantum annular homology.
\end{abstract}
\tableofcontents

\section{Introduction}  
The construction of Khovanov homology in \cite{Khovanov} was the first in a family of link homology theories categorifying quantum invariants of links in $\R^3$. 
To a planar diagram $D$ of an oriented link $L$ in $\R^3$ it associates a chain complex $CKh(D)$ of graded modules. The homotopy class of $CKh(D)$ is an invariant of the link, and the graded Euler characteristic of $CKh(D)$ is the Jones polynomial of $L$. Applications of Khovanov homology have been widely studied, and its functoriality properties with respect to surface cobordisms in $4$-space are a particularly important aspect of the theory.

Using the framework of the Cohen-Jones-Segal construction \cite{CJS}, Lipshitz and Sarkar constructed in \cite{LS} a {\em stable homotopy refinement} of Khovanov homology. (An alternative construction was proposed by Hu-Kriz-Kriz in \cite{HKK}.) This theory assigns to a link $L$ in $\R^3$ a suspension spectrum $\X_{Kh}(L)$ whose cohomology is the Khovanov homology of $L$. The stable homotopy type carries additional information about the link, not seen at the level of Khovanov homology; specifically it induces an action of the Steenrod algebra. Another construction of $\X_{Kh}(L)$, using the Burnside category, was given by Lawson, Lipshitz and Sarkar in \cite{LLS}. Its analogue for the odd Khovanov homology was introduced in \cite{SSS}, and
extensions to tangles were developed in \cite{LLS17,LLS3}. 
Some approaches
have been proposed \cite{HKS, JLS} for defining a stable homotopy refinement of Khovanov-Rozansky $\sl_N$ homology for $N\geq 3$, but a general theory over $\Z$ for all links in $S^3$ is not presently known.

\subsection{Annular homology theories} \label{sec:Annular homology theories}
This paper concerns {\em annular links}, that is links in the thickened annulus $\A\times I$, where $\A=S^1\times [0,1]$. Given a link $L$ in $\A\times I$, consider its projection $D$ onto the first factor $\A$. Following constructions by Asaeda-Przytycki-Sikora \cite{APS}, Bar-Natan \cite{BN2} and Roberts \cite{Roberts}, the triply graded annular Khovanov homology $Kh_{\A}(L)$ (sometimes called {\em sutured annular Khovanov homology}) may be obtained from the usual Khovanov chain complex \cite{Khovanov} of $D$, viewed as a diagram in ${\mathbb R}^2$ by including $\A\subset {\mathbb R}^2$, and then taking the annular degree zero part of the differential. Alternatively, $Kh_{\A}(L)$ may be obtained by applying a certain TQFT to the Bar-Natan category ${\mathcal{BN}}(\A)$ of the annulus. It was shown by Grigsby-Licata-Wehrli \cite{GLW} that this homology carries an action of $\sl_2$. 

The {\em quantum annular homology} $Kh_{\A_\q}(L)$, introduced by Beliakova-Putyra-Wehrli in \cite{BPW}, is a far-reaching extension. 
Consider a ring $\k$ and a fixed unit $\q\in\k$.  Following the notation of \cite{BPW}, we note that there are two $q$'s in the theory.  One corresponds to the usual $q$-grading and the second one is the unit $\q\in\k$; we distinguish them by using different fonts. In a sense $\q$ may be thought of as a deformation parameter, explaining the term ``quantum homology''.

A rough outline of the construction of $Kh_{\A_\q}(L)$ is as follows (see Section \ref{sec:overview of quantum annular homology} for a more detailed description.)  Given an annular link diagram $D\subset \A=S^1\times I$, we cut it along a seam $\{*\}\times I$ to obtain an $(n,n)$-tangle $D^{\rm cut}$. A construction of Chen and Khovanov \cite{CK} then yields graded {\em platform algebras} $A^n$ and a functor 
$
\F^{}_{CK} : \BN(n,m) \to \text{gBimod}(A^n, A^m),
$
where $\BN(n,m)$ is the Bar-Natan category of the rectangle with marked points, and $\text{gBimod}(A^n, A^m)$ is the category of graded $(A^n, A^m)$-bimodules.  In \cite{BPW}, the authors introduce \textit{quantum Hochschild homology}, denoted $qHH$,  a deformation of the usual Hochschild homology of bimodules.  The link homology theory is then defined using the quantum annular TQFT $\F_{\A_{\q}}$:
\begin{equation} \label{qHH eq}
\F_{\A_{\q}}(D) : = qHH(A^n, \F^{}_{CK}(D^{\rm cut})).
\end{equation}
The functorial extension to surfaces in $\A\times I\times I$ relies on the  theory  of  (twisted)  horizontal traces of bicategories.
Quantum annular homology has a number of interesting properties \cite{BPW}:

\begin{enumerate}[label=(\roman*)]
\item \label{item:homology depends on q}
The homology $Kh_{\A_\q}(L)$ in general depends on the choice of $\q$; for example different roots of unity $\q\in \C$ may give non-isomorphic theories.

\item \label{item:U_action}
$Kh_{\A_\q}(L)$ carries an action of $\U$.

\item \label{item:cob maps}
Let $\Sigma$ be a closed surface in $S^1\times {\mathbb R}^3$. Denoting by $L$ the link ${\Sigma}\cap ( * \times {\mathbb R}^3)$, $\Sigma^{\rm cut}$ gives a cobordism from $L$ to itself in $I\times {\mathbb R}^3$. 
The evaluation of $Kh_{\A_\q}(\Sigma)$
equals the graded Lefschetz trace of ${\Sigma}_* : Kh(L)\longrightarrow Kh(L)$, the endomorphism of the Khovanov homology of $L$ induced by $\Sigma^{\rm cut}$. 
In particular, $Kh_{\A_\q}(S^1 \times L)$
coincides with the Jones polynomial of $L$.

\end{enumerate}

\subsection{Spectra for annular links} \label{sec:Spectra for annular links}
A stable homotopy refinement ${\mathcal X}_{\A}(L)$ of the annular Khovanov homology may be defined along the lines of \cite{LS, LLS}. A different approach was used by Lawson-Lipshitz-Sarkar in \cite{LLS3}: they constructed a stable homotopy refinement of Chen-Khovanov algebras, giving rise to an alternative construction of ${\mathcal X}_{\A}(L)$ as the topological Hochschild homology of the resulting ring spectrum.

The main result of this paper is a
construction of a stable homotopy refinement of quantum annular homology.
We work over the Laurent polynomial ring $\k:=\Z[\q,\q^{-1}]$, and tensor the resulting theory with $\k_r:=\Z[\q,\q^{-1}]/(\q^r-1)$, where $r\geq 2$.
\begin{theorem} \label{equivariant thm} {\textsl
Let $L$ be an oriented link in the thickened annulus $\A\times I$. Then for each $ r\geq 2$, there
exists a naive ${\mathbb Z}/r{\mathbb Z}$-equivariant spectrum ${\mathcal X}^r_{\A_\q}(L)$ whose cohomology
is isomorphic to the quantum annular homology $Kh_{\A_\q}(L)$, as modules over $\Z[{\mathbb Z}/r{\mathbb Z}]$.
}
\end{theorem}
A key point in the construction ${\mathcal X}^r_{\A_\q}(L)$ is the interpretation of $\q$ as a generator of the cyclic group ${\mathbb Z}/r{\mathbb Z}$. The proof then proceeds by building in Section \ref{Quantum Annular Burnside section} the \emph{quantum annular Burnside functor}: a  strictly unitary lax $2$-functor to a suitably defined equivariant Burnside category, and then constructing in Section \ref{sec:from Burnside functors to stable homotopy types} an equivariant version of spatial refinement, building on the approaches of \cite{LLS, SSS}. 
The definition of $\X^r_{\A_\q}(D)$ is given for link diagrams $D$ in Definition \ref{def:quantum annular Khovanov spectrum} in Section \ref{sec:defining the homotopy type}; the proof of invariance with respect to all choices involved (including choice of diagram) is presented there via Theorems \ref{thm:X(D) well-defined} and \ref{thm:invariance of htpy type}.

Part of our construction involves a concrete description of generators and of the differential, starting from the quantum Hochschild homology definition \cite{BPW} of the quantum annular TQFT $\F_{\A_{\q}}(D)$ in (\ref{qHH eq}); this may be of independent interest to the reader interested in computational aspects of the theory. In fact, there is an important distinction between (annular) Khovanov homology and the quantum annular homology $Kh_{\A_\q}(L)$. In the setting of Khovanov homology, each resolution of the link diagram has a preferred collection of generators, and this is a crucial feature used in constructions of stable homotopy refinements in \cite{LS, LLS}. On the other hand, in the context of quantum annular homology, generators are well-defined only up to a multiple of a power of $\q$.
The proof of Theorem \ref{equivariant thm} involves a careful analysis of this indeterminacy and its relation to the group action on the spectrum.
Moreover, the differential does not admit an immediate calculation in terms of the combinatorics of a given curve configuration in the annulus; rather one has to work with the definition in terms of the quantum annular TQFT $\F_{\A_{\q}}$. A detailed analysis of the saddle maps defining the differential is given in Section \ref{sec:saddle maps}.
For $r>2$, the powers of $\q$ appearing in the differential affect the construction of the Burnside functor, similar to how the signs appearing in odd Khovanov homology affect the analysis in \cite{SSS}.

The proof of the following result is presented in Section \ref{sec:taking the quotient}.
\begin{theorem} \label{quotient theorem}
{\textsl
The
quotient of ${\mathcal X}^r_{\A_\q}(L)$ under the action of ${\mathbb Z}/r{\mathbb Z}$ recovers ${\mathcal X}_{\A}(L)$, the stable homotopy refinement of the classical annular Khovanov homology of $L$.
}
\end{theorem}

It is important to note that no group action is assumed to be present on the link $L\subset \A\times I$, so the context for our work is different from that in \cite{BPS, Musyt, SZ}.  Therefore ${\mathcal X}^r_{\A_\q}(L)$ may be thought of as an ``equivariant refinement'' of ${\mathcal X}_{\A}(L)$, a structure  that is not apparent in other constructions of the annular spectrum ${\mathcal X}_{\A}(L)$.

\subsection{Properties and questions}
It is an interesting question to what extent properties of the quantum annular homology theory $Kh_{\A_\q}(L)$ can be lifted to the level of spectra.
In Theorem \ref{thm:cobordism maps on spectra} we prove that a generically embedded cobordism $W\subset \A\times I\times [0,1]$ between two annular links $L_0$ and $L_1$ gives rise to a map
\[\varphi^r_W:\X^r_{\A_\q}(L_1)\rightarrow \X^r_{\A_\q}(L_0)\]
which induces the map  on quantum annular Khovanov homology over the ring $\k_r$, defined in \cite{BPW}. As usual, the construction proceeds by decomposing $W$ into elementary cobordisms, whose annular projections correspond to Reidemeister moves and Morse surgeries. However in the case of quantum annular homology, additional complexity arises from isotopies of the link diagram across the seam of the annulus. The map on quantum homology induced by cobordisms in $4$ dimensions in \cite{BPW} relies on the theory of horizontal traces and quantum Hochschild homology. To define maps on spectra, we need to introduce chain maps by specifying their values on chosen generators. A detailed discussion of these chain maps, as well as verification that they match the maps defined by the TQFT $\F_{\A_{\q}}$, are given in the proof of Theorem \ref{thm:cobordism maps on spectra} and in the Appendix.

We are now in a position to formulate a spectrum-level analogue of property \ref{item:cob maps} in Section \ref{sec:Spectra for annular links}.

\begin{theorem}
Let $L$ be a link in the 3-ball $B^3$, and  consider the surface $\h{W}=S^1\times L$ in $\A\times D^2\cong S^1\times B^3$. Let  $W$ denote a copy of $\h{W}$ perturbed to be generic, viewed as a cobordism from $\varnothing$ to itself.  
Then the map
\[\varphi_W^r:\X_{\A_\q}(\varnothing) \longrightarrow \X_{\A_\q}(\varnothing)\]
induces the map on quantum annular homology
$(\varphi_W^r)^*: \k_r\longrightarrow \k_r$
which is given by multiplication by the Jones polynomial of $L$, considered as an element of $\k_r$, up to a sign and an overall power of $\q$.
\end{theorem}

Here the spectrum $\X_{\A_\q}(\varnothing)$ associated to the empty set is the wedge sum of $r$ copies of the sphere spectrum, with cohomology isomorphic to $\k_r$. 
For brevity the theorem is stated for product surfaces $S^1\times L$; the graded Lefschetz trace statement for more general closed surfaces   holds as well.
See Corollary \ref{cor:cobordism map for sweep recovers Jones} and remarks following it for further details.

An important feature of stable homotopy refinement, not available on the level of link homology, is the action of Steenrod algebra. We do not address this aspect of the theory in the present paper; we plan to analyze the equivariant aspect of Steenrod operations on the spectra ${\mathcal X}^r_{\A_\q}(L)$ in a future work.

Recall property \ref{item:U_action} in  Section \ref{sec:Spectra for annular links}, stating that $Kh_{\A_\q}$ carries an action of $\U$; see \cite[Theorem B]{BPW} and Section \ref{sec:the K map} below  for a more detailed discussion.

\begin{conjecture} \label{Uq conjecture} {\textsl 
The action of $\U$ on quantum annular homology $Kh_{\A_\q}(L)$  can be lifted to an action on ${\mathcal X}^r_{\A_\q}(L)$.}
\end{conjecture}

In Section \ref{sec:the K map} we show that the invertible generator $K$ of $\U$ admits a lift to an equivariant automorphism of $\X^r_{\A_\q}(L)$, but lifting the other generators and the relations between them is outside the scope of this paper.  See the discussion at the end of Section \ref{sec:the K map} for more comments on this matter.

We conclude the introduction with another question.
As discussed above, recently Lawson-Lipshitz-Sarkar  gave a reformulation \cite{LLS3} of the annular Khovanov spectrum ${\mathcal X}_{\A}(L)$ as the topological Hochschild homology of their stable homotopy refinement of Chen-Khovanov algebras.
Our construction of the spectra ${\mathcal X}^r_{\A_\q}(L)$ is based on the definition of the quantum annular homology $Kh_{\A_\q}(L)$ in \cite{BPW} using quantum Hochschild homology of bimodules. It is an interesting question whether there is a formulation of ${\mathcal X}^r_{\A_\q}(L)$ using some twisted or equivariant version of topological Hochschild homology of the ring spectra associated in \cite{LLS3} to Chen-Khovanov algebras.

\textbf{Acknowledgements.}
We would like to thank Nick Kuhn, Krzysztof Putyra, Sucharit Sarkar and Matt Stoffregen for helpful conversations.

\section{The Quantum Annular TQFT}\label{sec:the quantum annular TQFT}

\subsection{Classical Annular Khovanov Homology}\label{sec:classical annular Khovanov homology}

This section reviews the construction of sutured annular Khovanov homology  \cite{APS, BN2, Roberts}. We will refer to it as {\em classical annular homology}, to distinguish it from the quantum version discussed in Section \ref{sec:overview of quantum annular homology}.  Let $I:= [0,1]$ denote the unit interval,  and we fix the notation  $\A$ for the annulus $S^1\times I$.
 An \textit{annular link} is a link in the thickened annulus $\A\times I$, and its diagram is a projection onto the first factor of $\A \times I$. Link diagrams are disjoint from the boundary of $\A$. Identifying $S^1\times (0,1)$ with $\R^2$ minus a point, we represent the annulus by simply indicating the deleted point using the symbol $\times$. Figure $\ref{fig:annular link}$ illustrates an example of a link diagram.

\begin{figure}[H]
\centering
\includegraphics{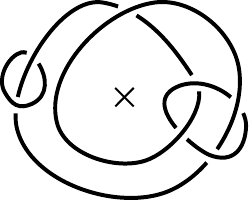}
\caption{An annular link diagram}\label{fig:annular link}
\end{figure}

Let $\BN(\A)$ denote the Bar-Natan category of the annulus \cite{BN2}. Its objects are formal $\Z$-linear combinations of formally graded collections of simple closed curves in $\A$.  Morphisms are matrices whose entries are formal $\Z$-linear combinations of dotted cobordisms embedded in $\A\times I$, modulo isotopy relative to the boundary, subject to the Bar-Natan relations, Figure \ref{fig:BN relations}.
\begin{figure}[H]
\centering
\includegraphics[scale=.8]{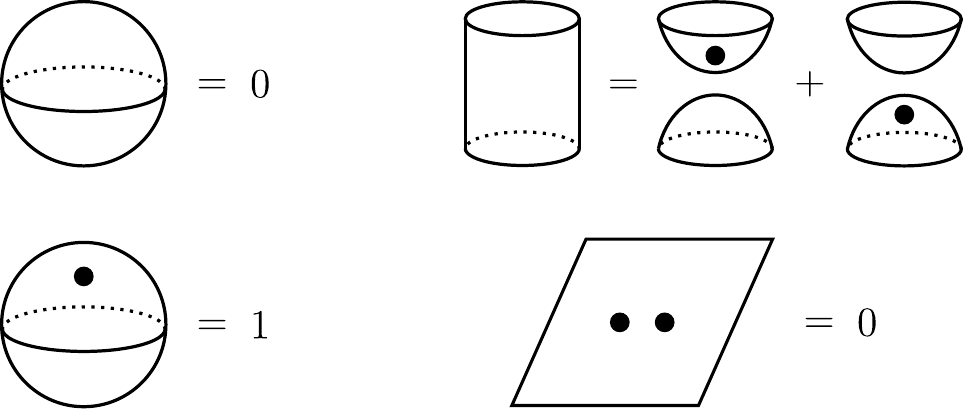}
\caption{Bar-Natan Relations}\label{fig:BN relations}
\end{figure}

Let $D$ be a diagram for an oriented annular link $L$.  We briefly review the construction of the chain complex $[[D]]$; a complete treatment can be found in \cite{BN2}.  To begin, one first forms the \textit{cube of resolutions} as follows.  Label the crossings of the diagram by $1,\ldots, n$. Every crossing may be resolved in two ways, called the \textit{0-smoothing} and \textit{1-smoothing}, as in \eqref{fig:0 and 1 smoothings}. For each $u = (u_1,\ldots, u_n)\in \{0,1\}^n$, perform the $u_i$-smoothing at the $i$-th crossing. The resulting diagram is a collection of disjoint simple closed curves in $\A$, which we denote $D_u$. Thinking of elements of $\{0,1\}^n$ as vertices of an $n$-dimensional cube, decorate the vertex $u$ by the smoothing $D_u$. 
\begin{equation}\label{fig:0 and 1 smoothings}
\vcenter{\hbox{
\includegraphics{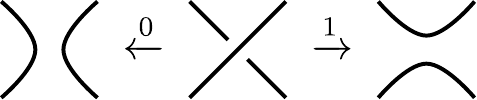}
}}
\end{equation}

Let $v=(v_1,\ldots, v_n)$ and $u=(u_1,\ldots, u_n)$ be vertices which differ only in the $i$-th entry, where $v_i = 0$ and $u_i=1$. Then the diagrams $D_v$ and $D_u$ are the same outside of a small disk around the $i$-th crossing. There is a cobordism from $D_v$ to $D_u$, which is the obvious saddle near the $i$-th crossing and the identity (product cobordism) elsewhere. We will call this the \textit{saddle cobordism} from $D_v$ to $D_u$, and denote it by $d_{v,u}$. Decorate each edge of the $n$-dimensional cube by these saddle cobordisms. We now have a commutative cube in the category $\BN(\A)$. There is a way to assign $s_{u,v} \in \{0,1\}$ to each edge so that multiplying the edge map $d_{v,u}$ by $(-1)^{s_{v,u}}$ results in an anti-commutative cube (see \cite[Section 2.7]{BN2}, also \cite[Definition 4.5]{LS}).

For $u=(u_1,\ldots, u_n) \in \{0,1\}^n$, let $\vert u \vert = \sum_i u_i$. Now, form the chain complex $[[D]]$ by setting 
\[
[[D]]^i = \bigoplus_{\vert u \vert = i+ n_-} D_{u}\{i + n_+ -n_-\} 
\] 
where $n_-$, $n_+$ are the number of negative and positive crossings in $D$, and the brackets $\{-\}$ denotes the formal grading shift in $\BN(\A)$. The differential is given on each summand by the edge map $(-1)^{s_{v,u}}d_{v,u}$. Anti-commutativity of the cube ensures that $[[D]]$ is a complex. 

\begin{theorem}\emph{(\cite[Theorem 1]{BN2})}
If diagrams $D$ and $D'$ are related by a Reidemeister move, then $[[D]]$ and $[[D']]$ are chain homotopy equivalent.
\end{theorem}

To obtain \textit{classical} \textit{annular Khovanov homology}, one applies the annular TQFT
\[
\F_\A : \BN(\A) \to \Mod(\Z)
\]
defined as follows. Let $V_\Z$ and $W_\Z$ be free rank two $\Z$-modules with bases $v_-, v_+$ and $w_-, w_+$ respectively. Equip each with two gradings, the quantum grading $\qdeg$ and the annular grading $\adeg$, defined on generators by
\begin{align}
& \qdeg(v_\pm) = 0 && \qdeg(w_\pm) = \pm{1}  \label{qdeg:simul1} \\
& \adeg(v_\pm) = \pm{1} && \adeg(w_\pm) = 0 \label{adeg:simul2}
\end{align}
We follow the grading convention of \cite{BPW}; note that the quantum grading on $V$ is different than the  quantum grading appearing elsewhere in the literature; see Remark \ref{rem:getting annular from regular}. 

There are two types of simple closed curves in $\A$; \textit{essential} curves and \textit{trivial} curves which bound disks in $\A$. The functor $\F_\A$ assigns $V_\Z$ to each essential circle and $W_\Z$ to each trivial circle. Then for $\Cs\subset \A$ a collection of disjoint simple closed curves with $e$ essential and $t$ trivial circles, the free abelain group $\F_\A(\Cs) =$ $\otimes^e V_\Z \otimes^t W_\Z$ (where the tensor product is taken over $\Z$) has a standard basis consisting of a label of $v_-$ or $v_+$ on each essential circle, and $w_-$ or $w_+$ on each trivial one. Following the conventions in \cite{BPW}, a generator of $\Cs$ will be represented as a choice of counterclockwise or clockwise orientations on each essential circle, corresponding to $v_+$ and $v_-$ respectively, and either a dot or no dot on each trivial circle, corresponding to $w_-$ and $w_+$.  We will often switch between the diagrammatic and algebraic representations of generators, Figure \ref{fig:diagrammatic representation of generators}.

\begin{figure}
\centering
\includegraphics{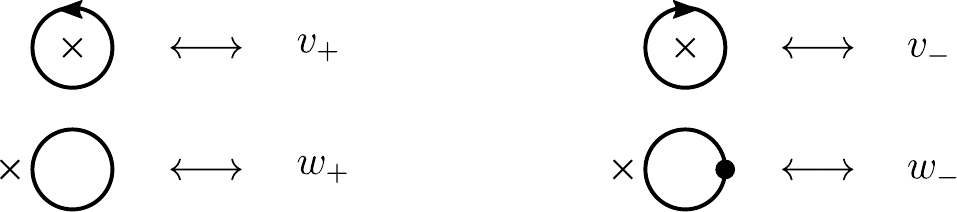}
\caption{Diagrammatic representation of generators.}\label{fig:diagrammatic representation of generators}
\end{figure}

To define $\F_\A$ on a cobordism, it is enough to consider cups, caps, and saddles. To a cup, $\F_\A$ assigns the unit $\varepsilon: \Z \to W_\Z$ defined by $\varepsilon(1) = w_+$. To a cap, $\F_\A$ assigns the counit $\eta: W_\Z \to \Z$ defined by 
\begin{align*}
& \eta(w_-) = 1 && \eta(w_+) = 0
\end{align*}
A saddle is assigned one of the six maps shown in Figure \ref{fig:annular formulas}, depending on whether it is a merge or a split and the types of curves involved. 

\begin{definition}\label{def:annular chain complex}
If $D$ is a diagram for an annular link $L$, define the \emph{annular Khovanov complex of $D$} to be 
\[CKh_{\A}(D):= \F_\A([[D]]);\]
it is an invariant of $L$ up to chain homotopy equivalence.
\end{definition}

\begin{figure}[H]
 \begin{center}
 \includegraphics[width=13cm]{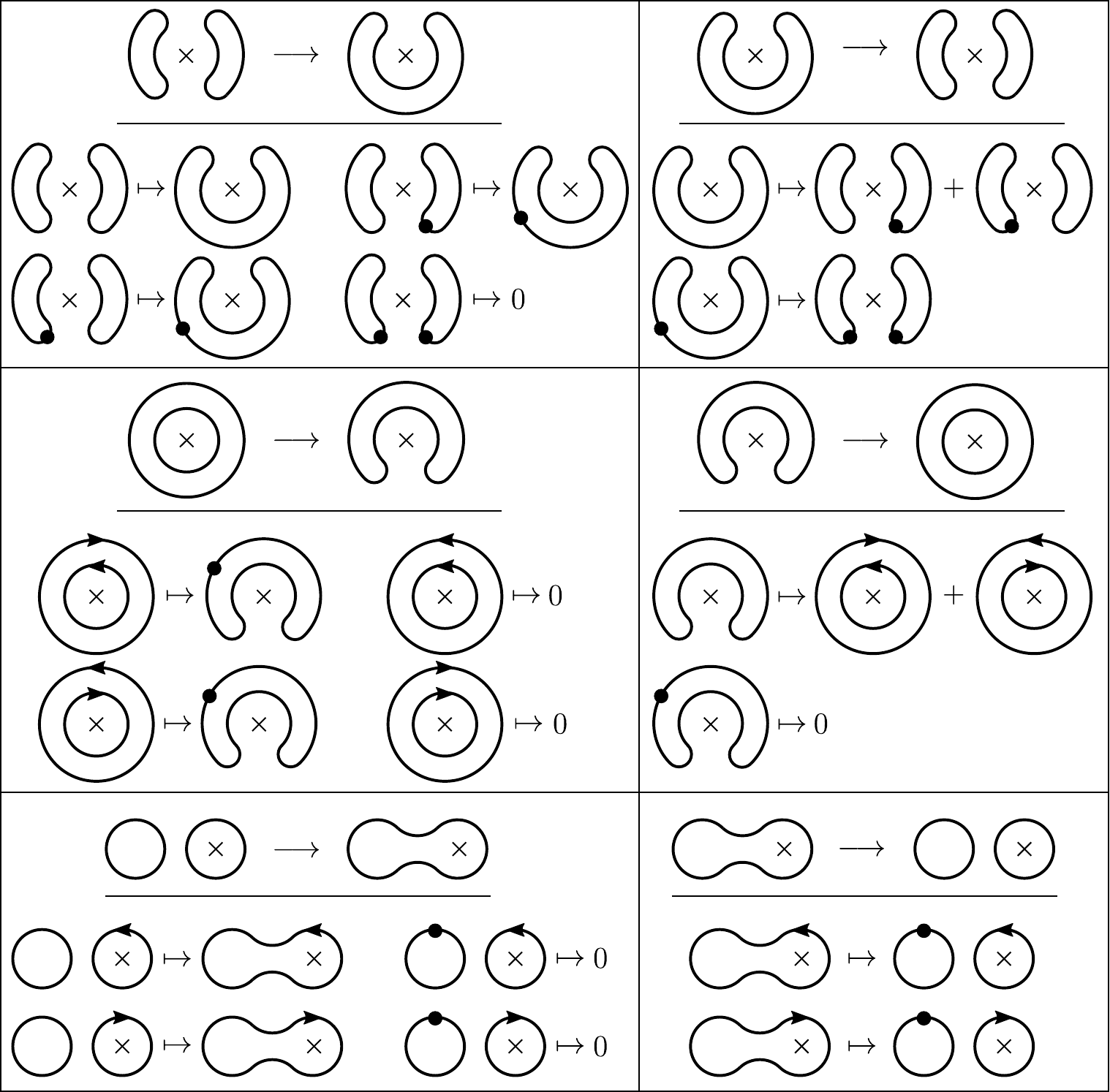}
 \caption{Surgery formulas in classical annular Khovanov homology. The table lists  topological types of surgeries and corresponding maps on generators.}\label{fig:annular formulas} 
 \end{center}
\end{figure}

\begin{remark}
Some of these formulas have an interpretation in terms of relations on cobordisms, as follows.  Let $\BBN(\A)$ \cite{BPW} denote the quotient of $\BN(\A)$ by Boerner's relation, which says that any cobordism carrying a dot and an essential curve is set to $0$ (see Figure \ref{fig:Boerner's relation}).
\begin{figure}
\centering
\includegraphics{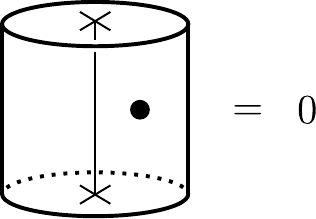}
\caption{Boerner's relation.}\label{fig:Boerner's relation}
\end{figure}

Algebraically, a dot on a cobordism corresponds to multiplication with $w_-$.  Then for a trivial circle $C$, the standard generator $w_+$ (resp. $w_-$) of $\F_\A(C) = W_\Z$ is the image of $1\in \Z$ under the undotted (resp. dotted) cup cobordism from $\varnothing$ to $C$.  The surgery formulas of Figure \ref{fig:annular formulas} then imply that $\F_\A$ factors through $\BBN(\A)$.  Algebraically, Boerner's relation can be seen as enforcing the equations $w_-\cdot v_- = w_- \cdot v_+ = 0$.
\end{remark}

\begin{remark}\label{rem:getting annular from regular}
To relate this to other constructions and grading conventions present in the literature (cf. \cite[Section 3]{Roberts}, \cite[Section 3.1]{GLW}), the annular chain complex may also be formed as follows. We may disregard the annular structure and view the resolutions $D_u$ in the cube $[[D]]$ as lying in the plane.  Then we may apply the usual Khovanov TQFT $\F_{Kh}$ to the cube, obtaining the Khovanov chain complex $CKh(D)$.  Every circle $C$ is assigned a free $\Z$-module $U$ generated by $u_+$ and $u_-$. The module $U$ carries an internal grading $\deg_U$, with $\deg_U(u_\pm) = \pm{1}$. Curves in the annulus carry an additional grading, $\adeg$, with $\adeg(u_{\pm})$ equal to $\pm{1}$ if $C$ is essential, and $0$ if $C$ is trivial. The Khovanov differential 
$d^{Kh}$ splits as $d^{Kh} = d^\A + d'$, where $d^\A$ preserves the annular grading and is precisely the map in Figure \ref{fig:annular formulas}, while $d'$ lowers the annular grading.  Thus the annular grading induces a filtration on the Khovanov complex $CKh(D)$. Taking the annual degree zero part of the differential and defining $\qdeg$ to be the difference between the degree $\deg_U$ and $\adeg$ yields precisely the classical annular chain complex  $CKh_\A(D)$. The gradings $\deg_U$, $\qdeg$, and $\adeg$ are denoted $j$, $j'$, and $k$ in \cite{GLW}, respectively.
\end{remark}

\subsection{Overview of Quantum Annular Homology}\label{sec:overview of quantum annular homology}
This section outlines the construction of the Beliakova-Putyra-Wehrli quantum annular link homology \cite{BPW}. The theory is built over a commutative ring $\k$ and a unit $\q\in \k$. We set $\k:=\Zq$, and the distinguished unit $\q\in \k$ is the same $\q$ appearing in $\Zq$. The main object is the quantum annular TQFT 
\[
\F_{\A_{\q}} : \BN_{\q}(\A) \to \Mod(\k)
\]
where $\Mod(\k)$ is the category of graded $\k$-modules and $\BN_{\q}(\A)$ is a certain deformation of the Bar-Natan category of the annulus. We will give a brief overview of the functor $\F_{\A_\q}$ and state a main theorem \cite[Theorem 6.3]{BPW}. 

\begin{remark} As mentioned above, we work over the Laurent polynomial ring $\k$ throughout this section. We will tensor the resulting theory with $\k_r:= \k/(\q^r-1)$ to construct the quantum annular Burnside functor in Section  
\ref{Quantum Annular Burnside section}.
\end{remark}

Let $\BN(n,m)$ denote the Bar-Natan category of the rectangle with $n$ points on the bottom and $m$ on top. Its objects are formal direct sums of formally graded  planar tangles in $I^2$ with $n$ endpoints on $I \times \{0\}$ and $m$ endpoints on $I \times \{1\}$. Such a tangle will be called a planar $(n,m)$-tangle. Morphisms in $\BN(n,m)$ are matrices whose entries are formal $\k$-linear combinations of embedded dotted cobordisms in $I^3$ between planar $(n,m)$-tangles, subject to the Bar-Natan relations (see Figure \ref{fig:BN relations}). 

\begin{figure}[H]
\centering
\includegraphics{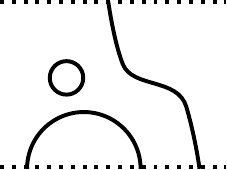}
\caption{A planar (3,1)-tangle}
\end{figure}

A \textit{seam} of $\A=S^1\times I$, denoted $\mu$, is an interval $\{*\}\times I$. In our representation of the interior of the annulus as $\R^2\smallsetminus\times$, we will fix the seam as the positive $x$-axis, ending on the left in $\times$.  See Figure \ref{fig:cutting a configuration} for an example.

 The quantum Bar-Natan category of the annulus, denoted $\BN_{\q}(\A)$, is a deformation of $\BN(\A$). 
 The objects of $\BN_{\q}(\A)$ are nearly the same as those of $\BN(\A)$, with the slight modification that curves in $\A$ must be transverse to $\mu$.
 Morphisms in $\BN_{\q}(\A)$ are also similar to those in $\BN(\A)$. In $\BN_{\q}(\A)$, isotopic cobordisms are identified if the isotopy fixes the \textit{membrane} $\mu\times I \subset \A\times I$. Otherwise, the cobordisms are scaled according to the degree of the part of the cobordism that passes through the membrane during the isotopy, accounting also for the coorientation of the membrane induced by the standard orientation of the core circle of $\A$. These will be referred to as \textit{trace moves}. The relations are depicted below in Figure $\ref{fig:trace moves}$; for details see \cite[Section 6.2]{BPW}.
\begin{figure}[H]
\centering
\includegraphics[height=4.5cm]{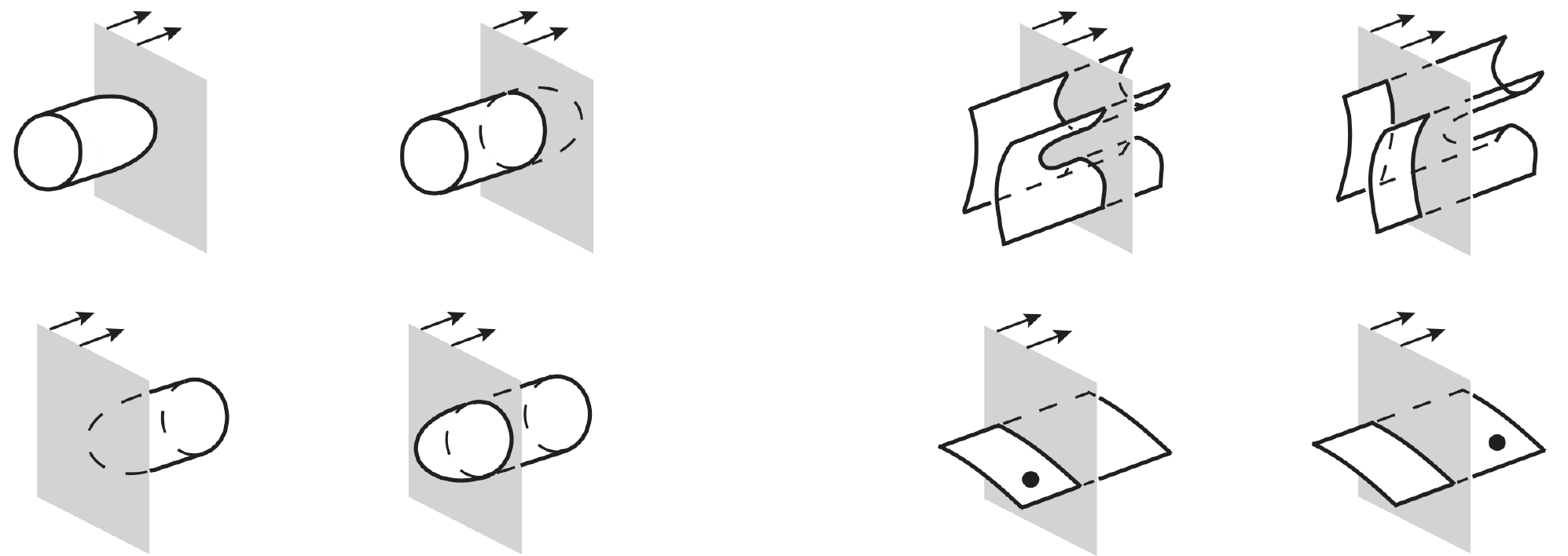}
    \put(-304,92){$=\,\q^{-1}$}
    \put(-299,30){$=\,\q$}
    \put(-83,92){$=\,\q$}
    \put(-86,30){$=\,\q^2$}
\caption{Relations in $\BN_{\q}(\A)$}\label{fig:trace moves}
\end{figure}

\noindent
The Bar-Natan relations (Figure \ref{fig:BN relations}) are imposed, where the local pictures are understood to be disjoint from the membrane. 

By general position if two annular cobordisms are isotopic, then they are a related by a sequence of trace moves and isotopies fixing the membrane. Therefore, if two cobordisms $S, S'\subset \A\times I$ are isotopic, then $S = \q^k S'$ as morphisms in $\BN_{\q}(\A)$, for some $k\in \Z$. (See also \cite[Proposition 6.2]{BPW}.)

 A \textit{configuration} $\Cs$ is a collection of disjoint simple closed curves in $\A$ which are transverse to $\mu$. Note that an object of $\BN_{\q}(\A)$ is a formal direct sum of formally graded configurations. Given a configuration $\Cs$ which intersects $\mu$ in $n$ points, we can cut along $\mu$ to obtain a planar $(n,n)$-tangle $\Cs^{\rm cut}$. See Figure \ref{fig:cutting a configuration} for an example. 
 \begin{figure}[H]
\centering
\includegraphics{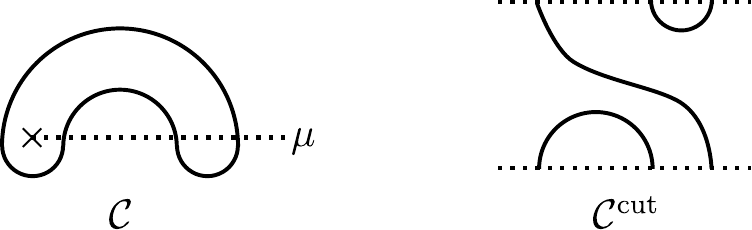}
\caption{Cutting open a configuration $\Cs$ along $\mu$ to obtain a (3,3)-tangle $\Cs^{\rm cut}$}\label{fig:cutting a configuration}
\end{figure}
A construction of Chen-Khovanov \cite{CK} yields graded $\k$-algebras $A^k$ for each $k\geq 0$, and a functor 
\[
\F_{CK} : \BN(n,m) \to \text{gBimod}(A^n, A^m)
\]
where $\text{gBimod}(A^n, A^m)$ is the category of graded $(A^n, A^m)$-bimodules. Let $\I^n$ denote the planar tangle consisting of $n$ vertical strands. Then, by definition of $\F_{CK}$, we have $\F_{CK}(\I^n) = A^n$. 

The \textit{quantum Hochschild homology}, denoted $qHH$ and defined in \cite[Section 3.8.5]{BPW}, is a deformation of the usual Hochschild homology of bimodules. It takes as input a graded $\k$-algebra $B$ and a graded $(B,B)$-bimodule $M$. The output $qHH(B,M) = \bigoplus_{i\geq 0} qHH_i(B,M)$ is a $\k$-module. Due to \cite[Proposition 6.6]{BPW} (stating that $qHH_i(\F_{CK}(\Cs^{\rm cut})) = 0$ for $i>0$),
we will mostly be interested in $qHH_0$. It follows immediately from the definition of $qHH$ that 
\begin{equation}\label{eq:qHH_0}
qHH_0(B,M) = M/ \text{span}_{\k} \{bm - \q^{\vert b \vert } mb \mid b\in B, m\in M\} 
\end{equation}
where $\vert b \vert$ denotes the degree of $b$.

 We are now ready to define $\F_{\A_{\q}}$ on objects. Let $\Cs$ be a configuration which intersects $\mu$ in $n$ points. Using the Chen-Khovanov functor, form the $(A^n, A^n$)-bimodule $\F_{CK}(\Cs^{\rm cut})$.
 The quantum annular TQFT $\F_{\A_{\q}}$ is then defined on objects by
\[
\F_{\A_{\q}}(\Cs) : = qHH(A^n, \F_{CK}(\Cs^{\rm cut})).
\]

By \cite[Proposition 6.6]{BPW}, we have $qHH_i(\F_{CK}(\Cs^{\rm cut})) = 0$
for $i>0$. Suppose $\Cs$ consists of $n$ essential curves each intersecting the seam once. Then $\Cs^{\rm cut} = \I^n$, so 
\[
\F_{\A_{\q}}(\Cs) = qHH_0(A^n, A^n).
\]
Let $A^n_0 \subset A^n$ denote the subalgebra consisting of elements of degree $0$.  By \cite[Proposition 6.6]{BPW}, the inclusion $A^n_0 \hookrightarrow A^n$ induces an isomorphism $qHH_0(A^n, A^n_0) \cong qHH_0(A^n, A^n)$. Moreover, $A^n_0$ is freely generated over $\k$ by $2^n$ elements $x_1,\ldots, x_{2^n}$, which are the \textit{primitive idempotents} of \cite[Section 5.5]{BPW}. They are in bijection with the \textit{cup diagrams} and satisfy $x_ix_j =\delta_{ij} x_i$. It follows from \eqref{eq:qHH_0} that 
\[
qHH_0(A^n, A^n_0) \cong \k^{2^n}.
\]
Every configuration $\Cs$ is isomorphic in $\BN_{\q}(\A)$ to a configuration $\Cs^\c$ in which every curve intersects the seam at most once. 
 If $\Cs$ has $e$ essential and $t$ trivial circles, then by delooping, one obtains
\[
\F_{\A_{\q}}(\Cs) \cong \F_{\A_{\q}}(\Cs^\c) \cong \k^{2^{e+t}}.
\]

We have so far only explained what $\F_{\A_{\q}}$ does on objects. The full construction of $\F_{\A_{\q}}$ in \cite{BPW} follows from a more general theory of (twisted) horizontal traces of bicategories, which we will not describe. The definition of $\F_{\A_{\q}}$ on morphisms follows from this general theory. The rest of this subsection describes the set-up for \cite[Theorem 6.3]{BPW}, which is stated as our Theorem \ref{thm:BPW}, and which is the main computational tool. 

Let $\BBN_{\q}(\A)$ denote the quotient of $\BN_{\q}(\A)$ by Boerner's relation, see Figure \ref{fig:Boerner's relation}. The functor $\F_{\A_{\q}}$ factors through $\BBN_{\q}(\A)$. Given a diagram $D$ for an annular link $L$ such that $D$ is transverse to $\mu$ and the crossings are disjoint from $\mu$, we form the cube of resolutions $[[D]]$ in the usual manner and view the result as a chain complex over the quantized category $\BBN_{\q}(\A)$.
\begin{definition}\label{def:quantum annular chain complex}
If $D$ is a diagram for an annular link $L$ which is transverse to the seam, define the \emph{quantum annular Khovanov complex of $D$} to be 
\[CKh_{\A_\q}(D):= \F_{\A_\q}([[D]]).\]
\end{definition}
The chain complex $CKh_{\A_\q}(D)$ is an invariant of $L$ up to chain homotopy equivalence by \cite[Proposition 6.8]{BPW}.

Let $TL$ denote the additive closure of the formally graded Temperley-Lieb category (\cite[Appendix A.1]{BPW}). Its objects are formal direct sums of formally graded finite collections of points on a line, and morphisms are $\k$-linear combinations of planar tangles between the points, modulo planar isotopy and the local relation that a circle is set to $\q+\q^{-1}$. Composition is given by stacking planar tangles, see Figure \ref{fig:TL example} for an example.  There is a functor $S^1\times (-) : TL \to \BN_{\q}(\A)$, which sends a collection of $n$ points to $n$ essential circles in $\A$, each intersecting $\mu$ once, and sends a planar tangle $T$ to the cobordism $S^1\times T$.  The relations in $\BN_{\q}(\A)$ (see Figure \ref{fig:trace moves}) imply that a torus wrapping once around the annulus evaluates to $\q+\q^{-1}$, so that $S^1 \times (-)$ is well-defined.

\begin{figure}[H]
\centering
\includegraphics{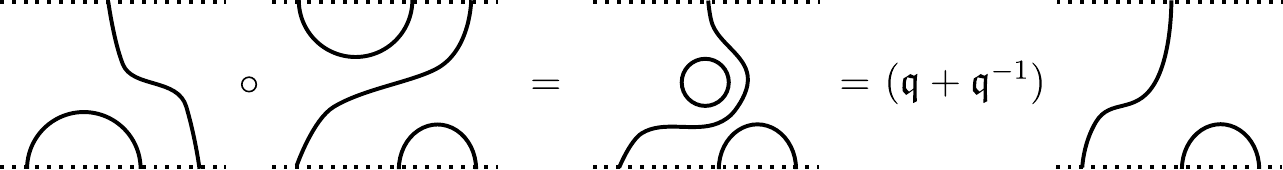}
\caption{Composition and relations in $TL$}\label{fig:TL example}
\end{figure}

Let $\text{gRep}(\U)$ denote the category of graded representations of $\U$. We follow the conventions established in \cite[Appendix A.1]{BPW} concerning $\U$; also see Section \ref{sec:the K map} of this paper. There is another functor $\F_{TL} : TL \to \text{gRep}(\U)$, defined as follows. Let $V_1 = \langle v_{-1}, v_1 \rangle $ be the fundamental representation of $\U$ and $V_1^* = \langle v_{-1}^* , v_1^* \rangle$ its dual. Let $V$ be free over $\k$ with basis $\{v_+, v_-\}$. Consider two $\k$-linear isomorphisms $\alpha: V_1 \to V$ and $\beta: V_1^*\to V$ defined by 
\begin{align*}
& \alpha : v_1 \mapsto v_+  && \beta: v_1^* \mapsto v_- \\
& \alpha: v_{-1} \mapsto v_- && \beta: v_{-1}^* \mapsto q^{-1}v_+
\end{align*}
These equip $V$ with two actions of $\U$, which are detailed in \cite[Appendix A.1]{BPW}. Note that $\beta^{-1}\circ \alpha : V_1 \to V_1^*$ is not $\U$-linear, even though $V_1$ and $V_1^*$ are isomorphic as $\U$-modules. 

The functor $\F_{TL}$ assigns $V^{\otimes n}$ to a collection of $n$ points. Since $\beta^{-1}\alpha$ is not $\U$-linear, there is an ambiguity in specifying the $\U$-module structure on $V^{\otimes n}$. The convention is that the $m$-th point is assigned $V_1$ if $m$ is odd, and $V_1^*$ if $m$ is even, so that the module assigned to $n$ points is given the $\U$-action according to the identification
\[
V^{\otimes n} \cong V_1 \otimes V_1^* \otimes V_1\otimes \cdots 
\] 

To define the value of $\F_{TL}$ on any planar tangle, it suffices to specify its value on caps and cups.  For a cap $\cap$, $\F_{TL}$ assigns the evaluation map $ev : V\otimes V  \to \k$, defined by
\begin{align*}
& v_+ \otimes v_+  \mapsto 0 && v_+ \otimes v_-  \mapsto q \\
& v_- \otimes v_-  \mapsto 0 && v_- \otimes v_+ \mapsto 1 
\end{align*}
On a cup $\cup$, $\F_{TL}$ assigns the coevaluation $coev: \k \to V\otimes V$, defined by
\[
1 \mapsto v_+ \otimes v_- + q^{-1} v_- \otimes v_+
\]
The evaluation map is always identified with either $V_1 \otimes V_1^* \to \k$ or $V_1^*\otimes V_1 \to \k$, and the coevaluation is identified with either $\k \to V_1 \otimes V_1^*$ or $\k \to V_1^* \otimes V_1$. With these identifications, the cap and cup are assigned $\U$-linear maps by $\F_{TL}$.

We have now explained the functors $\F_{TL} : TL\to \text{gRep}(\U)$ and $S^1\times (-): TL\to \BN_\q(\A)$. To compare $\F_{TL}$ with the composition $\F_{\A_\q}\circ S^1\times (-)$ in the statement of Theorem \ref{thm:BPW}, the value of $\F_{\A_\q}$ on $n$ essential circles intersecting the seam once needs to be given a $\U$-module structure. Recall that the Chen-Khovanov functor assigns the $\k$-algebra $A^n$ to the planar tangle $\I^n$ consisting of $n$ vertical strands. There is distinguished $\k$-linear isomorphism 
\begin{equation}\label{eq:value of QATQFT on essential circles}
qHH_0(A^n, A^n) \cong V^{\otimes n} 
\end{equation}
which we now describe. Recall that the inclusion $A^n_0 \hookrightarrow A_n$ induces an isomorphism on $qHH_0$, and that $qHH_0(A^n, A^n_0)$ has a distinguished $\k$-basis $\{x_1,\ldots, x_{2^n}\}$ corresponding to cup diagrams. Chen-Khovanov in \cite[Section 6]{CK} assign to each $x_i$ an element $p_i\in V^{\otimes n}$ such that the collection $\{p_i\}$ forms a basis of $V^{\otimes n}$. The isomorphism \eqref{eq:value of QATQFT on essential circles} is obtained by composing $qHH_0(A^n, A^n) \cong qHH_0(A^n, A^n_0)$ with the assignment $x_i\mapsto p_i$. 

\begin{theorem}\emph{(\cite[Theorem 6.3]{BPW})}\label{thm:BPW} There is a commuting diagram
\[
\begin{tikzcd}[column sep=tiny]
TL \arrow[rr, "S^1\times (-)"] \arrow[rd, "\F_{TL}"'] & & \BBN_{\q}(\A) \arrow[ld, "\F_{\A_{\q}}"] \\
 & \emph{gRep}(\U) & 
\end{tikzcd}
\]
with the horizontal functor an equivalence of categories. 
\end{theorem}
This theorem will play an important role in determining the values of the differential on generators in the next section.

\subsection{Fixing Generators}\label{sec:fixing generators}

In the construction of  Khovanov homotopy types in \cite{LS}, \cite{SSS} it is important to have a fixed set of generators for each configuration. Due to the definition of $\F_{\A_\q}$, the situation is more complicated in quantum annular homology. In this subsection, we explain how to fix generators for a general configuration. 

We will say a configuration $\Cs$ is \textit{standard} if every component intersects the seam in at most one point.   Every configuration $\Cs$ is isotopic in $\A$ to a standard configuration, denoted $\Cs^\c$, which is unique up to planar isotopy of the cut-open planar tangle. Next we explain in detail how Theorem \ref{thm:BPW} gives a canonical choice of generators for $\F_{\A_{\q}}(\Cs)$ when $\Cs$ is standard.

For a configuration $\Cs$, we will write $\Cs_E$ to denote the essential circles in $\Cs$, and $\Cs_T$ to denote the trivial circles. Figure \ref{fig:a configuration and its related things} illustrates these conventions. 
 
\begin{figure}[H]
\centering
    \begin{subfigure}[b]{.4\textwidth}
    \begin{center}
    \includegraphics{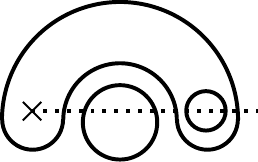}
    \caption{A configuration $\Cs$}
    \end{center}
    \end{subfigure}
    \begin{subfigure}[b]{.4\textwidth}
    \begin{center}
    \includegraphics{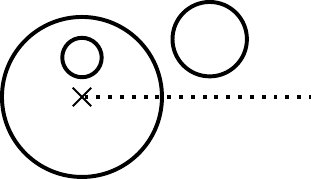}
    \caption{Its standard form $\Cs^\c$}
    \end{center}
    \end{subfigure}\\
    \begin{subfigure}[b]{.4\textwidth}
    \begin{center}
    \includegraphics{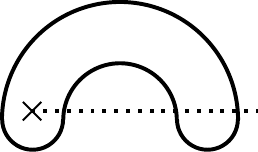}
    \caption{$\Cs_E$}
    \end{center}
    \end{subfigure}
    \begin{subfigure}[b]{.4\textwidth}
    \begin{center}
    \includegraphics{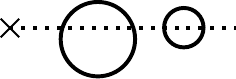}
    \caption{$\Cs_T$}
    \end{center}
    \end{subfigure}
\caption{}\label{fig:a configuration and its related things}
\end{figure}

For a cobordism $S\subset \A \times I$, let $\b{S}$ denote its reflection in the $I$ coordinate. For a configuration $\Cs$, we will often suppress the notation $\F_{\A_{\q}}(\Cs)$ when it is clear from context; that is, $x\in \Cs$ means $x\in \F_{\A_{\q}}(\Cs)$. Likewise, for a cobordism $S:\Cs\to \Cs'$, we will often write $S(x)$ to mean $\F_{\A_\q}(S)(x)$, where $\F_{\A_\q}(S): \F_{\A_{\q}}(\Cs) \to \F_{\A_{\q}}(\Cs')$ is the induced map. For cobordisms $\Cs \xrightarrow{S_1} \Cs'$ and $\Cs' \xrightarrow{S_2} \Cs''$, we will write $S_2S_1$ to denote their composition.

Suppose $C\subset \A$ is a trivial circle, and set $n=\vert C \cap \mu \vert$. 
The circle $C$ bounds an embedded disk $D\subset \A$, and we may push the interior of $D$ down in the $I$ coordinate to obtain a cobordism $\Sigma\subset \A\times I$ from the empty set to $C$ which intersects the membrane in exactly $n/2$ arcs. We refer to $\Sigma$ as the \textit{cup cobordism on $C$}. Similarly, we may pull $D$ up in the $I$ coordinate to obtain the \textit{cap cobordism on C}, which is simply $\b{\Sigma}$.

Let $W$ denote the $\k$-module assigned by $\F_{\A_{\q}}$ to a trivial circle $C$ which is disjoint from the seam. The module $W = \langle w_-, w_+\rangle$ is free of rank $2$. The standard generator $w_+$ (resp. $w_-$) is the image of $1\in \k$ under the undotted (resp. dotted) cup cobordism on $C$. Therefore, we will often identify $w_\pm$ with these cup cobordisms. Diagrammatically, we will signify that a trivial circle $C$ in $\Cs$ is labelled by $w_-$ by drawing a dot on $C$, as in Figure \ref{fig:diagrammatic representation of generators}. 

Suppose $\Cs$ is a standard configuration with $e$ essential circles and $t$ trivial circles. Order the trivial circles in some way, and order the essential circles from the innermost to the outermost. The exact ordering of the trivial circles is irrelevant, but it is important to order the essential circles in this way in light of Theorem $\ref{thm:BPW}$ and the asymmetry of the evaluation and coevaluation maps. We have that
\begin{equation} \label{trivial essential product eq}
\F_{\A_{\q}}(\Cs)  = V^{\otimes e} \otimes W^{\otimes t},
\end{equation}
where the tensor products above are understood to be over $\k$, and the identification of the value of $\F_{\A_\q}$ on $e$ standard essential circles with $V^{\otimes e}$ is the isomorphism from \eqref{eq:value of QATQFT on essential circles}. The modules $V$ and $W$ are each bigraded, carrying a quantum grading $\qdeg$ and an annular grading $\adeg$. The degrees of generators are as in $\eqref{qdeg:simul1}$ and $\eqref{adeg:simul2}$. Algebraically, we will write a standard generator as
\[
v_{a_1} \otimes \cdots \otimes v_{a_e} \otimes w_{b_1} \otimes \cdots \otimes w_{b_t}
\]
where each $a_i, b_j \in \{-, +\}$, the $v_{a_i}$ label the essential circles, and the $w_{b_j}$ label the trivial circles. We will often shorten the notation to $v_\Is \otimes w_\Js$, where $\Is$ is a sequence of $\pm$ labelling the essential circles and $\Js$ is a sequence of $\pm$ labelling the trivial ones. 
Note also that each standard generator $x = v_\Is \otimes w_\Js \in \Cs$ of a standard configuration $\Cs$ is the image of $v_\Is$ under the cobordism 
\[
\Sigma_\Js : \Cs_E \to \Cs
\]
which is the identity on $\Cs_E$ and a cup cobordism on each trivial circle, with some cups possibly carrying dots as specified by the labels $\Js$. Diagrammatically, we will use the same convention as in Figure \ref{fig:diagrammatic representation of generators}.

The following lemma concerns general (not necessarily standard) configurations; the argument is similar to the proof of \cite[Lemma 6.4]{BPW}.

\begin{lemma}\label{lem:inverse of an isotopy}
Let $\Cs, \Cs' \subset \A$ be two isotopic configurations. Let $\phi$ be an isotopy from $\Cs$ to $\Cs'$, and denote by $S:\Cs\to \Cs'$ the cylindrical cobordism in $\A\times I$ formed by $\phi$. Then $S$ is an isomorphism in $\BN_{\q}(\A)$, with $S^{-1} = \q^k \b{S}$ for some $k\in \Z$.
\end{lemma}

\begin{proof}

Isotopic cobordisms are equal in $\BN_{\q}(\A)$ if the isotopy between them fixes the membrane. We may therefore assume that the isotopy $\phi$ is a sequence of the local moves in Figure $\ref{fig:PN moves}$, denoted $P, P^{-1}, N,$ and $N^{-1}$. 

\begin{figure}[H]
\centering
\includegraphics{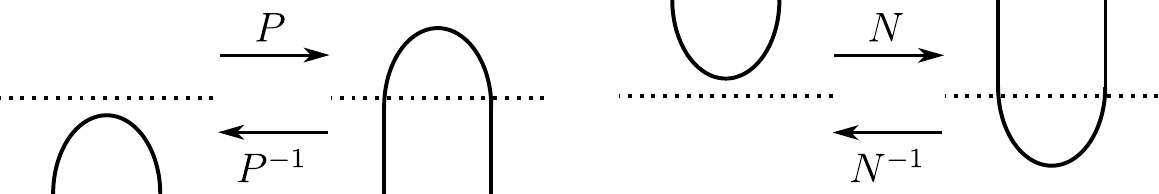}
\caption{}\label{fig:PN moves}
\end{figure}

Let  $p$ denote the number of moves of type $P$ or $P^{-1}$, let $n$ denote the number of moves of type $N$ or $N^{-1}$, and set $k=n-p$. It follows from the relations in Figure \ref{fig:trace moves} that 
\[
\q^k\b{S}S = \id_{\Cs},\ \q^k S\b{S} = \id_{\Cs'}.
\]
\end{proof}

\begin{lemma}\label{lem:generators under self-isotopy}
Let $\Cs$ be a standard configuration. Let $\phi$ be an isotopy from $\Cs$ to itself, with corresponding cobordism $S: \Cs \to \Cs$. For any standard generator $x\in \Cs$, $\q^k S(x) = x$ for some $k\in \Z$.
\end{lemma}

\begin{proof}
As discussed earlier, a standard generator $x= v_\Is \otimes w_\Js \in \Cs$ is the image of $v_\Is$ under a cobordism
\[
\Sigma_\Js : \Cs_E \to \Cs
\]
which is the identity on $\Cs_E$ and a cup cobordism on all circles in $\Cs_T$, with some cups possibly carrying dots as specified by $\Js$. Every component of the cobordism $S \Sigma_\Js$ is either an undotted annulus between essential circles or a possibly dotted disk with trivial boundary. Each disk can be isotoped to a cup cobordism on its trivial boundary circle at the cost of multiplying by a power of $\q$.  Then, at the cost of introducing further powers of $\q$, the remaining annuli may be isotoped to the identity cobordism on $\Cs_E$ while fixing each cup cobordism. 
This yields $\q^k S\Sigma_\Js = \Sigma_\Js$ for some $k\in \Z$.
\end{proof}

\begin{lemma}\label{lem:generators differ by power of q}
Let $\Cs$ be a configuration and let $\phi_1$, $\phi_2$ two isotopies from $\Cs^\c$ to $\Cs$. Denote the corresponding cobordisms by $S_1, S_2: \Cs^\c \to \Cs$. If $x\in \Cs^\c$ is a standard generator, then $S_1(x) = \q^k S_2(x)$ for some $k\in \Z$. 
\end{lemma}

\begin{proof}
Consider the cobordism $\overline{S_1}S_2: \Cs^\c\to \Cs^\c$, which is formed by the isotopy $\phi_1^{-1} \phi$. By Lemma \ref{lem:generators under self-isotopy}, we have $\q^m \overline{S_1}S_2 (x) = x$ for some $m\in \Z$. By Lemma \ref{lem:inverse of an isotopy}, we know $\q^\l \b{S_1} S_1 = \id_{\Cs}$ for some $\l\in \Z$. Then 
\[
 \q^\l \b{S_1} S_1 (x) = x = \q^m \b{S_1}S_2(x),
\]
and we obtain
\[
S_1(x) = \q^{m-\l} S_2(x).
\]
\end{proof}

\begin{remark} \label{non uniform} Given $S_1, S_2$ as in lemma \ref{lem:generators differ by power of q}, in general the power $k$ of $\q$ depends on the generator $x\in \Cs^\c$.
\end{remark}

So far the discussion concerned only generators of standard configurations. Next we consider generators for arbitrary configurations. 

\begin{definition} \label{generators of a configuration} {\rm
Fix a configuration $\Cs$ and an isotopy from $\Cs^\c$ to $\Cs$. Let $S$ denote the resulting cobordism $\Cs^\c \to \Cs$. The {\em generators of $\F_{\A_{\q}}(\Cs)$ corresponding to the cobordism $S$}, are the images of the standard generators of $\Cs^\c$ under $S$. We will also write generators of $\Cs$ as $v_\Is \otimes w_\Js$, which is to be understood as the image of the corresponding standard generator of $\Cs^\c$. Note that this image $v_\Is\otimes w_\Js$ depends on the choice of isotopy $\Cs^\c \to \Cs$, which we suppress from the notation.}
\end{definition}
 By Lemma \ref{lem:generators differ by power of q}, these generators of $\Cs$ are well-defined up to multiplication by a (non-uniform, according to Remark \ref{non uniform}) power of $\q$. We assume throughout that there is a fixed isotopy $\Cs^\c \to \Cs$, which will often not be named. Likewise, an unnamed cobordism $\Cs\to \Cs^\c$ denotes the inverse of $\Cs^\c \to \Cs$.

As discussed earlier, a standard generator $x\in \Cs^\c$ is the image of the corresponding standard generator of $\Cs^\c_E$ under a cobordism $\Sigma_\Js : \Cs^\c_E \to \Cs^\c$, where $\Sigma_\Js$ is the identity on $\Cs^\c_E$ and a cup cobordism on each circle in $\Cs^\c_T$, with some cups possibly carrying dots. Up to a power of $\q$, the cobordism $S\Sigma_\Js$ represents a cobordism which traces out an isotopy $\Cs^\c_E \to \Cs_E$ and is a cup cobordism on each circle in $\Cs_T$. Then each standard generator of $\Cs$ can also be realized as the image of a cup cobordism on each trivial circle and an isotopy on the essential circles. 

 Note also that we do not make any assumptions about how these isotopies are picked for different configurations within a single cube of resolutions.

\subsection{Computation of the Saddle Maps}\label{sec:saddle maps}

In this subsection we compute saddle maps in quantum annular homology using the relations in $\BBN_\q(\A)$ and Theorem \ref{thm:BPW}. These results will be used in the formulation of the quantum annular Burnside functor in Section \ref{Quantum Annular Burnside section}.

We start with several examples; the general case is treated in Proposition 
\ref{prop:LBM1}. Saddle maps for various types of configurations (where intersections with the seam are minimal) are summarized in Figure \ref{fig:quantum annular formula table}. In the first two examples the calculation relies on the Boerner relation and relations satisfied by cobordisms in the Bar-Natan category. Specifically, one uses the neck-cutting relation and delooping, cf. \cite[Proposition 5.3]{BPW}. (Note that delooping makes sense only for trivial, and not for essential circles in the annulus.)

To analyze our saddle maps, we will use the language of \textit{surgery arcs}, as in \cite[Section 2]{LS}. For a configuration $\Cs$, a surgery arc is an  interval embedded in $\A$ whose endpoints lie on $\Cs$ and whose interior is disjoint from $\Cs$. In the construction of quantum annular homology, link diagrams are assumed transverse to $\mu$, and all crossings are away from $\mu$. In light of this, we will assume that surgery arcs are disjoint from the seam. For a configuration $\Cs$ with a surgery arc, let $s(\Cs)$ denote the configuration obtained by surgery on the arc. There is a saddle cobordism $\Cs\to s(\Cs)$, which is well-defined in $\BN_{\q}(\A)$. In terms of the cube of resolutions of a link diagram, a surgery arc may be placed at a $0$-smoothing to indicate that there will be a saddle cobordism at that smoothing.  

\begin{example}\label{ex1}
For the saddle
\begin{center}
\includegraphics{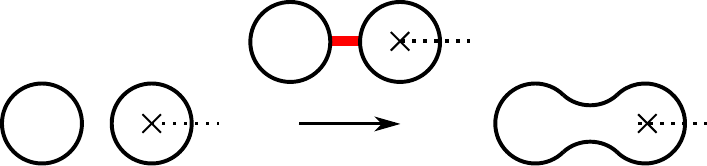}
\end{center}
between standard configurations, we have the following formulas

\begin{center}
\includegraphics{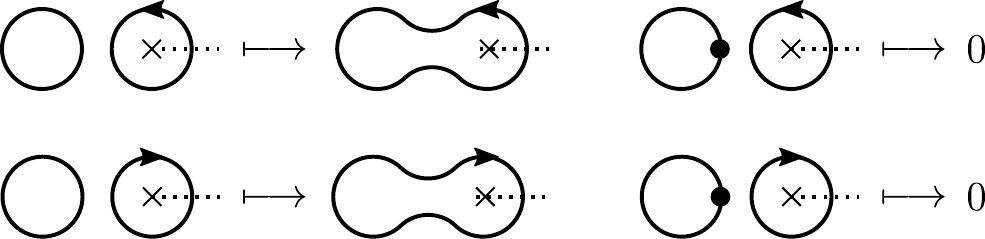}
\end{center}
Algebraically, this is written as 
\begin{align*}
& v_+ \otimes w_+ \mapsto v_+ && v_+ \otimes w_- \mapsto 0  \\
& v_- \otimes w_+ \mapsto v_- && v_- \otimes w_- \mapsto 0
\end{align*}
These can be deduced from Boerner's relation (Figure \ref{fig:Boerner's relation}) and the fact that the two standard generators of a trivial circle are picked out by an undotted cup and a once dotted cup. 
\end{example}

\begin{example}\label{ex2}
For the saddle 

\begin{center}
\includegraphics{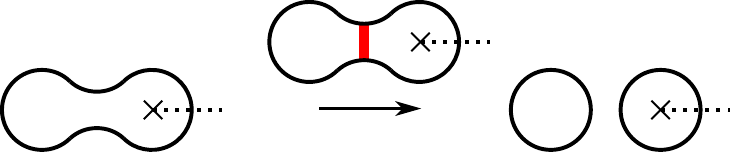}
\end{center}
we have the formulas

\begin{center}
\includegraphics{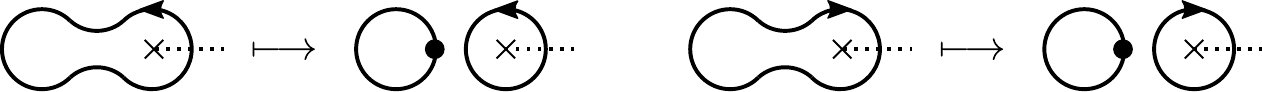}
\end{center}
which, algebraically, can be written as 
\begin{align*}
& v_+ \mapsto v_+ \otimes w_- && v_- \mapsto v_- \otimes w_-
\end{align*}
This can be deduced by cutting the neck along the trivial circle which splits off as a result of the saddle, and then applying Boerner's relation. 
\end{example}

The next two examples are also discussed in \cite[Section 6.4]{BPW}.
\begin{example}\label{ex3} 
Let $S$ denote the following saddle 
\begin{center}
\includegraphics{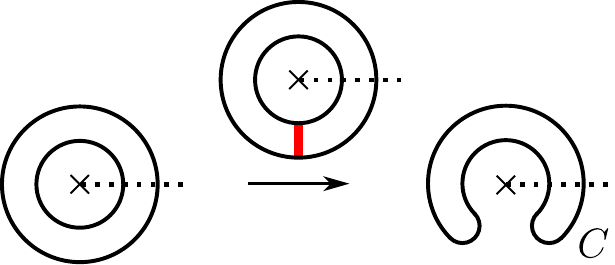}
\end{center}
Let $C$ denote the trivial circle on the right-hand side above. 
Since $C$ is not standard, we need to pick an isotopy to specify its generators. For the sake of calculation, we pick the following isotopy

\begin{center}
\includegraphics{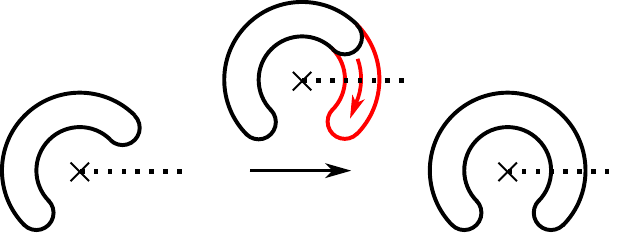}
\end{center}
which specifies generators, represented diagrammatically, as 
\begin{center}
\includegraphics{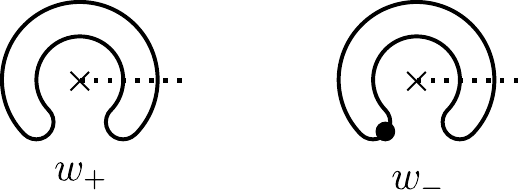}
\end{center}
These generators are the images of $1\in \k$ under the undotted and dotted cup cobordisms on $C$, respectively. Since the cup cobordism on $C$ intersects the membrane, the placement of the dot is relevant, and the diagram shows where the dot is placed. 

Now, let $\Sigma$ denote the undotted cap cobordism on $C$, and let $\Sigma'$ denote the dotted cap cobordism on $C$, with the dot placed as in the generator $w_-$. Using the relations in $\BN_{\q}(\A)$, we obtain
\begin{align*}
&\Sigma(w_+) = 0 && \Sigma(w_-) = \q^{-1} \\
& \Sigma'(w_+) = \q^{-1} && \Sigma'(w_-) = 0
\end{align*}
Composing with the saddle $S$, observe that $\Sigma S = S^1\times \cap$ in $\BBN_\q(\A)$, and that $\Sigma'S = 0$ by Boerner's relation. We are now in a position to write down formulas for $S$. For example, we may write 
\begin{equation}\label{eq1}
 S(v_+ \otimes v_-) = \alpha w_+ + \beta w_-
\end{equation}
for some $\alpha, \beta \in \k$. Applying $\Sigma$ to the above equality, we obtain 
\[
\Sigma S(v_+ \otimes v_-) = \q^{-1} \beta .
\]
Theorem $\ref{thm:BPW}$ tells us $\Sigma S(v_+ \otimes v_-) = ev(v_+ \otimes v_-) = \q$, so that $\beta = \q^2$. By applying $\Sigma'$ to both sides of $\ref{eq1}$, we obtain $\alpha = 0$. A similar argument for the remaining generators yields the full table of formulas for $S$:
\begin{figure}[H]
\centering
\includegraphics{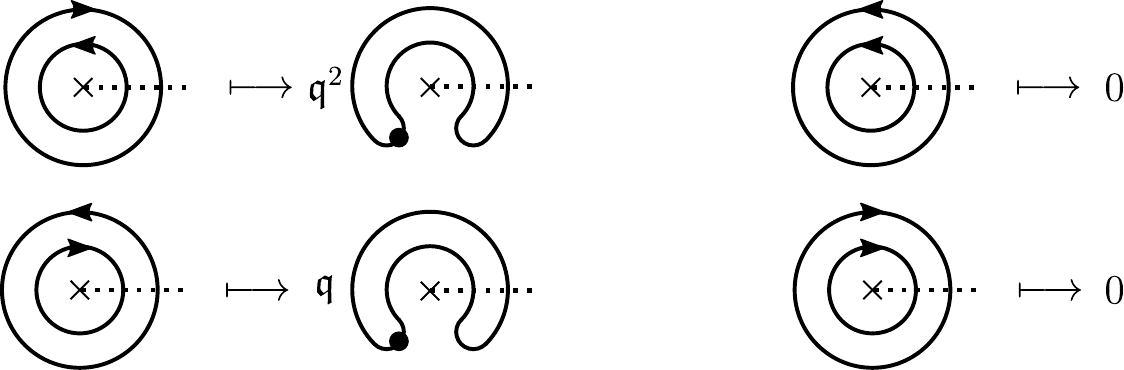}
\end{figure}
\noindent
Equivalently, 
\begin{align*}
& S(v_+\otimes v_-) = \q^2 w_- && S(v_+\otimes v_+) = 0 \\
& S(v_-\otimes v_+) = \q w_- && S(v_- \otimes v_-) = 0
\end{align*}
\end{example}

\begin{example}\label{ex4}
Let $S$ denote the following saddle. 
\begin{center}
\includegraphics{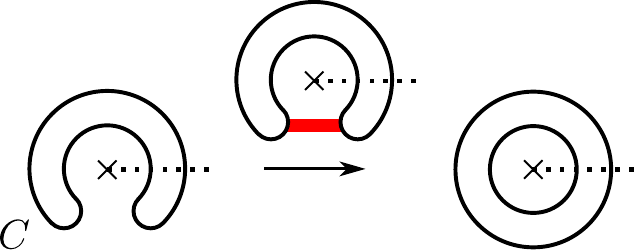}
\end{center}
Pick generators $w_+$ and $w_-$ for the left-hand trivial circle $C$ as in Example \ref{ex3}. 
Let $\Sigma$ and $\Sigma'$ be the undotted and dotted cups on $C$, so that $w_+ = \Sigma(1)$ and $w_- = \Sigma'(1)$. Observe that $S \Sigma' = 0$ by Boerner's relation, so 
\[
S(w_-) = 0
\]
Finally, note that $S \Sigma = S^1 \times \cup$. By Theorem $\ref{thm:BPW}$, we obtain 
\[
S(w_+) = v_+ \otimes v_- + \q^{-1} v_- \otimes v_+.
\]
Diagrammatically, the formulas for $S$ are 

\begin{figure}[H]
\centering
\includegraphics{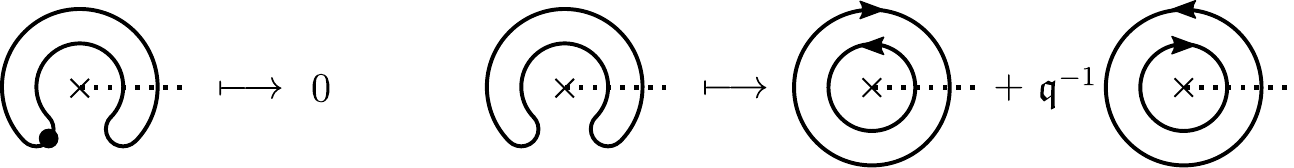}
\end{figure}
\end{example}

Saddle maps for various topological types of configurations, including the result of calculations in examples \ref{ex1} - \ref{ex4}, are summarized in Figure \ref{fig:quantum annular formula table}. The next example illustrates a calculation of the saddle map in a case of higher multiplicity of intersections between the configuration and the seam. 

\begin{figure}
\centering
\includegraphics[width=13cm]{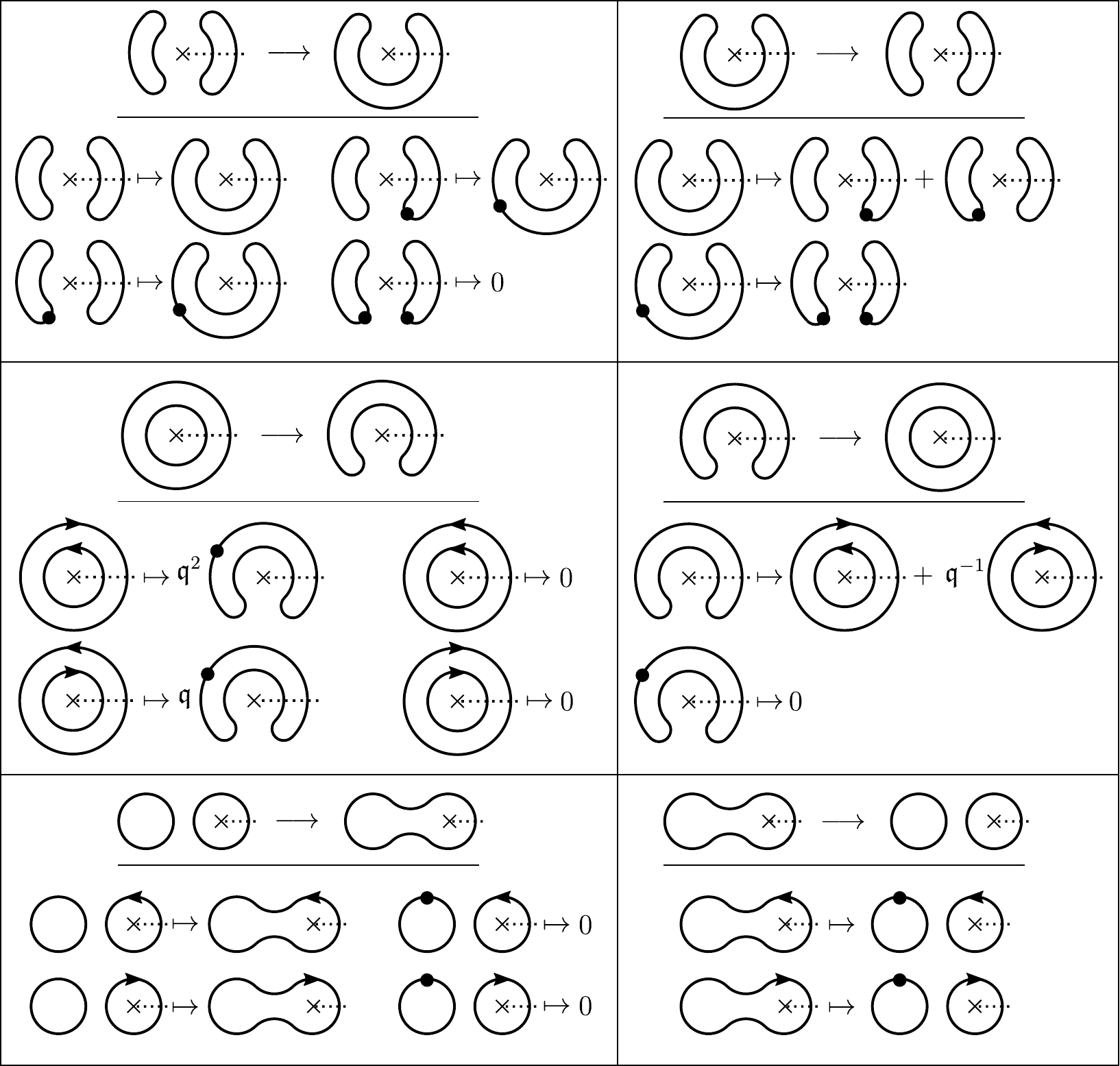}
\caption{Surgery formulas in quantum annular Khovanov homology for curves with minimal intersections with the seam.}\label{fig:quantum annular formula table}
\end{figure}

\begin{example}\label{ex5}
Here is a slightly more involved version of Example \ref{ex3}. Consider the configurations $\Cs_1$ and $\Cs_2$, and the saddle $S: \Cs_1 \to \Cs_2$ as shown in Figure \ref{fig:wacky configuration}. 

\begin{figure}[H]
\centering
\includegraphics[scale=.85]{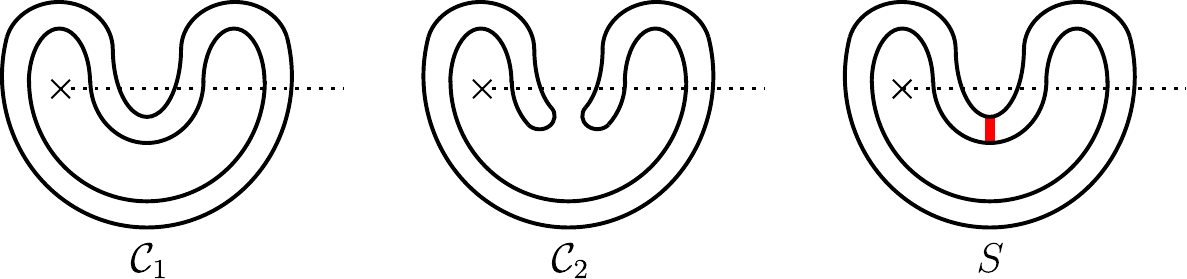}
\caption{}\label{fig:wacky configuration}
\end{figure}
\noindent
Fix generators for $\Cs_1$ and $\Cs_2$ using the isotopies $S_1$ and $S_2$ depicted in Figure \ref{fig:isotopy1 for ex5} and Figure \ref{fig:isotopy2 for ex5}. 

\begin{figure}[H]
\centering
\includegraphics[scale=.85]{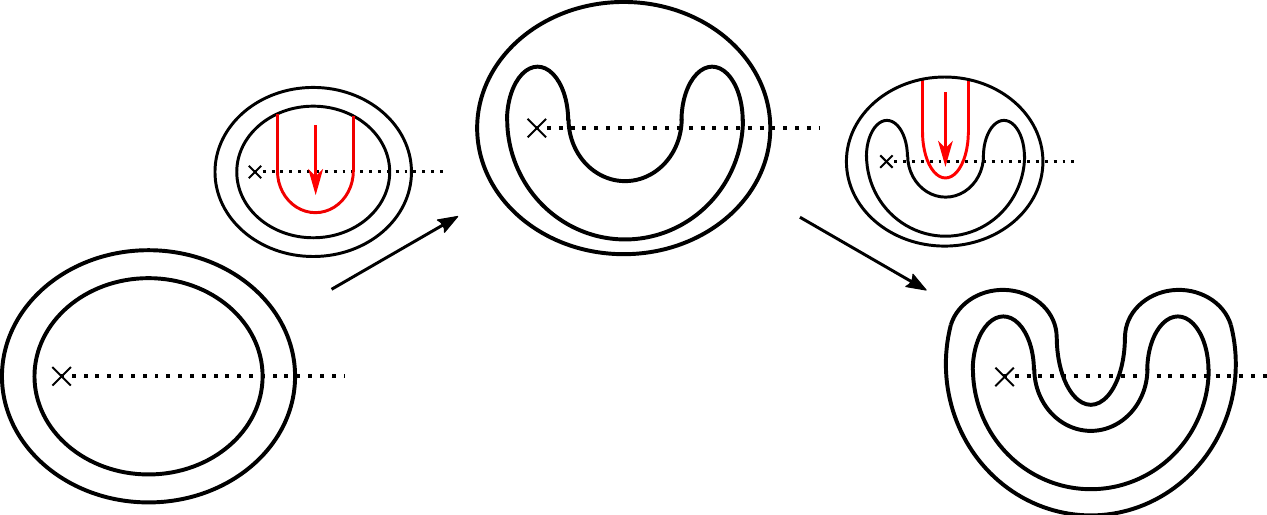}
\caption{An isotopy $S_1: \Cs_1^\c \to \Cs_1$}\label{fig:isotopy1 for ex5}
\end{figure}

\begin{figure}[H]
\centering
\includegraphics[scale=.9]{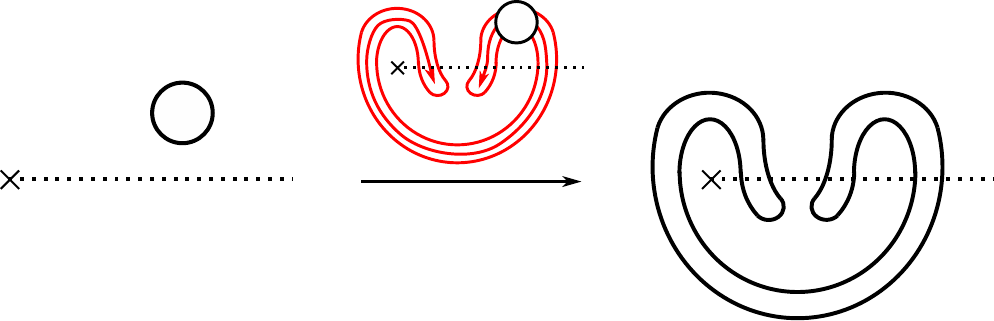}
\caption{An isotopy $S_2: \Cs_2^\c \to \Cs_2$}\label{fig:isotopy2 for ex5}
\end{figure}
Since generators of $\Cs_2$ are the images of the standard generators of $\Cs_2^\c$ under the isomorphism $S_2: \Cs_2^\c \to \Cs_2$, it suffices to write down formulas for the composition 
\begin{equation}\label{eq:composition in ex5}
\Cs^\c_1 \rar{S_1} \Cs_1 \rar{S} \Cs_2 \rar{S_2^{-1}} \Cs^\c_2
\end{equation}
Note that $S_2^{-1} = \q^3\b{S_2}$ (see Lemma \ref{lem:inverse of an isotopy}). Let $\Phi$ denote the composition \eqref{eq:composition in ex5}.

Let $\Sigma$ and $\Sigma'$ be the undotted and dotted cap cobordisms, respectively, on the trivial circle $\Cs_2^\c$.  Note that $\Sigma'\Phi = 0$ by Boerner's relation. Applying a trace move from Figure \ref{fig:trace moves} to the part of the cobordism depicted in \eqref{fig:cob through membrane}, we see that 
\[
\Sigma\b{S_2}SS_1
\]
is equal to $\q^{-1}(S^1\times \cap)$, so that 
\[
\Sigma\Phi = \q^2 (S^1\times \cap)
\]
Arguing as in Example \ref{ex3}, we obtain
\begin{align*}
& S(v_+\otimes v_-) =  \q^3w_- && S(v_+\otimes v_+) = 0 \\
& S(v_-\otimes v_+) = \q^2 w_- && S(v_- \otimes v_-) = 0
\end{align*}
\begin{equation}\label{fig:cob through membrane}
\vcenter{\hbox{
\includegraphics[width=3.5cm]{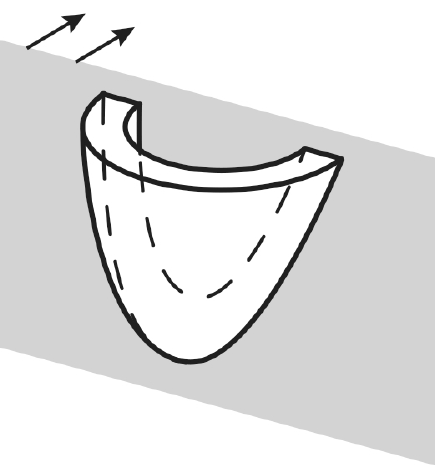}
}}
\end{equation}

\begin{remark}
Here is a slightly different way to finish the computation in Example \ref{ex5}, which will be used in Proposition $\ref{prop:LBM1}$. In the morphism $\Phi : \Cs_1^\c \to \Cs_2^\c$, we may cut the neck along a small push-off of the trivial circle $\Cs_2^\c$ to write $\Phi$ as a sum of two dotted cobordisms. One of the summands is $0$ by Boerner's relation, and the other is isotopic to a disjoint union of
\[
S^1 \times \cap
\]
and a dotted cup cobordism on $\Cs_2^\c$, which is denoted $\b{\Sigma'}$ using the notation of Example \ref{ex5}. Then, using the trace relations in $\BN_{\q}(\A)$, we see that 
\[
\Phi = \q^2 (S^1 \times \cap) \sqcup \b{\Sigma'},
\]
and the formulas in Example \ref{ex5} follow. 
\end{remark}

\end{example}

Example \ref{ex5} shows that there is considerable complexity in computing the saddle map when curves have multiple intersections with the seam.
The next proposition extends  Examples $\ref{ex3}$ - \ref{ex5} to the case of arbitrary configurations. It will be important for the analysis in Section \ref{sec:a ladybug configuration}. Recall the numbering of circles discussed in the paragraph preceding (\ref{trivial essential product eq}).

\begin{proposition}\label{prop:LBM1}
Let $\Cs$ be a configuration with a surgery arc $A$. Let $S: \Cs \to s(\Cs)$ denote the saddle. 
\begin{enumerate}[label= \emph{(\arabic*)}]
\item Suppose both endpoints of $A$ are on a trivial circle $C$, and that surgery along $A$ splits $C$ into two essential circles. Assume $C$ is first in the ordering on trivial circles of $\Cs$, and it splits into the $i$-th and $(i+1)$-th essential circles in $s(\Cs)$. Let $x= v_{\Is'} \otimes v_{\Is''} \otimes w_+ \otimes w_\Js$ be a generator of $\Cs$ in which $C$ is undotted, where $\Is'$ labels the first $i-1$ essential circles. Then 
\[
S(x) = \q^a v_{\Is'} \otimes v_+ \otimes v_-\otimes v_{\Is''}\otimes w_\Js+ \q^{a-1} v_{\Is'} \otimes v_- \otimes v_+\otimes v_{\Is''}\otimes w_\Js
\]
for some $a\in \Z$.

\item Suppose the endpoints of $A$ are on the $i$-th and $(i+1)$-th essential circles of $\Cs$. Consider the generators 
\begin{align*}
y_1 &= v_{\Is'} \otimes v_+ \otimes v_-\otimes v_{\Is''} \otimes w_\Js\\
y_2 &= v_{\Is'} \otimes v_- \otimes v_+ \otimes v_{\Is''} \otimes w_\Js
\end{align*}
of $\Cs$, where $\Is'$ labels the first $i-1$ essential circles. 
Then
\begin{align*}
S(y_1) &= \q^{b+1} v_{\Is'} \otimes v_{\Is''} \otimes w_- \otimes w_\Js \\
S(y_2) &= \q^b v_{\Is'} \otimes v_{\Is''} \otimes w_- \otimes w_\Js
\end{align*}
for some $b\in \Z$. 

\end{enumerate}

\end{proposition}

\begin{proof} 
For both $(1)$ and $(2)$, it is enough to show that the result holds after applying $s(\Cs) \to s(\Cs)^\c$. \par
(1). The generator $x$ is the image of $v_{\Is'}\otimes v_{\Is''}$ under the composition $\Cs_E^\c \xrightarrow{\Sigma_\Js} \Cs^\c \to \Cs$, where $\Sigma_\Js$ is a cup cobordism on trivial circles in $\Cs^\c$, with the cups dotted according to $\Js$. Note that the cup on $C$ is undotted. Let $\Psi$ denote the composition
\[
\Cs^\c_E \xrightarrow{\Sigma_\Js} \Cs^\c \to \Cs \xrightarrow{S} s(\Cs) \to s(\Cs)^\c
\]
The cobordism $\Psi$ isotopic to a disjoint union of 
\[
(*)\ S^1 \times \, 
\begin{tikzpicture}[baseline={([yshift=-.7ex]current bounding box.center)},x=.9em,y=1.2em]

\draw[thick] (0,0)--(0,1)
(2,0)--(2,1)
(3,1) -- (3,.5) to[out=-90,in=-90] (4,.5) -- (4,1)
(5,0)--(5,1)
(7,0)--(7,1);
\node at (1,.5){$\,\cdots$};
\node at (6,.5){$\,\cdots$};

\end{tikzpicture}
\]
and cup cobordisms on each trivial circle in $s(\Cs)^\c$, with dots placed according to $\Js$. By Theorem $\ref{thm:BPW}$, the cobordism $(*)$ induces the map
\[
v_{\Is'} \otimes v_{\Is''} \mapsto v_{\Is'}\otimes v_+\otimes v_-\otimes v_{\Is''} + \q^{-1} v_{\Is'} \otimes v_- \otimes v_+ \otimes v_{\Is''}
\]
and the result follows.\par
(2). The generators $y_1$ and $y_2$ are the images of $v_{\Is'} \otimes v_+ \otimes v_-\otimes v_{\Is''}$ and $v_{\Is'} \otimes v_- \otimes v_+ \otimes v_{\Is''}$, respectively, under 
\[
\Cs^\c_E \xrightarrow{\Sigma_\Js} \Cs^\c \to \Cs
\]
where $\Sigma_\Js$ is a cup cobordism on trivial circles with dots placed according to $\Js$. Let $\Phi$ denote the composition 
\[
\Cs^\c_E \xrightarrow{\Sigma_\Js} \Cs^\c \to \Cs \xrightarrow{S} s(\Cs) \to s(\Cs)^\c
\]

Let $C\in s(\Cs)$ denote the trivial circle obtained by surgery along $A$, let $C'\in s(\Cs)^\c$ denote the corresponding circle. In the cobordism $\Phi$, we may cut the neck along $C'$ to write $\Phi$ as a sum of two cobordisms. One of the summands is $0$ by Boerner's relation, and the other is isotopic to the disjoint union of 
\[
(**)\ S^1 \times \,
\begin{tikzpicture}[baseline={([yshift=-.7ex]current bounding box.center)},x=.9em,y=1.2em]

\draw[thick] (0,0)--(0,1)
(2,0)--(2,1)
(3,0) -- (3,.5) to[out=90,in=90] (4,.5) -- (4,0)
(5,0)--(5,1)
(7,0)--(7,1);
\node at (1,.5){$\,\cdots$};
\node at (6,.5){$\,\cdots$};

\end{tikzpicture}
\]
and cup cobordisms on each trivial circle. Observe also that the cup cobordism on $C'$ is dotted. Finally, Theorem \ref{thm:BPW} says that the cobordism $(**)$ induces the map
\begin{align*}
v_{\Is'} \otimes v_+ \otimes v_-\otimes v_{\Is''} & \mapsto \q v_{\Is'} \otimes v_{\Is''}\\
v_{\Is'} \otimes v_- \otimes v_+ \otimes v_{\Is''} & \mapsto v_{\Is'} \otimes v_{\Is''},
\end{align*}
and the result follows.
\end{proof}

We end this subsection with a discussion about recovering classical annular homology. Consider the map $\k \to \Z$ which is the identity on $\Z \subset \k$ and sends $\q$ to $1$. It induces a functor $(-) \otimes_{\k} \Z : \Mod(\k) \to \Mod(\Z)$. Thus one can consider the composition 
\[
\BBN_{\q}(\A) \xrightarrow{\F_{\A_{\q}}} \Mod(\k) \to \Mod(\Z)
\]
which we denote $\F_{\A_{\q}} \otimes_{\k} \Z$. Tensoring with $\Z$ forgets the action of $\q$, in the sense that isotopic cobordisms induce equal maps under $\F_{\A_{\q}} \otimes_{\k} \Z$ even when the isotopy is not required to fix the seam. Let $\Cs$ be a configuration with $e$ essential and $t$ trivial circles. By Lemma \ref{lem:generators differ by power of q}, there is a canonical isomorphism
\[
\F_{\A_{\q}}(\Cs) \otimes_{\k} \Z \cong V_\Z^{\otimes e} \otimes_\Z W_\Z^{\otimes t} 
\]
obtained by picking an isotopy $\Cs^\c \to \Cs$.  It is implicit in \cite{BPW} that $\F_{\A_{\q}} \otimes_{\k} \Z$ is the classical annular functor $\F_\A$; indeed this is straightforward to check using the relations in $\BBN_\q(\A)$ and Theorem \ref{thm:BPW} as in the proof of Proposition \ref{prop:LBM1}. 

\begin{lemma}\label{lem:saddles}
Let $\Cs$ be a configuration with a single surgery arc, and let $S: \Cs \to s(\Cs)$ denote the saddle. Let $x\in \Cs$ be a generator. Then 
\[
S(x) = \sum \varepsilon_y y
\]
where the sum is over generators of $s(\Cs)$ and each $\varepsilon_y$ is either $0$ or a power of $\q$. Moreover, $\varepsilon_y \neq 0$ if and only if $y$ appears in $\F_{\A}(S)(x)$, where $\F_\A$ is the classical (unquantized) annular TQFT. 
\end{lemma}

\begin{proof}
There are six types of saddles to check, corresponding to merges and splits between various combinations of essential and trivial circles as in Figure \ref{fig:annular formulas}. The first part of the lemma was verified for two of these types of saddles in Proposition \ref{prop:LBM1}.
It is straightforward to verify the lemma for the other four using similar arguments. The second part follows from the discussion preceding the lemma.
\end{proof}

\subsection{Ladybug Configurations}\label{sec:a ladybug configuration}

In this subsection, we analyze the \textit{ladybug configuration} (\cite[Definition 5.6]{LS}, see also \cite[Figure 5.1]{LS}). Examining ladybug configurations is crucial in the construction of  Khovanov homotopy types in various contexts. We start with a discussion of ladybug configurations in classical annular homology. We will then examine a particular type of ladybug configuration in quantum annular homology, and we will indicate how the analysis differs from that in classical annular homology. 

We recall the notion of a ladybug configuration. A circle $C\subset \A$ with two surgery arcs forms a ladybug configuration if the endpoints of the two arcs alternate around $C$. We will say a configuration $\Cs$ with surgery arcs has a ladybug configuration if a circle $C$ in $\Cs$ and two of the surgery arcs forms a ladybug configuration.  

First, consider ladybug configurations in classical annular homology. Let $C\subset \A$ be a circle carrying two surgery arcs $A_1$ and $A_2$ which form a ladybug configuration. Figure \ref{fig:ladybug in classical annular homology} illustrates the three possibilities in the annulus. 

\begin{figure}[H]
\centering
    \begin{subfigure}[b]{.3\textwidth}
    \begin{center}
 \includegraphics{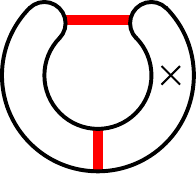}
\caption{}\label{fig:ladybug configuration ex1}
\end{center}
    \end{subfigure}
    \begin{subfigure}[b]{.3\textwidth}
    \begin{center}
 \includegraphics{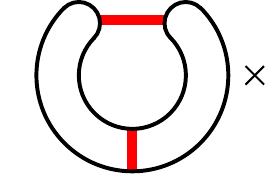}
    \caption{}\label{fig:ladybug configuration ex2}
    \end{center}
    \end{subfigure}
    \begin{subfigure}[b]{.3\textwidth}
    \begin{center}
    \includegraphics{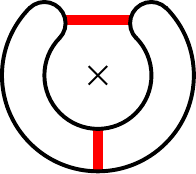}
    \caption{}\label{fig:ladybug configuration ex3}
    \end{center}
    \end{subfigure}
    \caption{Three ladybug configurations in the annulus}\label{fig:ladybug in classical annular homology}
\end{figure}

For $i=1,2$, let $\Cs_i$ denote the configuration obtained by performing surgery along $A_i$, and let $d_i : \F_\A(C) \to \F_\A(\Cs_i)$ denote the maps assigned to the saddles in classical annular Khovanov homology. Let $\Cs'$ denote the final configuration, obtained by performing surgery on $C$ along both $A_1$ and $A_2$. When $C$ is essential, as in Figure \ref{fig:ladybug configuration ex1}, composing two saddle maps yields $0$ (see the formulas in Figure \ref{fig:annular formulas}). Now consider the cases where $C$ is trivial, as in Figure \ref{fig:ladybug configuration ex2} and Figure \ref{fig:ladybug configuration ex3}. The dotted generator $w_-$ is sent to $0$ by the composition of two saddle maps. On the other hand, the two summands appearing in each of $d_1(w_+)$ and $d_2(w_+)$ are mapped to the same element in the final configuration $\Cs'$.  The case of Figure \ref{fig:ladybug configuration ex3} is illustrated in Figure \ref{fig:classical ladybug generators}.

\begin{figure}[H]
\begin{center}
\includegraphics{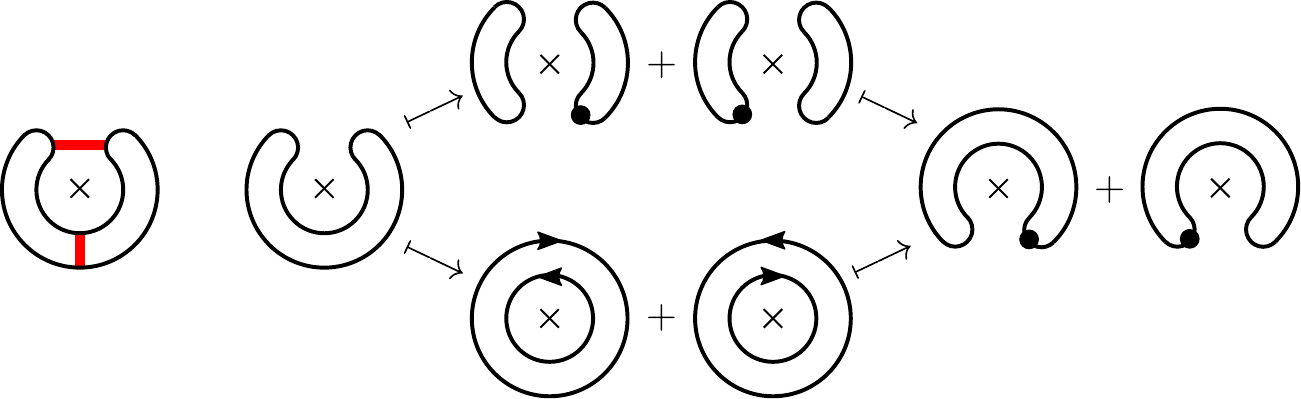}
\end{center}\caption{}\label{fig:classical ladybug generators}
\end{figure}

When constructing stable homotopy refinements in this framework, it is important to have a bijection between these intermediate generators which appear in $d_1(w_+)$ and $d_2(w_+)$, such that these bijections are coherent, in an appropriate sense, within higher dimensional cubes (cf. Section \ref{sec:Burnside categories and functors} below). The bijections are the \textit{ladybug matchings} defined in \cite[Section 5.4]{LS}.

In the case of Figure \ref{fig:ladybug configuration ex2}, the annulus and seam play no role, and the ladybug matching from \cite{LS} can be used without significant alteration for both classical and quantum annular homology.  The remainder of this subsection examines the ladybug configuration of the type in Figure \ref{fig:ladybug configuration ex3} in quantum annular homology. 

Let $\Cs$ be a configuration with two surgery arcs $A_T$ and $A_E$, both having endpoints on a trivial circle $C$ in $\Cs$. Suppose also that surgery along $A_T$ splits $C$ into two trivial circles and surgery along $A_E$ splits $C$ into two essential circles, cf. Figure \ref{fig:LBM ex}. Let $s_T(\Cs)$ and $s_E(\Cs)$ denote the configurations obtained by surgery along $A_T$ and $A_E$ respectively. Let $\Cs'$ denote the final configuration, obtained by surgery along both arcs. Let $C'\in \Cs'$ denote the circle obtained by both surgeries. We have the commutative square.
\begin{equation}\label{eq:square of resolutions for ladybug configuration}
\begin{tikzcd}[row sep=small]
& s_T(\Cs) \arrow[dr, "m_T"] & \\
\Cs \arrow[ur, "\Delta_T"] \arrow[dr, "\Delta_E"'] & & \Cs' \\
& s_E(\Cs) \arrow[ur, "m_E"'] &
\end{tikzcd}
\end{equation}
Figure $\ref{fig:LBM ex}$ exhibits the specific instance of this set-up corresponding directly to Figure \ref{fig:ladybug configuration ex3}, but there are many such cases depending on how $C$ intersects the seam.  As usual, we will not distinguish between cobordisms and their induced maps. 
\begin{figure}[H]
\centering
\includegraphics{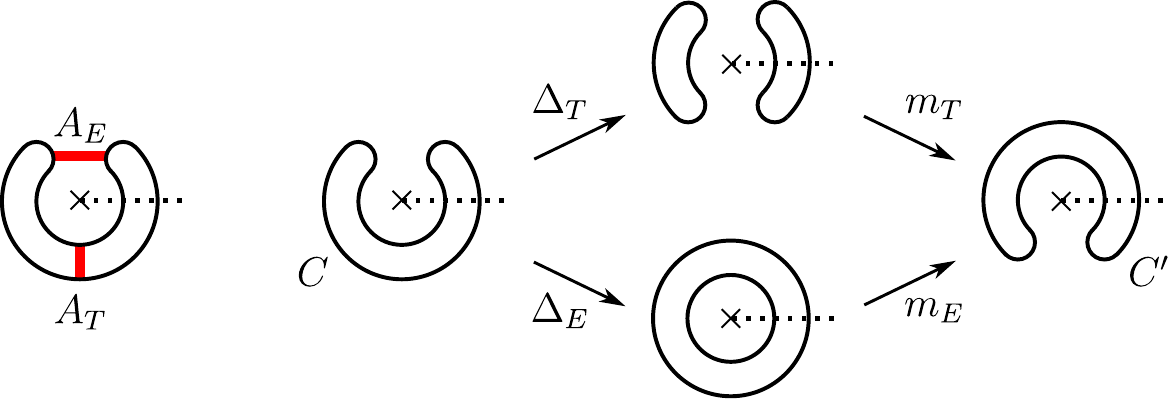}
\caption{}\label{fig:LBM ex}
\end{figure}
We assume that $C$ occurs first in the ordering on trivial circles in $\Cs$. Surgery along $A_T$ splits $C$ into two trivial circles in $s_T(\Cs)$, which we assume are the first two trivial circles in $s_T(\Cs)$. Finally, we order these first two circles as follows. Orient the arc $A_T$ such that it points from the outer essential circle in $s_E(\Cs)$ to the inner one. Declare that the first circle $C_1$ is to the left of $A_T$ and the second $C_2$ is to the right of $A_T$. Our ordering convention is illustrated in Figure \ref{fig:ordering}. 

\begin{figure}
\includegraphics{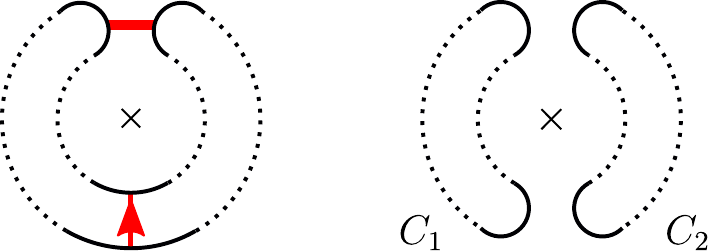}
\caption[scale=.85]{Ordering convention on $C_1$ and $C_2$}\label{fig:ordering}
\end{figure}

Let $x=  v_\Is \otimes w_+ \otimes w_\Js \in \Cs$ be a generator in which $C$ is undotted. By Lemma \ref{lem:saddles}, we obtain 
\[
\Delta_T(x) = \q^k v_\Is \otimes w_-\otimes w_+ \otimes w_\Js + \q^\l v_\Is \otimes w_+ \otimes w_- \otimes w_\Js
\]
for some $k,\l \in \Z$. 

The following corollary implies that, in fact, in the quantum annular setting, there is no need for a ladybug matching for this ladybug configuration 
because intermediate generators are mapped to different elements in the final configuration.

\begin{corollary}\label{cor:1+q^2 splitting}
With the above notation, 
\[
m_E(\Delta_E(x)) = \q^m v_\Is \otimes w_-\otimes w_\Js + \q^{m+2} v_\Is \otimes w_-\otimes w_\Js,
\]
for some $m\in \Z$. Moreover, one of $m_T(\q^k v_\Is \otimes w_-\otimes w_+ \otimes w_\Js)$ or $m_T(\q^\l v_\Is \otimes w_+ \otimes w_- \otimes w_\Js)$ is equal to $\q^m v_\Is \otimes w_-\otimes w_\Js$, and the other is equal to $\q^{m+2} v_\Is \otimes w_-\otimes w_\Js$. 
\end{corollary}

\begin{proof}
In the case when the whole configuration ${\mathcal C}$ consists of just the curve $C$ intersecting the seam in two points as in Figure \ref{fig:LBM ex}, the first statement can be easily checked using the formulas in Figure \ref{fig:quantum annular formula table}. (Also see figure \ref{fig:LBM ex formulas} below.) In full generality it  follows from Proposition \ref{prop:LBM1}. The second statement follows from the commutativity of the square \eqref{eq:square of resolutions for ladybug configuration}.
\end{proof}

\begin{remark} Note that generators for each configuration depend, up to a power of $\q$, on a choice of a cobordism: see Definition
\ref{generators of a configuration} and discussion following it. The powers $k, \ell$, and $m$ above are determined by the cobordisms chosen for the different configurations.
\end{remark}

It will be important for Section \ref{sec:taking the quotient} to know which of $m_T(\q^k v_\Is \otimes w_-\otimes w_+ \otimes w_\Js)$ or $m_T(\q^\l v_\Is \otimes w_+ \otimes w_- \otimes w_\Js)$ is equal to $\q^m v_\Is \otimes w_-\otimes w_\Js$. This is addressed in the following proposition. 

\begin{proposition}\label{prop:LBM2}
With the above notation, 
\begin{align*}
m_T(\q^k v_\Is \otimes w_-\otimes w_+ \otimes w_\Js) &=\q^m v_\Is \otimes w_-\otimes w_\Js\\
m_T(\q^\l v_\Is \otimes w_+ \otimes w_- \otimes w_\Js) &= \q^{m+2} v_\Is \otimes w_-\otimes w_\Js
\end{align*}
\end{proposition}

\begin{proof}
The generator $x =  v_\Is \otimes w_+ \otimes w_\Js$ is the image of $v_\Is$ under $\Cs^\c_E \xrightarrow{\Sigma_\Js} \Cs^\c \xrightarrow{S} \Cs$.  Here $\Sigma_\Js$ is the usual identity cobordism on the essential part together with an undotted cup corresponding to $w_+$ on $C$ and various other cups dotted according to $\Js$, while $S$ is some chosen cobordism from the standard $\Cs^\c$ to $\Cs$. Isotope the resulting disk in $S\Sigma_\Js$ bounding $C$ to a cup cobordism on $C$ to obtain a new cobordism $\Sigma' : \Cs^\c_E \to \Cs$, so that $S\Sigma_\Js =\q^c \Sigma'$ for some $c\in \Z$. 

Since $C$ is trivial, it bounds a disk $D$ in the annulus. Note that $A_T$ lies inside $D$, so we may push it into the cup cobordism on $C$ in $\Sigma'$ to obtain an arc $A$. We may also pull $A_T$ onto the saddle $\Delta_T$ to obtain another arc $A'$. Performing neck-cutting on the circle $A\cup A'$ on $S\Sigma'$ yields 
\[
\Delta_T \Sigma' = S_1 + S_2
\]
where $S_1$ and $S_2$ are labelled such that $S_1$ is dotted on $C_1$ and $S_2$ is dotted on $C_2$. Figure \ref{fig:LBM saddle} illustrates the local picture near $A_T$; the surgery arc $A_T$ is in red, and the arcs $A$, $A'$ are in blue. 
\begin{figure}[H]
\centering
\includegraphics{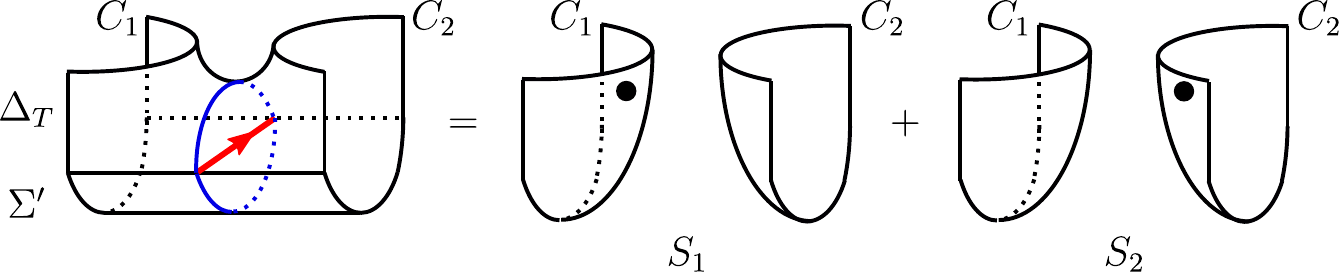}
\caption{}\label{fig:LBM saddle}
\end{figure}
From this we obtain 
\[
\Delta_T(x) = \Delta_TS\Sigma_\Js(v_\Is) = \q^c \Delta_T\Sigma'(v_\Is) = \q^c S_1(v_\Is) + \q^c S_2(v_\Is),
\]
and it follows that 
\begin{align*}
\q^cS_1(v_\Is) &= \q^k v_\Is \otimes w_-\otimes w_+ \otimes w_\Js\\
\q^cS_2(v_\Is) &= \q^\l v_\Is \otimes w_+ \otimes w_- \otimes w_\Js. 
\end{align*}

The relation
\begin{center}
\includegraphics{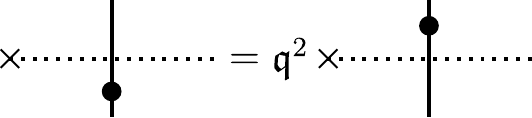}
\end{center}
implies that $\q^d m_TS_1 = m_TS_2$ for some $d\in \Z$. To compute $\q^d$, we need to move the dot on $S_2$ along $C'$ until it is in the same position as the dot on $S_1$, and count (with sign) the number of times the dot intersects the membrane during this process. This is the same as the signed intersection between the seam and one of the essential circles in $s_E(\Cs)$ obtained by surgery on $C$. The situation is depicted below in \eqref{fig:moving dot}; we need to move the dot on the right diagram along the circle, without intersecting the surgery arc, to the other side of the arc. 
\begin{equation}\label{fig:moving dot}
\vcenter{\hbox{
\includegraphics{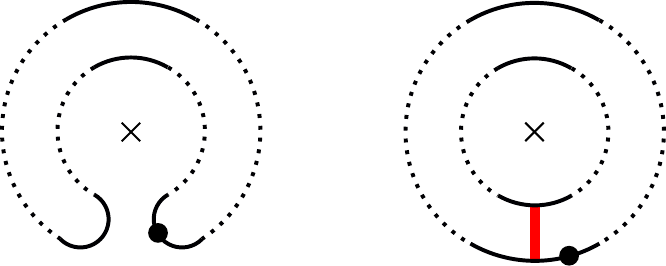}
}}
\end{equation}
Our convention of ordering $C_1$ and $C_2$ (see Figure \ref{fig:ordering}) guarantees that the dot is moved counter-clockwise along an essential circle in $s_E(\Cs)$, so that $m_TS_2= \q^2 m_TS_1$. Finally, this implies 
\[
\q^2 m_T(\q^k v_\Is \otimes w_-\otimes w_+ \otimes w_\Js) = m_T(\q^\l v_\Is \otimes w_+ \otimes w_- \otimes w_\Js)
\]
and the result follows. 
\end{proof}

There is always either a ``right'' or ``left'' choice which is made at the very beginning of defining the ladybug matching (see \cite[Section 5.4]{LS}). Consider the classical annular differential for the ladybug configuration 
of Figure \ref{fig:classical ladybug generators}. The ladybug matching made with the left choice identifies the circles in the middle smoothings as shown in Figure \ref{fig:ladybug matching with left pair on circles}. Then the ladybug matching pairs up the intermediate generators appearing in $\Delta_T(w_+)$ and $\Delta_E(w_+)$ as shown in Figure \ref{fig:ladybug matching with left pair on generators}. 
\begin{figure}[H]
\begin{center}
    \begin{subfigure}[b]{0.4\textwidth}
    \[\includegraphics{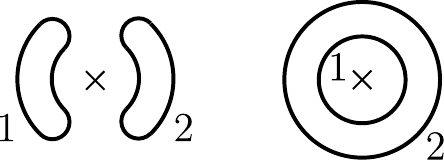}\]
    \caption{}\label{fig:ladybug matching with left pair on circles}
    \end{subfigure}
    \begin{subfigure}[b]{0.4\textwidth}
    \[\includegraphics{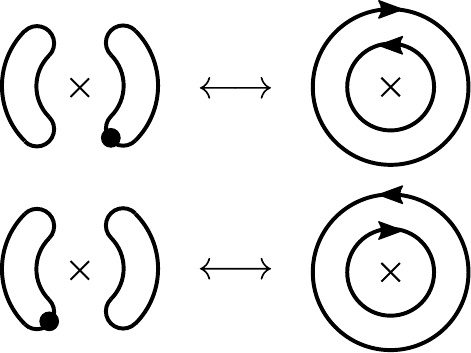}\]
    \caption{}\label{fig:ladybug matching with left pair on generators}
    \end{subfigure}
    \caption{}\label{fig:ladybug matching with left pair}
\end{center}
\end{figure}

Now consider the quantum annular surgery formulas for the same configuration (Figure \ref{fig:LBM ex}) which are detailed in Figure \ref{fig:LBM ex formulas}.
\begin{figure}[H]
\centering
\includegraphics{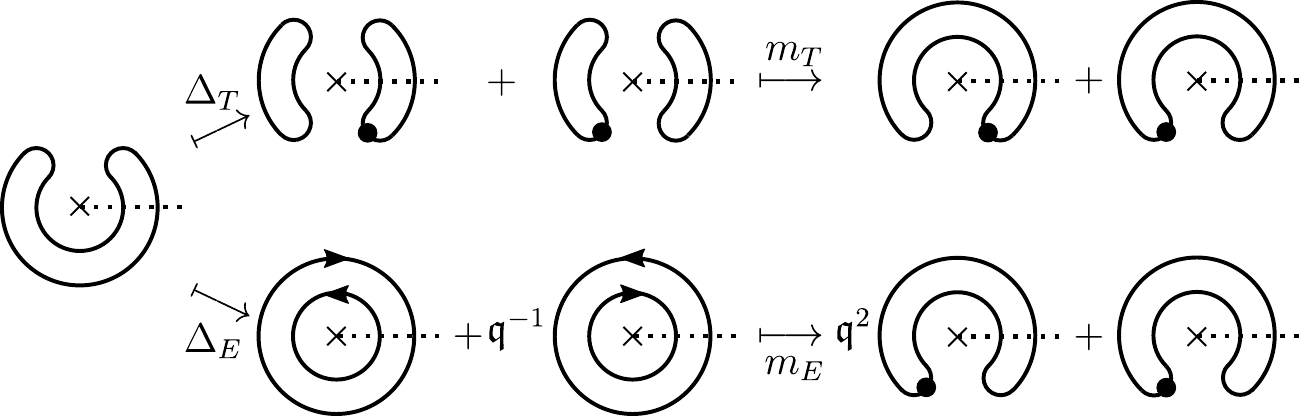}
\caption{Note that the term with coefficient ${\q}^2$ on bottom right matches the term directly above it, because dragging the dot across the seam amounts to multiplication by ${\q}^2$ }\label{fig:LBM ex formulas}
\end{figure}
We see that the intermediate generators get paired up as in \eqref{fig:1+q^2 splitting makes left choice}. 
\begin{equation}\label{fig:1+q^2 splitting makes left choice}
\vcenter{\hbox{
\includegraphics{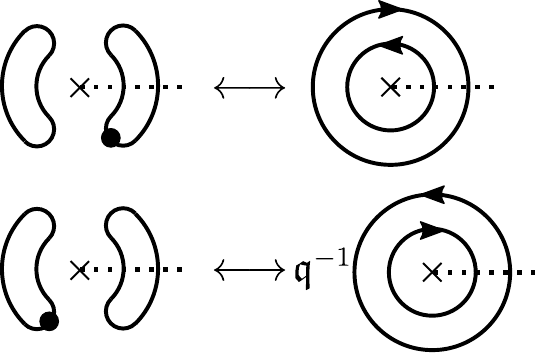}
}}
\end{equation}
Algebraically, the matching is 
\begin{align*}
&w_+\otimes w_- \longleftrightarrow  v_+\otimes v_-\\
&w_-\otimes w_+ \longleftrightarrow \q^{-1} v_-\otimes v_+
\end{align*}
where the ordering on trivial circles follows the convention illustrated in Figure \ref{fig:ordering}. After forgetting powers of $\q$, the matching in \eqref{fig:1+q^2 splitting makes left choice} is consistent with the matching in Figure \ref{fig:ladybug matching with left pair}. Our remaining goal in this section is to show that this holds in general. 

We will use the notation and conventions established in this section. By Proposition \ref{prop:LBM1} (1), we can write 
\[
\Delta_E(x) = \q^a v_{\Is'} \otimes v_+ \otimes v_-\otimes v_{\Is''}\otimes w_\Js+ \q^{a-1} v_{\Is'} \otimes v_- \otimes v_+\otimes v_{\Is''}\otimes w_\Js
\]
for some $a\in \Z$. Proposition \ref{prop:LBM1} (2) and Corollary \ref{cor:1+q^2 splitting} imply that in the quantum setting, the pairing on intermediate generators is forced to be 
\begin{equation}\label{eq:1+q^2 matching}
\begin{split}
&\q^\l v_\Is \otimes w_+ \otimes w_- \otimes w_\Js \longleftrightarrow \q^a v_{\Is'} \otimes v_+ \otimes v_-\otimes v_{\Is''}\otimes w_\Js   \\
&\q^k v_\Is \otimes w_- \otimes w_+ \otimes w_\Js  \longleftrightarrow \q^{a-1} v_{\Is'} \otimes v_- \otimes v_+\otimes v_{\Is''}\otimes w_\Js. 
 \end{split}
\end{equation}

\begin{corollary}\label{cor:1+q^2 splitting makes left choice}
After forgetting powers of $\q$, the matching in \eqref{eq:1+q^2 matching} is the same as the ladybug matching made with the left pair. 
\end{corollary}

\begin{proof}
Looking at the surgery arc $A_T$, the left choice makes the following identification on circles in $s_E(\Cs)$ and $s_T(\Cs)$.
\begin{equation}
\vcenter{\hbox{
\includegraphics{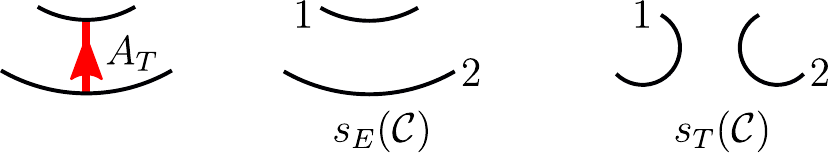}
}}
\end{equation}
Therefore the ladybug matching makes the following identification on generators
\begin{equation}
\vcenter{\hbox{
\includegraphics[scale=.85]{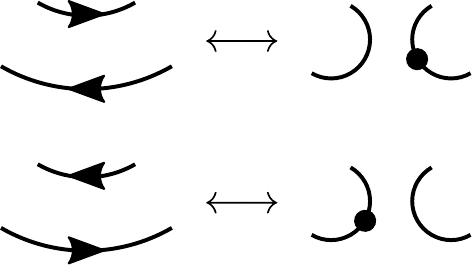}
}}
\end{equation}
Comparing with  our ordering convention on the circles in $s_T(\Cs)$  in  Figure \ref{fig:ordering} demonstrates that this is consistent with \eqref{eq:1+q^2 matching}.
\end{proof}

\section{Burnside Categories and Functors}
Following the general strategy of \cite{LLS}, the first step towards lifting a Khovanov homology theory to a spectrum is to build a Burnside functor from the cube category $\2^n$ \cite[Section 2.1]{LLS} to the Burnside category $\B$ \cite[Section 4.1]{LLS} which encodes the information underlying the chain complex in a higher categorical manner, beyond the data of chain groups and differentials.  In this section we review the general framework of such categories and functors along with their equivariant versions, before turning to the specific case of quantum annular Khovanov homology in Section \ref{Quantum Annular Burnside section}.
In particular, a strategy for constructing natural isomorphisms of Burnside functors, outlined in Section \ref{sec:a strategy for constructing natural isomorphisms}, will be used in follow-up sections.

\subsection{The Cube Category}

We first recall the \textit{cube category} $\2^n$ from \cite[Section 2.1]{LLS}. The objects of $\2^n$ are the elements of $\{0,1\}^n$, thought of as vertices of the $n$-dimensional cube $I^n$. There is a natural partial order on $\{0,1\}^n$: for vertices $u=(u_1,\ldots, u_n)$ and  $v=(v_1, \ldots, v_n)$ in  $\{0,1\}^n$, write $u\geq v$ if each $u_i \geq v_i$. The set of morphisms $\Hom_{\2^n}(u,v)$ is defined to be empty unless $u\geq v$, in which case $\Hom_{\2^n}(u,v)$ consists of a single element, denoted $\varphi_{u,v}$. Therefore, for $u\geq w$, we have $\varphi_{u,w} = \varphi_{v,w} \circ \varphi_{u,v}$ for any $v$ such that $u\geq v \geq w$. We note that the edges in $\2^n$ point in the opposite direction of those in the cube of resolutions of a link diagram. 

For a vertex $u=(u_1,\ldots, u_n)$, define $\lr{u} := \sum_i u_i$. Write $u\geq_k v$ if $ u \geq v$ and $\lr{u} - \lr{v} = k$. In particular, $u \geq_1 v$ means there is an edge from $u$ to $v$ in the cube $I^n$. 

\subsection{Burnside Categories and Functors}\label{sec:Burnside categories and functors} 

This section discusses the Burnside category, $\B$, and Burnside functors, following \cite[Section 4.1]{LLS}. 

For sets $X$ and $Y$, a \textit{correspondence} from $X$ to $Y$ is a triple $(A,s,t)$ where $A$ is a set and $s:A\to X$, $t: A\to Y$ are functions, called the \textit{source map} and \textit{target map}, respectively. We will often denote a correspondence by $X \lar{s} A \rar{t} Y$ or simply $X \ot A \to Y$. Given correspondences $X \lar{s_A} A \rar{t_A} Y$ and $Y \lar{s_B} B \rar{t_B} Z$, their composition $(B, s_B, t_B) \circ (A, s_A, t_A)$ is the correspondence $(C, s, t)$ from $X$ to $Z$ obtained as the fiber product
\[
C = B\times_Y A = \{ (b,a) \in B\times A \mid s_B(b) = t_A(a) \}
\]
with the source and target maps 
\[
s(b,a) = s_A(a) \hskip3em t(b,a) = t_B(b).
\]
The composition can be summarized by the fiber product diagram
\[
\begin{tikzcd}
& & C \arrow[dl] \arrow[dr] \arrow[ddrr, bend left= 40, "t"] \arrow[ddll, bend right=40, "s"'] &  &\\
& A \arrow[dl, "s_A"'] \arrow[dr, "t_A"] & & B \arrow[dl, "s_B"'] \arrow[dr, "t_B"] &  \\
X & & Y & & Z
\end{tikzcd}
\]

A \textit{morphism} from a correspondence $X \lar{s_A} A \rar{t_A} Y$ to a correspondence $X \lar{s_B} B \rar{t_B} Y$ is a bijection $f:A\to B$ which commutes with the source and target maps. That is, $f$ is a bijection fitting into the commutative diagram 
\[
\begin{tikzcd}[row sep=small]
 & A\arrow[dd, "f"] \arrow[dl ,"s_A"'] \arrow[dr, "t_A"] & \\
 X & & Y\\
 & B \arrow[ul, "s_B"] \arrow[ur, "t_B"'] & 
\end{tikzcd}
\]

The collection of sets, correspondences between them, and morphisms of correspondences forms a \textit{bicategory} in the language of \cite{BPW} or, equivalently, a \textit{weak $2$-category} in the language of \cite{LLS}. The objects are sets, $1$-morphisms are correspondences, and $2$-morphisms are morphisms of correspondences. We will use the terms bicategory and weak $2$-category interchangebly. A quick reference for the notion of bicategories is \cite[Appendix A.4]{BPW}.

The identity $1$-morphism of a set $X$ is the identity correspondence 
\[X \lar{\id_X} X \rar{\id_X} X.
\]
Given correspondences $X \lar{s_A} A \rar{t_A} Y$ and $Z \lar{s_B} B \rar{t_B} X$, the compositions $(A, s_A, t_A) \circ (X, \id_X, \id_X)$ and $(X, \id_X, \id_X) \circ (B, s_B, t_B)$ are not equal to $(A, s_A, t_A)$ and $(B, s_B, t_B)$, but there are natural $2$-morphisms 
\begin{align*}
(A, s_A, t_A) \circ (X, \id_X, \id_X) \cong (A, s_A, t_A) \\
(X, \id_X, \id_X) \circ (B, s_B, t_B) \cong (B, s_B, t_B)
\end{align*}
Similarly, composition of correspondences is not strictly associative, but is associative up to natural isomorphism.  This is the sense in which we have only a weak 2-category, as opposed to a strict 2-category.

The \textit{Burnside category}, denoted $\B$, is the sub-bicategory of the above consisting of finite sets and finite correspondences. Even though $\B$ is a bicategory, it will always be referred to as a category. 

The construction of Khovanov homotopy types in \cite{LLS}, \cite{SSS} utilizes functors $F: \2^n \to \B$, which we explain here. First, make $\2^n$ into a (strict) $2$-category by introducing only identity $2$-morphisms. There is a notion of a \textit{lax 2-functor} between $2$-categories, and also of a \textit{strictly unitary lax 2-functor}. The complete definitions, consisting of a slew of data and various natural morphisms, can be found in \cite[Definition 4.2]{LLS} and \cite[Defintion 4.3]{LLS}. We will only be interested in the notion of a \textit{Burnside functor}, which is a strictly unitary lax 2-functor $F: \2^n \to \B$. Lemma \ref{lem:hexagon relation} specifies the data needed to define a Burnside functor uniquely up to natural isomorphism (see Section \ref{sec:natural transformations of Burnside functors} for the definition of natural isomorphisms).

\begin{lemma}\emph{(\cite[Lemma 4.5]{LLS}, \cite[Proposition 4.3]{LLS2})}\label{lem:hexagon relation}
Consider the following data: 

\begin{itemize}
\item A finite set $F(u)$ for each vertex $u\in \2^n$.

\item A finite correspondence $F(\varphi_{u,v})$ from $F(u)$ to $F(v)$ for each pair of vertices $u,v\in \2^n$ with $u\geq_1 v$.

\item A $2$-morphism
\[
F_{u,v,v',w} : F(\varphi_{v,w}) \circ F(\varphi_{u,v}) \to F(\varphi_{v',w}) \circ F(\varphi_{u,v'})
\]
for each $2$-dimensional face of $\2^n$ with vertices $u,v,v',w$ satisfying $u\geq_1 v, v' \geq_1 w$.

\end{itemize}

Suppose also that the above data satisfies the following conditions:

\begin{enumerate}[label= \emph{(\arabic*)}]

\item $F_{u,v,v',w}^{-1} = F_{u,v',v,w}$

\item For every $3$-dimensional sub-cube of $\2^n$ as in Figure \ref{fig:3D cube}, the hexagon of Figure \ref{fig:hexagon} commutes.

\begin{figure}[H]
\centering
    \begin{subfigure}[b]{0.3\textwidth}
    \[
    \begin{tikzcd}[row sep = small, column sep = small]
  &v'   \ar{rr} \ar{dd}  & &  w  \ar{dd}  \\
    u \ar[crossing over]{rr} \ar{dd} \ar{ur} & & v \ar{ur} \\
      & w'  \ar{rr}  & &  z \\
    v'' \ar{rr} \ar{ur} && w'' \ar{ur} \ar[from=uu,crossing over]
\end{tikzcd}
    \]
    \caption{3D cube}\label{fig:3D cube}
    \end{subfigure}\hskip5em
    \begin{subfigure}[b]{0.5\textwidth}
    \[
    \begin{tikzcd}[row sep = small, column sep = small, outer sep = 2pt]
     & \c\ar[rr, "F_{v',w,w',z} \times \emph{\id}"] &  & \c \ar[ddr, "\emph{\id} \times F_{u,v',v'',w'}"] &  \\
     &  &  &  &  \\
    \c \ar[uur, "\emph{\id} \times F_{u,v,v',w}"] \ar[ddr, "F_{v,w,w'',z} \times \emph{\id}"'] &  &  &  & \c \\
     &  &  &  &  \\
     & \c \ar[rr, "\emph{\id} \times F_{u,v,v'',w''}"'] &  & \c \ar[uur, "F_{v'',w'',w',z}\times \emph{\id}"']&  
    \end{tikzcd}
    \]
    \caption{The hexagon relation}\label{fig:hexagon}
    \end{subfigure}
    \caption{}\label{fig:hexagon relation}
\end{figure}
\end{enumerate}
\noindent
Then the data can be extended to a strictly unitary lax $2$-functor $F:\2^n \to \B$, which is unique up to natural isomorphism. 
\end{lemma}

The hexagon of Figure \ref{fig:hexagon} comes from two ways traversing the faces of the 3-dimensional cube, starting from the correspondence $F(\varphi_{w,z}) \circ F(\varphi_{v,w}) \circ F(\varphi_{u,v})$ and ending at $F(\varphi_{w',z}) \circ F(\varphi_{v'',w'})\circ F(\varphi_{u,v''})$. The top half of the hexagon comes from traversing the faces as in Figure \ref{fig:going around the cube1}, and the bottom half comes from traversing the faces as in Figure \ref{fig:going around the cube2}.  Lemma \ref{lem:hexagon relation} states that the hexagon relation is enough to guarantee that the functor is coherent on all higher dimensional cubes as well.

\begin{figure}[H]
    \begin{subfigure}[b]{0.4\textwidth}
 \[
    \begin{tikzcd}[row sep = small, column sep = small]
  &v'\ar[dddl, Rightarrow, "3"]   \ar{rr} \ar{dd}  & &  w \ar[ddll,bend left= 20, Rightarrow, "2"]  \ar{dd}  \\
    u \ar[crossing over]{rr} \ar{dd} \ar{ur} & & v \ar[ul, Rightarrow, "1"'] \ar{ur} \\
      & w'  \ar{rr}  & &  z \\
    v'' \ar{rr} \ar{ur} && w'' \ar{ur} \ar[from=uu,crossing over]
\end{tikzcd}
\]
\caption{}\label{fig:going around the cube1}
    \end{subfigure}
    \begin{subfigure}[b]{0.4\textwidth}
  \[
    \begin{tikzcd}[row sep = small, column sep = small]
  &v'   \ar{rr} \ar{dd}  & &    w  \ar{dd} \ar[dddl, Rightarrow, "1"'] \\
    u \ar[crossing over]{rr} \ar{dd} \ar{ur} & & v \ar{ur}  \\
      & w'  \ar{rr}  & &  z \ar[from=ll] \\
    v'' \ar{rr} \ar{ur} \ar[from=uurr, crossing over, bend right=20, Rightarrow, start anchor=south west, "2"'] && w'' \ar{ur} \ar[from=uu,crossing over] \ar[from=uuur,crossing over, Rightarrow] \ar[ul, Rightarrow, "3"]
\end{tikzcd}
    \]
    \caption{}\label{fig:going around the cube2}
    \end{subfigure}
    \caption{}\label{fig:where hexagon comes from}
\end{figure}

\begin{remark}\label{rem:equivalent hexagon relation}
We end this subsection with some comments about verifying the hexagon relation which will be useful in proving Theorem \ref{thm:the quantum annular Burnside functor}. Suppose we have $2$-morphisms $F_{u,v,v',w}$ for each square face $u\geq_1 v,v' \geq_1 w$ of $\2^n$, which satisfy $F_{u,v,v',w}^{-1} = F_{u,v',v,w}$ as in condition (1) of Lemma \ref{lem:hexagon relation}. Then verifying the hexagon relation is equivalent to the following. Start at the correspondence $F(\varphi_{w,z}) \circ F(\varphi_{v,w}) \circ F(\varphi_{u,v})$ and traverse the six faces of the cube using the $2$-morphisms; i.e., first move across the three faces as shown in Figure \ref{fig:going around the cube1}, and then move across the remaining three faces as in Figure \ref{fig:going around the cube2}, except in the reverse order. Composing these six $2$-morphisms yields a $2$-morphism
\[
\Phi_{u,v,w,z}: F(\varphi_{w,z}) \circ F(\varphi_{v,w}) \circ F(\varphi_{u,v}) \to F(\varphi_{w,z}) \circ F(\varphi_{v,w}) \circ F(\varphi_{u,v}).
\]
Verifying commutativity of the hexagon of Lemma \ref{lem:hexagon relation} is equivalent to verifying that $\Phi_{u,v,w,z}$ is the identity.  Moreover, for each 3-dimensional sub-cube of $\2^n$, it suffices to verify $\Phi_{u,v,w,z}=\id$ for just one tuple of vertices $u\geq_1 v \geq_1 w \geq_1 z$ within the sub-cube.

Furthermore, such verifications are immediate under certain circumstances which we now describe.  Let $A$ denote the correspondence $F(\varphi_{w,z}) \circ F(\varphi_{v,w}) \circ F(\varphi_{u,v})$, and let $s: A\to F(u)$, $t: A\to F(z)$ denote the source and target maps respectively. Suppose that for every $x\in F(u)$ and $y\in F(z)$, $s^{-1}(x) \cap t^{-1}(y)$ is either empty or has one element. Let $a\in A$ and let $a' = \Phi_{u,v,w,z}(a)$. Since $\Phi_{u,v,w,z}(a)$ is a $2$-morphism, we have $s(a') = s(a)$ and $t(a') = t(a)$. Then $a'=a$, so the hexagon relation is satisfied for this $3$-dimensional sub-cube. In this situation, we will say that this $3$-dimensional cube is \textit{simple}\label{def:simple correspondence}.
\end{remark}

\subsection{The Equivariant Burnside Category}\label{sec:the equivariant Burnside category}

Let $G$ be a finite group. The \textit{$G$-equivariant Burnside category}, denoted $\B_G$, is an equivariant analogue of $\B$. Objects of $\B_G$ are finite free $G$-sets. A $1$-morphism from $X$ to $Y$ is a triple $(A, s, t)$ where $A$ is another finite free $G$-set, and $s : A\to X$, $t : A\to Y$ are $G$-equivariant maps. We will call such a triple $(A,s,t$) an \textit{equivariant correspondence}. Given equivariant correspondences  $X\lar{s_A} A \rar{t_A} Y$ and $Y \lar{s_B} B \rar{t_B} Z$, the composition $(B, s_B, t_B) \circ (A, s_A, t_A)$ is the same as in $\B$. The $G$-action on $B\times_Y A$ is inherited from the diagonal $G$-action on $B\times A$; that is, $g(b,a) = (gb,ga)$. The additional requirement on $2$-morphisms between correspondences is that they be $G$-equivariant. 

The equivariant Burnside category is discussed in \cite[Section 3.3]{SSS} in the case $G= \Z_2$. Lemma 3.2 in \cite{SSS} (our Lemma \ref{lem:hexagon relation}) gives sufficient conditions for defining a functor $F:\2^n \to \B_{\Z_2}$, and the same conditions clearly work for general $G$. The modification to the data of Lemma \ref{lem:hexagon relation} is that all sets should be finite free $G$-sets and all set maps should be equivariant. Note that if $G=\{1\}$, then $\B_{\{1\}} = \B$, so everything stated about $\B_G$ in the following sections holds just as well for $\B$. 

We will later be interested in the \textit{quotient functor} $(-)/G : \B_G\to \B$, which simply takes the quotient of all sets and set maps. Explicitly, the quotient functor sends a $G$-set $X$ to the set of orbits $X/G = X/(x\sim gx)$. For $G$-sets $X$ and $Y$ and an equivariant map $f:X\to Y$, there is an induced map $f/G : X/G \to Y/G$, given by $(f/G)([x]) = [f(x)]$. The quotient functor sends an equivariant correspondence $X\lar{s} A \rar{t} Y$ to the correspondence $X/G \lar{s/G} A/G \rar{t/G} Y/G$. Likewise, a $2$-morphism $f:A\to B$ is assigned $f/G: A/G \to B/G$. 

\subsection{Totalizations of Burnside Functors}\label{sec:totalizations of Burnside functors}
Associated to any Burnside functor $\2^n \to \B_G$ is a chain complex, called the \textit{totalization} (see \cite[Definition 5.1]{LLS2}, \cite[Section 3.6]{SSS}). 
For a set $X$, let $\Ab(X)$ denote the free abelian group generated by $X$. Let $X \lar{s} A \rar{t} Y$ be an equivariant correspondence. Define a map $\Ab(A) : \Ab(X) \to \Ab(Y)$ by 
\begin{equation}\label{eq:totalization}
\Ab(A)(x) = \sum_{a\in s^{-1}(x)} t(a)
\end{equation}
Let $F:\2^n \to \B_G$ be a Burnside functor. For $u\geq v$, let $A_{u,v}$ denote the correspondence assigned by $F$ to the morphism $\varphi_{u,v} : u \to v$. The complex $\Tot(F)$ is defined by 
\[
\Tot(F) : = \bigoplus_{ u \in \2^n} \Ab (F(u))
\]
with the term $\Ab(F(u))$ in homological degree $\lr{u}$. The differential
\[
\d : \bigoplus_{\lr{u} = i } \Ab(F(u)) \to \bigoplus_{\lr{v}=i+1} \Ab(F(v))
\]
is given on summands by maps $\d_{u,v} : \Ab(F(u)) \to \Ab(F(v))$, for $\lr{u} = i$, $\lr{v} = i+1$, defined as 
\[
\d_{u,v}(x) = (-1)^{s_{u,v}} \Ab(A_{u,v}).
\]
The sign assignment $s_{u,v}$ ensures that $\d^2=0$ (see \cite[Section 2.7]{BN2}, also \cite[Definition 4.5]{LS} for a discussion of $s_{u,v}$). 

If $X$ is a $G$-set, then $\Ab(X)$ is naturally a $\Z[G]$-module. Moreover, If $X \lar{s} A \rar{t} Y$ is an equivariant correspondence, then the map $\Ab(A)$ is $\Z[G]$-linear. Thus if $F$ is an equivariant Burnside functor taking values in $\B_G$, then $\Tot(F)$ is a complex of $\Z[G]$-modules.

\subsection{Natural Transformations of Burnside Functors}\label{sec:natural transformations of Burnside functors}

There is a canonical identification $\2^{n+1} = \2 \times \2^n$. A natural transformation $\eta: F_1\to F_0$ of Burnside functors $F_1, F_0: \2^n \to \B_G$ is a functor $\eta:\2^{n+1} \to \B_G$ such that the restriction of $\eta$ to $\{i\} \times \2^n$ is equal to $F_i$. 
A \textit{natural isomorphism} from $F_1$ to $F_0$ is a natural transformation $\eta: F_1\to F_0$ such that $\eta(e_u) : F_1(u) \to F_0(u)$ is an isomorphism in $\B_G$ for each vertex $u$.

In the context of natural transformations, we will think of $\2^{n+1} = \2 \times \2^n$ as two``horizontal"  copies of $\2^n$ with vertical edges connecting them, pointing downwards. The top copy of $\2^n$ corresponds to $\{1\} \times \2^n \subset \2 \times \2^n$, and likewise the bottom copy corresponds to $\{0\}\times \2^n \subset \2\times \2^n$. Recall that for $u\geq v$, $\varphi_{u,v}$ denotes the unique element in $\Hom_{\2^n}(u,v)$. We distinguish two types of morphisms in $\2 \times \2^n$. First, for each $u\in \2^n$, there is a morphism 
\[
(\varphi_{1,0}, \id_u): (1,u) \to (0,u).
\]
We denote this edge by $\edge_u$, and think of it a vertical arrow
\[
\begin{tikzcd}
(1,u)\ar[d, "\edge_u"']\\
(0,u)
\end{tikzcd}
\]
The second type consists of morphisms in $\2\times \2^n$ of the form \[
(\id_i,\varphi_{u,v}): (i,u) \to (i,v)
\]
where $i \in \{0,1\} $, and $u,v\in \2^n$ with $u\geq v$.
We will denote these morphisms by $\varphi^i_{u,v}$, and think of them as living in the ``horizontal'' cube $\{i\} \times \2^n$. 

With these conventions, the diagram
\begin{equation}\label{eq:horizontal maps in nat transf}
\begin{tikzcd}[column sep = tiny, row sep = small]
  & (i,v')   \ar[rr, "\varphi^i_{v',w}"]  & &   (i,w)    \\
    (i,u) \ar[rr, "\varphi^i_{u,v}"']  \ar[ur, "\varphi^i_{u,v'}"] & & (i,v) \ar[ur, "\varphi^i_{v,w}"']
\end{tikzcd}
\end{equation}
lives in the horizontal cube $\{i\}\times \2^n$, and the diagram 
\begin{equation}\label{eq:vertical maps in nat transf}
\begin{tikzcd}
(1,u) \ar[d, "\edge_u"'] \ar[r, "\varphi^1_{u,v}"] & (1,v) \ar[d, "\edge_v"] \\
(0,u) \ar[r, "\varphi^0_{u,v}"'] & (0,v)
\end{tikzcd}
\end{equation}
is between the horizontal cubes. In the context of natural transformations, we will often not label some or all of the edges, with the understanding that the above conventions $\eqref{eq:horizontal maps in nat transf}$ and $\eqref{eq:vertical maps in nat transf}$ are followed. To define a natural transformation $\eta:F_1\to F_0$, one needs to specify a correspondence $\eta(\edge_u)$ for each vertical edge $\edge_u$, 
a $2$-morphism $\eta_{u,v} : \eta(e_v) \circ \eta(\varphi^1_{u,v}) \to \eta(\varphi^0_{u,v}) \circ \eta(e_u)$ for each vertical face as in \eqref{eq:vertical maps in nat transf}, and verify that the hexagon of Lemma \ref{lem:hexagon relation} commutes. 

Given a natural transformation $\eta: F_1\to F_0$, there is an induced map $\Tot(\eta) : \Tot(F_1) \to \Tot(F_0)$. The map $\Tot(\eta)$ is defined on each summand by  $\Ab(\eta(\edge_u)): \Ab(F_1(u)) \to \Ab(F_0(u))$.

\subsection{A Strategy for Constructing Natural Isomorphisms}\label{sec:a strategy for constructing natural isomorphisms}

We will have several occasions to show that two Burnside functors are isomorphic (Propositions \ref{prop:changing generators}, \ref{prop:isotoping the link diagram}, and \ref{prop:quotient of QABF}).  The general strategy is the same in all these cases, so we outline it here. 

Note that a correspondence $X \lar{s} A \rar{t} Y$, thought of as a morphism in $\B_G$, is an isomorphism if and only if $s$ and $t$ are bijective. In particular, given an equivariant bijection $t: X\to Y$, the correspondence $X \lar{\id} X \rar{t} Y$ is an isomorphism in $\B_G$, with inverse $Y \lar{t} X \rar{\id} X$. 

Suppose we are given two functors $F_1, F_0 : \2^n \to \B_G$ and equivariant bijections $\psi_{u}: F_1(u) \to F_0(u)$ for each vertex $u \in \2^n$. For $u\geq_1 v$, let
\[
F_i(u) \lar{s^i_{u,v}} A_{u,v}^i \rar{t^i_{u,v}} F_i(v)
\] be the correspondence assigned to the edge $\varphi_{u,v}:u\to v$ by $F_i$, for $i=0,1$. Suppose also that for each $u\geq_1 v$, the following conditions hold. 

\smallskip

\begin{adjustwidth}{12pt}{5pt}
\begin{enumerate}[label= (NI \arabic*)]

\item\label{nat iso condition1} $A^i_{u,v} \subset F_i(u) \times F_i(v)$, and the $G$-action on $A^i_{u,v}$ is inherited from the diagonal $G$-action on $F_i(u)\times F_i(v)$ (i.e., $g(x,y) = (g x, g y)$ for $g\in G$, $(x,y) \in F_i(u) \times F_i(v)$). 

\smallskip

\item\label{nat iso condition2} The map $F_1(u) \times F_1(v) \rar{\psi_u\times\psi_v} F_0(u) \times F_0(v)$ restricts to an bijection $A^1_{u,v} \to A^0_{u,v}$, denoted $\psi_{u,v}$. 

\smallskip

\item\label{nat iso condition3} The source and target maps, $s^i_{u,v}$ and $t^i_{u,v}$, are restrictions of the projections $F_i(u) \twoheadleftarrow F_i(u)\times F_i(v) \twoheadrightarrow F_i(v)$.
\end{enumerate}
\end{adjustwidth}

\smallskip
In this situation, we have a systematic method for building a natural isomorphism $\eta : F_1 \to F_0$ using Lemma \ref{lem:hexagon relation} as follows.  Define $\eta$ on objects by
\[
\eta(i,u) := F_i(u)
\]
for $i\in \{0,1\}$ and $u\in \{0,1\}^n$. We then define $\eta$ on each vertical edge $\edge_u : (1,u) \to (0,u)$ by
\[
\eta(\edge_u) = \left( F_1(u) \lar{\id} F_1(u) \rar{\psi_u} F_0(u) \right),
\]
That is, the underlying set of the correspondence $\eta(\edge_u)$ is simply $F_1(u)$, the source map is the identity, and the target map is the given equivariant bijection $\psi_{u}$.

We have now specified $\eta$ on objects and edges. It remains to define the $2$-morphisms for each square face of $\2^{n+1}$. Since $\eta$ must restrict to $F_i$ on $\{i\}\times \2^n$, we need only to specify a $2$-morphism
\[
\eta_{u,v}: F_1(v) \times_{F_1(v)} A^1_{u,v} \to A^0_{u,v} \times_{F_0(u)} F_1(u).
\]
corresponding to the vertical square faces \eqref{eq:vertical maps in nat transf} of $\2\times\2^n$.

The situation is illustrated in Figure \ref{fig:Defining eta_uv}.  Note that every element of $F_1(v) \times_{F_1(v)} A^1_{u,v}$ is of the form $(y,x,y)$, where $(x,y) \in A^1_{u,v} \subset F_1(u) \times F_1(v)$. Likewise, an element of $A^0_{u,v} \times_{F_0(u)} F_1(u)$ is of the form $(a,b, \psi_u^{-1}(a))$, where $(a,b) \in A^0_{u,v} \subset F_0(u) \times F_0(v)$. Condition \ref{nat iso condition1} ensures that the bijections 
\begin{align*}
F_1(v) \times_{F_1(v)} A^1_{u,v} &\to A^1_{u,v}  &&\hskip.5em A^0_{u,v} \to A^0_{u,v} \times_{F_0(u)} F_1(u)\\
(y,x,y) & \mapsto (x,y)  &&(a,b) \mapsto (a,b, \psi_u^{-1}(a))
\end{align*}
are equivariant. Then the composition
\begin{equation}\label{eq:bootstrapping 2-morphism}
F_1(v) \times_{F_1(v)} A^1_{u,v} \to A^1_{u,v} \rar{\psi_{u,v}} A^0_{u,v} \to A^0_{u,v} \times_{F_0(u)} F_1(u)
\end{equation}
is given by $(y,x,y) \mapsto (\psi_u(x), \psi_v(y), x)$, and condition \ref{nat iso condition2} guarantees that it is also an equivariant bijection. Moreover, condition \ref{nat iso condition3} ensures that this composition commutes with the source and target maps. Therefore, we may define the $2$-morphism $\eta_{u,v}$ to be the composition 
$\eqref{eq:bootstrapping 2-morphism}$.

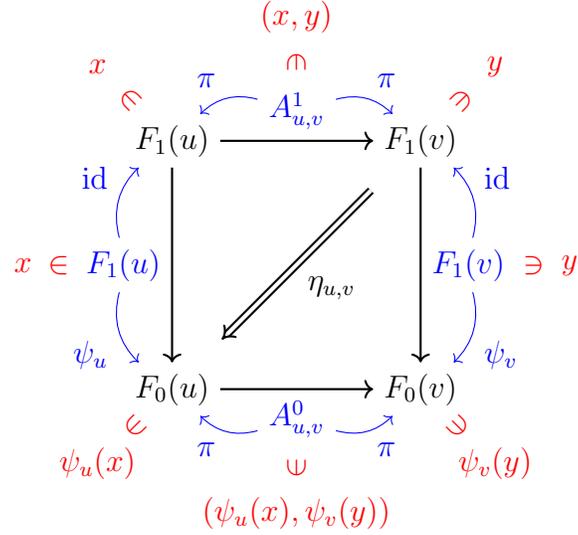
\begin{figure}
\[
\begin{tikzpicture}[x=4em,y=4em]
 \node(F1u) at (-1,1) {$F_1(u)$};
 \node(F0u) at (-1,-1) {$F_0(u)$};
 \node(F1v) at (1,1) {$F_1(v)$};
 \node(F0v) at (1,-1) {$F_0(v)$};
 
 \draw[thick,->] (F1u) to node[midway,above,blue](A1){$A^1_{u,v}$} (F1v);
   \draw[->,blue] (A1) to[bend right] node[midway,above left]{$\pi$} (F1u);
   \draw[->,blue] (A1) to[bend left] node[midway,above right]{$\pi$} (F1v);

 \draw[thick,->] (F1u) to node[midway,left,blue](F1uspan){$F_1(u)$} (F0u);
   \draw[->,blue] (F1uspan) to[bend left] node[midway,above left]{$\id$} (F1u);
   \draw[->,blue] (F1uspan) to[bend right] node[midway,below left]{$\psi_u$} (F0u);
   
 \draw[thick,->] (F0u) to node[midway,below,blue](A0){$A^0_{u,v}$} (F0v);
   \draw[->,blue] (A0) to[bend left] node[midway,below left]{$\pi$} (F0u);
   \draw[->,blue] (A0) to[bend right] node[midway,below right]{$\pi$} (F0v);

 \draw[thick,->] (F1v) to node[midway,right,blue](F1vspan){$F_1(v)$} (F0v);
   \draw[->,blue] (F1vspan) to[bend right] node[midway,above right]{$\id$} (F1v);
   \draw[->,blue] (F1vspan) to[bend left] node[midway,below right]{$\psi_v$} (F0v);
 
 \foreach \set/\elmnt/\coordx/\coordy/\inni in
      { A1/{(x,y)}/0/2/\in , F1v/y/1.6/1.6/\ni , F1vspan/y/2.2/0/\ni , F0v/{\psi_v(y)}/1.6/-1.6/\ni , A0/{(\psi_u(x),\psi_v(y))}/0/-2/\in , F0u/{\psi_u(x)}/-1.6/-1.6/\in , F1uspan/x/-2.2/0/\in , F1u/x/-1.6/1.6/\in }
 {
   \node[red] (\set el) at (\coordx,\coordy) {$\elmnt$};
   \path (\set el) -- (\set) node[sloped,midway,red] {$\inni$};
 }
 
 \draw[thick,-implies,double equal sign distance] (.6,.6) to node[midway,below right]{$\eta_{u,v}$} (-.6,-.6);
\end{tikzpicture}
\]
\caption{We draw the square diagram required for building our natural isomorphism $\eta:F_1\rightarrow F_0$.  Vertices and edges are indicated in black; the correspondences for each edge are indicated in blue, together with their source and target maps.  Specific elements are indicated in red, showing how the 2-morphism $\eta_{u,v}$ (indicated by the double arrow) should be defined such that the entire diagram is consistent.}
\label{fig:Defining eta_uv}
\end{figure}

To extend $\eta$ to a natural transformation, one still needs to check the hexagon relation of Lemma \ref{lem:hexagon relation}. We need only to verify commutativity of the hexagon coming from a three dimensional cube of the form

\[
\adjustbox{scale=.72}{
\begin{tikzcd}[column sep = 6]
  &(1,v')   \ar{rr} \ar{dd}  & &   (1,w)  \ar{dd}  \\
    (1,u) \ar[crossing over]{rr} \ar{dd} \ar{ur} & & (1,v) \ar{ur} \\
      & (0,v')  \ar{rr}  & &  (0,w) \\
    (0,u) \ar{rr} \ar{ur} && (0,v) \ar{ur} \ar[from=uu,crossing over]
\end{tikzcd}
}
\]
Let 
\[
\phi^i_{u,v,v',w} : A^i_{v,w} \times_{F_i(v)} A^i_{u,v} \to A^i_{v',w} \times_{F_i(v')} A^i_{u,v'}
\]
be the $2$-morphism assigned by the functor $F_i$ corresponding to the horizontal square face 
\[
\begin{tikzcd}[column sep = 6]
  & (i,v')   \ar[rr, "\varphi^i_{v',w}"]  & &   (i,w)    \\
    (i,u) \ar[rr, "\varphi^i_{u,v}"']  \ar[ur, "\varphi^i_{u,v'}"] & & (i,v)  \ar[ur, "\varphi^i_{v,w}"'] 
\end{tikzcd}
\]
of $\2\times \2^n$. In this situation, checking that the hexagon of Lemma \ref{lem:hexagon relation} commutes comes down to verifying commutativity of the diagram in \eqref{eq:simplified hexagon relation} below. 

\begin{equation}\label{eq:simplified hexagon relation}
\begin{tikzcd}
A^1_{v,w} \times_{F_1(v)} A^1_{u,v} \arrow[d,"\psi_{v,w} \times \psi_{u,v}"'] \arrow[r, "\phi^1_{u,v,v',w}"] & A^1_{v',w} \times_{F_1(v')} A^1_{u,v'}   \arrow[d, "\psi_{v',w} \times \psi_{u,v'}"]\\
A^0_{v,w} \times_{F_0(v)} A^0_{u,v}  \arrow[r, "\phi^0_{u,v,v',w} "] & A^0_{v',w} \times_{F_0(v')} A^0_{u,v'}
\end{tikzcd}
\end{equation}
If the diagram \eqref{eq:simplified hexagon relation} commutes, then $\eta$ extends to a natural transformation $\eta : F_1 \to F_0$. Moreover, since each $\eta(\edge_u) : F_1(u) \to F_0(u)$ is an isomorphism in $\B_G$, the natural transformation $\eta$ is a natural isomorphism of Burnside functors.

\section{The Quantum Annular Burnside Functor}\label{Quantum Annular Burnside section}

In this section, we construct the \emph{quantum annular Burnside functor} corresponding to an annular link diagram $D$.  Before giving the outline of the section, we emphasize one small but important caveat. In the quantum annular theory over the base ring $\k=\Z[\q, \q^{-1}]$, every configuration is assigned a module which has infinite rank over $\Z$, with generators of the form $\q^k x$ for $\k\in \Z$. 
In our set-up, this would correspond to assigning an infinite set to each vertex in the cube of resolutions. This would require considering spaces of infinitely many boxes in Section \ref{sec:from Burnside functors to stable homotopy types}, and also of CW-complexes with a $\Z$-action. Although we believe that such a version of the theory could be worked out, in the present paper we stay in the context of finite cyclic group actions. This is motivated in part by the fact that a substantial part of equivariant homotopy theory is formulated for compact group actions. To this end, we make the following modification to the quantum annular complex.

For $r> 0$, set $\k_r: = \k/(\q^r-1)$.  Let $\F^r_{\A_{\q}}$ denote the composition
\begin{equation}\label{eq:modified quantum annular TQFT}
\BBN_{\q}(\A) \rar{\F_{\A_{\q}}} \Mod(\k) \rar{(-)\otimes_\k \k_r} \Mod(\k_r).
\end{equation}
We can define a modified quantum annular homology by applying $\F^r_{\A_{\q}}$ to each vertex in the cube of resolutions of an annular link diagram $D$. The result is the same as applying $(-)\otimes_\k \k_r$ to the quantum annular chain complex $CKh_{\A_{\q}}(D)$. Every vertex is assigned a free $\k_r$-module, and the formulas in Section \ref{sec:the quantum annular TQFT} remain true, modulo the relation $\q^r=1$. 

With this modification in place, we proceed as follows.  Given an annular link diagram $D$ with $n$ crossings, we will define the quantum annular Burnside functor $F_\q : \2^n \to \B_G$, where 
\[
G = \langle \q \mid \q^r=1\rangle
\]
is the finite cyclic group of order $r$ with distinguished generator $\q$. Note that there is a natural ring isomorphism $\Z[G] \cong \k_r$, so that it is possible to compare the cellular cohomology of the stable homotopy type, which is a $\Z[G]$-module, with the modified quantum annular homology, which is a $\k_r$-module. The dependence on $r$ will be omitted from the group $G$ and the Burnside functor $F_\q$ in order to simplify the notation.  

Because we want the totalization $\Tot(F_{\q})$ to recover the quantum annular complex $$CKh_{\A_{\q}}(D) \otimes_\k \k_r,$$ we already know what $F_{\q}$ should assign to vertices and edges of $\2^n$. The subtlety here is that generators are defined only up to a power of $\q$, and the formulas for the differential non-trivially depend on the configuration, so the full extent of the analysis in Sections \ref{sec:fixing generators}, \ref{sec:saddle maps} is used here. Once $F_\q$ is determined on vertices and edges, it remains to assign specific bijections to the (identity) 2-morphisms in $\2^n$ and check the hexagon relation of Lemma \ref{lem:hexagon relation}.  This will be done in Section \ref{sec:the quantum annular Burnside functor for a link diagram} for the case $r>2$; this restriction is to guarantee that $\q^2\neq 1$, allowing the use of Corollary \ref{cor:1+q^2 splitting} to simplify the analysis. We will show in Proposition \ref{prop:changing generators} that the Burnside functor $F_\q$ is independent of the choice of generators for each vertex, and in Section \ref{sec:isotoping the link diagram} we will show that planar isotopies of the link diagram induce natural isomorphisms of the corresponding Burnside functors. Finally in Section \ref{sec:case r=1,2} we will address the cases $r=1,2$.

\subsection{The Quantum Annular Burnside Functor for a Link Diagram}\label{sec:the quantum annular Burnside functor for a link diagram}
Let $D$ be a diagram for an annular link with $n$ crossings, which are assumed to be disjoint from the seam. For each $u\in \{0,1\}^n$, let $D_u$ denote the smoothing of $D$ corresponding to $u$. Fix $r>2$, and set $G= \langle \q \mid \q^r=1\rangle$. We will specify the data of the \textit{quantum annular Burnside functor} $F_{\q}: \2^n \to \B_G$. 

For each vertex $u\in \{0,1\}^n$, pick a set of generators $\Gamma(u)$ of $D_u$, following Section \ref{sec:fixing generators}. Define $F_{\q}$ on vertices by 
\begin{equation}\label{eq:F_q on vertices}
F_{\q}(u) = G\times \Gamma(u).
\end{equation}
The $G$-action on $F_{\q}(u)$ is on the first factor: $\q^k \cdot (\q^\l, x) = (\q^{k+\l},x)$. 

We will write an element of $F_{\q}(u)$ as $\q^kx$ instead of $(\q^k, x)$. 

For $u\geq_1 v$, let $d_{v,u}$ denote the map assigned to the edge $v\to u$ by the modified quantum annular functor $\F_{\A_\q}^r$. Recall from Lemma \ref{lem:saddles} that for each $x\in \Gamma(v)$, 
\[
d_{v,u}(x) = \sum_{y\in \Gamma(u)} \varepsilon_y y
\]
where each $\varepsilon_y$ is either $0$ or $\q^k$ for some $k\in \Z$. We will say that $\q^k y$ \textit{appears in} $d_{v,u}(\q^\l x)$ if in the equation 
\[
d_{v,u}(\q^\l x) = \sum_{y\in \Gamma(u)} \varepsilon_y y,
 \]
the coefficient $\varepsilon_y$ is equal to $\q^k$. For $u\geq_1 v$, define the correspondence $A_{u,v} \subset F_{\q}(u) \times F_{\q}(v)$ from $F_{\q}(u)$ to $F_{\q}(v)$ by 
\begin{equation}\label{eq:F_q on edges}
 A_{u,v} = \{ (\q^k y, \q^\l x) \in F_{\q}(u) \times F_{\q}(v) \mid \q^k y \text{ appears in } d_{v,u}(\q^\l x)\}
\end{equation}
The source and target maps of $A_{u,v}$ are  the projections to $F_{\q}(u)$ and $F_{\q}(v)$, respectively. Note that $A_{u,v}$ is a sub $G$-set of $F_{\q}(u)\times F_{\q}(v)$ since $d_{v,u}(\q x) = \q d_{v,u}(x)$. 

We will show that the above data extends to a Burnside functor $F_q : \2^n \to \B_G$. The following lemma will be useful in our analysis of the hexagon relation. 

\begin{lemma}\label{lem:3d cube with ladybug configuration}
Let $\Cs$ be a configuration with three surgery arcs $A_1$, $A_2$, and $A_3$. Let $\Cs'$ denote the curves in $\Cs$ containing the endpoints of the surgery arcs. Assume there is a circle $C$ in $\Cs'$ such that $C\cup A_1 \cup A_2$ forms a ladybug configuration. Then one of the following holds. 
\begin{enumerate}[label= \emph{(\arabic*)}]
\item The diagram $\Cs'\cup A_1 \cup A_2 \cup A_3$ is trivial in the annulus; i.e. $\Cs'$ and the three surgery arcs lie in a disk in $\A$. 
\item The composition of three edge maps is $0$. 
\item The $3$-dimensional cube is simple (see the discussion in Remark \ref{rem:equivalent hexagon relation}). 
\item $C\cup A_1 \cup A_2$ is trivial in the annulus and disjoint from $A_3$. 
\end{enumerate}
\end{lemma}
\begin{proof}
If $C$ is essential in the annulus, then  (2) follows from the neck-cutting and Boerner's relations (see Figures \ref{fig:BN relations}, \ref{fig:Boerner's relation}). We may therefore assume that $C$ is trivial. Let $C'$ denote the (necessarily trivial) circle obtained by performing surgery along both $A_1$ and $A_2$. Note that the result of composing the two saddle maps corresponding to $A_1$ and $A_2$ will send any generator of $\Cs$ in which $C$ is undotted to a sum of elements in which $C'$ is dotted, and will send any generator which is dotted on $C$ to $0$. It therefore suffices to consider the effect of surgery along $A_3$ on a dotted $C'$. 

First, assume that $C \cup A_1 \cup A_2$ is trivial but $\Cs'\cup A_1 \cup A_2 \cup A_3$ is not. There are several cases to consider.  If neither endpoint of $A_3$ is on $C$, then (4) holds.  If precisely one endpoint of $A_3$ is on $C$, then the other endpoint must be on another circle $\overline{C}$, as in Figure \ref{fig:trivial ladybug with three arcs 1}. In this situation, Boerner's relation implies that (2) holds. Finally, suppose both endpoints of $A_3$ are on $C$, as in Figure \ref{fig:trivial ladybug with three arcs 2}. Then surgery along $A_3$ must split $C'$ into two essential circles.  In this situation, (2) holds again, since a dotted trivial circle splitting into two essential circles is sent to $0$.
\begin{figure}[H]
\centering
    \begin{subfigure}[b]{.45\textwidth}
    \begin{center}
    \includegraphics{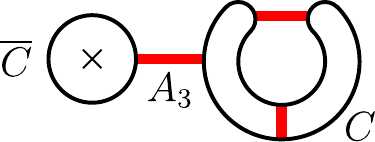}
    \end{center}
    \caption{One endpoint of $A_3$ is on $C$}\label{fig:trivial ladybug with three arcs 1}
    \end{subfigure}
    \begin{subfigure}[b]{.45\textwidth}
    \begin{center}
    \includegraphics{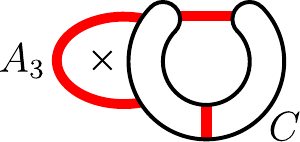}
    \end{center}
    \caption{Both endpoints of $A_3$ are on $C$}\label{fig:trivial ladybug with three arcs 2}
    \end{subfigure}
\caption{The two cases where $C\cup A_1 \cup A_2$ is trivial. Note that these are only schematic depictions, since the interaction of the curves with the seam could be complicated.}\label{fig:trivial ladybug with three arcs}
\end{figure}

Now suppose that the diagram $C\cup A_1 \cup A_2$ is non-trivial.  Then we are in the situation of Section \ref{sec:a ladybug configuration} (see Figure \ref{fig:LBM ex} for example) where Corollary \ref{cor:1+q^2 splitting} shows we have a sum of terms $y+\q^2 y$ in which $C'$ is dotted.  If surgery along $A_3$ either splits off a trivial circle from $C'$ or merges $C'$ with another trivial circle, then (3) holds since the two terms retain distinct powers of $\q$ (see Figure \ref{fig:quantum annular formula table}). If surgery along $A_3$ splits $C'$ into two essential circles or merges $C'$ with an essential circle, then (2) holds as above. 

\end{proof}

\begin{theorem}\label{thm:the quantum annular Burnside functor}
There is a functor $F_{\q}: \2^n \to \B_G$ which extends the data \eqref{eq:F_q on vertices} and \eqref{eq:F_q on edges}.
\end{theorem}

\begin{proof}
Following Lemma \ref{lem:hexagon relation}, it remains to define the $2$-morphisms 
\[
\phi_{u,v,v',w} : A_{v,w} \times_{F_{\q}(v)} A_{u,v} \to A_{v',w} \times_{F_{\q}(v')} A_{u,v'}
\]
for each square face of $\2^n$ with vertices $u \geq_1 v,v'\geq_1 w$, and to verify the hexagon relation.

For all cases except the ladybug configuration, the $2$-morphism $\phi_{u,v,v',w}$ is uniquely determined by the property that it commutes with the source and target maps. Therefore we need to consider only the ladybug configurations. Assume a circle $C$ in $D_w$ carries two surgery arcs as in the ladybug configuration. We distinguish three cases, as in Figure \ref{fig:ladybug in classical annular homology}.

\begin{enumerate}[label= (\alph*)]
\item $C$ is essential.
\item $C$ is trivial, and surgery along both arcs results in trivial circles.
\item $C$ is trivial, surgery along one arc produces two trivial circles, and surgery along the other arc produces two essential circles.
\end{enumerate}
For (a), the composition of two edge maps is $0$. Therefore
\[
A_{v,w} \times_{F_{\q}(v)} A_{u,v} = \varnothing = A_{v',w} \times_{F_{\q}(v')} A_{u,v'},
\]
and there is no $2$-morphism to specify. For (b), we rely on the ladybug matching made with the \textit{left pair} (see \cite[Section 5.4]{LS}). Finally, for (c), note that generators dotted on $C$ are sent to $0$ by the composition of two edges. For generators undotted on $C$, Corollary $\ref{cor:1+q^2 splitting}$ implies that $\phi_{u,v,v',w}$ is uniquely determined by the property that it commutes with the source and target maps. 

It remains to verify the hexagon relation. Let $\Cs$ denote a configuration with three surgery arcs. We may assume that two of the three surgery arcs form a ladybug configuration, since otherwise the $3$-dimensional cube is simple. Then the analysis consists of the four cases in Lemma \ref{lem:3d cube with ladybug configuration}. In case (1), the verification reduces to classical Khovanov homology (see, for example, \cite[Proposition 6.1]{LLS2}). For case (2), the composition of three correspondences coming from any three edge maps is empty, so there is nothing to check. Similarly, In case (3), the hexagon relation follows from the discussion in Remark \ref{rem:equivalent hexagon relation}. Finally, case (4) is straightforward to check by hand since the disjoint arc $A_3$ cannot interfere with the classical Khovanov ladybug matching used on $C\cup A_1 \cup A_2$.
\end{proof}

\begin{proposition}\label{prop:changing generators}
Up to natural isomorphism, $F_{\q}$ is independent of the choices of generators $\Gamma(u)$.
\end{proposition}

\begin{proof}
For each $u\in \{0,1\}^n$, let $\Gamma(u)$, $\Gamma'(u)$ be two sets of generators of $D_u$, obtained by picking different isotopies from the standard configuration $D_u^\c$ to $D_u$. Let $F_{\q}, F_{\q}' :\2^n\to \B_G$ denote the corresponding functors, and let $A_{u,v}$, $A'_{u,v}$ denote the correspondences assigned to edges by $F_{\q}$ and $F_{\q}'$ respectively. We will use the strategy of Section \ref{sec:a strategy for constructing natural isomorphisms} to build a natural isomorphism $F_\q \to F_\q'$. 

There is a clear bijection $\Gamma(u) \to \Gamma'(u)$, denoted $x\mapsto x'$. Lemma \ref{lem:generators differ by power of q} says that for any $x\in \Gamma(u)$, there exists $m_x\in \Z$ such that $x=\q^{m_x} x'$. Consider the equivariant bijection $\psi_{u}: F_{\q}(u) \to F_{\q}'(u)$ defined by $\psi_u(\q^k x) =\q^{k+ m_x}x'$. Observe that conditions \ref{nat iso condition1} and \ref{nat iso condition3} in Section \ref{sec:a strategy for constructing natural isomorphisms} hold by definition of $F_\q$ and $F_\q '$. Condition \ref{nat iso condition2} holds since $x= \q^{m_x} x'$. To complete the proof, observe that the diagram \eqref{eq:simplified hexagon relation} commutes, since the vertical maps act on generators by multiplication by powers of $\q$, and thus do not interfere with the ladybug matchings. 
\end{proof}

\subsection{Isotoping the Link Diagram} \label{sec:isotoping the link diagram}
In this section we show that a planar isotopy of link diagrams induces a natural isomorphism between the corresponding quantum annular Burnside functors.  Elementary isotopies away from the seam are trivial, but isotopies which involve intersections with the seam need to be handled more carefully.  These results will also be used to show Reidemeister invariance for the stable homotopy type in Section \ref{sec:defining the homotopy type}.

\begin{proposition}\label{prop:isotoping the link diagram}
Let $D$ be a link diagram with $n$ crossings, and let $D'$ be a link diagram obtained from $D$ by one of the following moves.
\begin{enumerate}[label= \emph{(\arabic*)}]
\item Moving an arc (as in the $P^{\pm{1}}$ and $N^{\pm{1}}$ moves of Figure $\ref{fig:PN moves}$) across the seam.
\item Moving a crossing across the seam (see Figure $\eqref{fig:isotoping crossing}$)
\end{enumerate}
Let $F_{\q}$ and $F_{\q}'$ be Burnside functors for $D$ and $D'$ respectively. Then $F_{\q}$ is naturally isomorphic to $F_{\q}'$. 
\end{proposition}

\begin{proof}
We will again follow the strategy of Section \ref{sec:a strategy for constructing natural isomorphisms}. By Proposition \ref{prop:changing generators}, we are free to choose generators for each configuration $D_u$ using any isotopy $D_u^\c \to D_u$, and likewise for $D_u'$. For each $u\in\2^n$, let $S_u: D_u^\c \to D_u$ be the cobordism formed by a fixed choice of isotopy from $D_u^\c$ to $D_u$. There is also an obvious isotopy $R_u : D_u\to D'_u$, corresponding to the moves in the statement of the proposition. Since $D_u^\c = D_u'^\c$, we can choose generators for $D_u'$ using the cobordism $R_uS_u$. For $u,v\in \2^n$ with $u\geq_1 v$, let $d_{v,u}$ and $d'_{v,u}$ denote the maps assigned to the edge $v\to u$ in $[[D]]$ and $[[D']]$, respectively. 
 
Suppose we are in the situation (1). We have equivariant bijections $\psi_u: F_{\q}(u) \to F_{\q}'(u)$ for each $u$, given by $\psi_v(\q^k x) = \q^k R_u(x)$ (as usual, we do not distinguish between a cobordism and its induced map). Conditions \ref{nat iso condition1} and \ref{nat iso condition3} are satisfied by definition. For each $u,v\in \2^n$ with $u\geq_1 v$, we have
\[
R_u d_{v,u} = d'_{v,u} R_v.
\]
It follows that, for $\q^k y \in F_\q(u)$ and $\q^\l x\in F_\q(v)$, $\q^k y \text{ appears in } d_{v,u}(\q^\l x)$ if and only if $\q^k R_u(y)$ appears in $d'_{v,u}(\q^k R_v(x))$. Therefore condition \ref{nat iso condition2} is satisfied as well. Observe that the diagram \eqref{eq:simplified hexagon relation} commutes since we do not interfere with any potential ladybug matchings.

For case $(2)$, the maps $\psi_{u}$ need to be modified slightly in order to satisfy \ref{nat iso condition2}. We illustrate one case in detail. Suppose $D$ and $D'$ are as shown in \eqref{fig:isotoping crossing}. Assume also that the crossing shown is first in the ordering of crossings. 
\begin{equation}\label{fig:isotoping crossing} 
\vcenter{\hbox{
\includegraphics{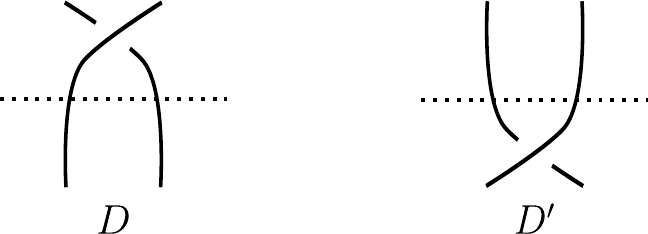}
}}
\end{equation}
Let $u\in \2^n$. If $u_1=0$, define $\psi_u : F_\q(u) \to F_\q'(u)$ as in case (1). If $u_1=1$, then define $\psi_u$ by 
\[
\psi_u(\q^k x) = \q^{k+1} R_u(x)
\]

We will now verify that condition \ref{nat iso condition2} holds for this choice of equivariant bijections $\{\psi_u\}_{u\in \2^n}$. Let $u,v\in \2^n$ with $u\geq_1 v$. If $u_1=v_1$, then the edge maps $d_{v,u}$ and $d'_{v,u}$ are induced by changing the smoothing at a crossing away from the one shown in \eqref{fig:isotoping crossing}. As in case (1) above, we have 
\[
R_u d_{v,u} = d'_{v,u} R_v,
\]
so condition \ref{nat iso condition2} holds. Suppose now that $u_1=1$ and $v_1=0$. Then 
\begin{equation}\label{eq:moving saddle across membrane}
d'_{v,u} R_v = \q R_u d_{v,u},
\end{equation}
where the factor of $\q$ comes from moving a saddle across the membrane (see Figure \ref{fig:trace moves}).  The situation is depicted in the (noncommutative!) diagram \eqref{fig:moving saddle across membrane}.

\begin{equation}\label{fig:moving saddle across membrane} 
\vcenter{\hbox{
\includegraphics{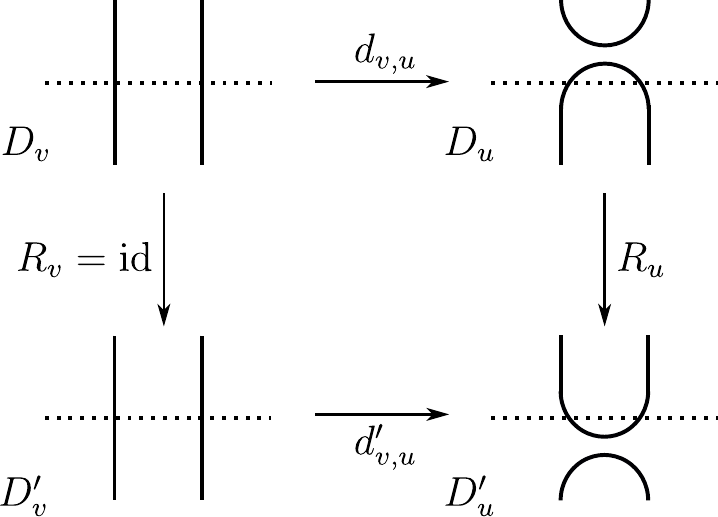}
}}
\end{equation}

Let $\q^k y\in F_\q(u)$ and $\q^\l x\in F_\q(v)$. It follows from \eqref{eq:moving saddle across membrane} that $\q^k y$ appears in $d_{v,u}(\q^\l x)$ if and only if $\q^{k+1} R_u(y)$ appears in $d'_{v,u}(\q^\l R_v(x))$. Therefore condition \ref{nat iso condition2} is satisfied. Again, the hexagon relation is satisfied because the $2$-morphisms do not interfere with the ladybug matching. 

\end{proof}

\subsection{The Cases $r=1,2$ and the Classical Annular Homotopy Type}\label{sec:case r=1,2}
Our proof of Theorem \ref{thm:the quantum annular Burnside functor} relies on $r>2$.  In this section we address the cases $r=1,2$ in that order.

Let $D$ be a diagram for an annular link with $n$ crossings. When $r=1$, the modified quantum annular chain complex 
\[
CKh_{\A_\q}(D) \otimes_\k \k/(\q-1)
\]
is just the classical annular chain complex $CKh_\A(D)$ (see the discussion preceding Lemma \ref{lem:saddles}).  
We sketch how to define the annular Burnside functor, denoted $F_1$, for the classical annular Khovanov chain complex below. An alternative construction, using 
Hochschild homology of Chen-Khovanov spectra for tangles, was recently introduced in \cite{LLS3}.

Let $F_{Kh} : \2^n \to \B$ denote the usual Khovanov Burnside functor (see \cite[Example 4.21]{LLS}, also \cite[Section 6]{LLS2}) where the ladybug matching is made with the \textit{left pair}. For $u\geq_1 v$, let
\[
d^\A_{v,u} : \F_{\A}(D_v) \to \F_\A(D_u)
\]
denote the classical annular differential, and let $d^{Kh}_{v,u}$ denote the usual Khovanov differential. We have that $d^{Kh}_{v,u} = d^\A_{v,u} + d'_{v,u}$ (see the discussion in Remark \ref{rem:getting annular from regular}). 

Recall from Section \ref{sec:classical annular Khovanov homology} that annular Khovanov generators may be taken to be the usual Khovanov generators (where the annular link diagram is considered as a planar diagram under the inclusion $\A\subset \R^2$), so set $F_1(u) = F_{Kh}(u)$ for each $u\in \2^n$. For $u\geq_1 v$, let $A^{Kh}_{u,v}$ denote the correspondence assigned by $F_{Kh}$ to the edge $\varphi_{u,v} :u\to v$. Define the correspondence $A^{\A}_{u,v}$ from $F_1(u)$ to $F_1(v)$ by
\[
A^{\A}_{u,v} := \{ (y,x) \in F_1(u) \times F_1(v) \mid y \text{ appears in } d^\A_{v,u}(x)\},
\]
with the obvious source and target maps, and set $F_1(\varphi_{u,v}) = A^{\A}_{u,v}$. Note that $A^{\A}_{u,v} \subset A^{Kh}_{u,v}$. 

Let $\phi^{Kh}_{u,v,v',w}$ be the $2$-morphism assigned by $F_{Kh}$ to the square face with vertices $u \geq_1 v,v' \geq_1 w$. One can check that $\phi^{Kh}_{u,v,v',w}$ restricts to
\[
\phi^\A_{u,v,v',w}: A^\A_{v,w} \times_{F_1(v)} A^\A_{u,v} \to A^\A_{v',w} \times_{F_1(v')} A^\A_{u,v'} .
\]
Taking $\phi^\A_{u,v,v',w}$ to be the $2$-morphisms assigned to square faces by $F_1$, the conditions of Lemma \ref{lem:hexagon relation} are satisfied as a consequence of the construction of $F_{Kh}$.

When $r=2$, Lemma \ref{lem:3d cube with ladybug configuration} and the ensuing analysis in case (c) of the proof of Theorem \ref{thm:the quantum annular Burnside functor} do not hold, since $\q^2=1$. Instead, we rely on the ladybug matching made with the left pair to define the $2$-morphism in case (c) of Theorem \ref{thm:the quantum annular Burnside functor}. 
Let us verify the hexagon relation, using the formulation in Remark \ref{rem:equivalent hexagon relation}. Start with an element $x=(\q^i a, \q^j b , \q^k c, \q^\l d)$ in the correspondence obtained as the composition of correspondences for three consecutive edge maps. Going around the six faces of the cube, the $2$-morphisms send $x$ to an element $x' = (\q^i a, \q^{j'} b' , \q^{k'} c', \q^\l d)$ in the same correspondence. It follows from classical annular case, where powers of $\q$ are disregarded, that labels on the circles in the generators $b$ and $c$ match the labels on the circles in the generators $b'$ and $c'$, respectively. Then $b=b'$, and since both $\q^j b$ and $\q^{j'} b' = \q^{j'} b$ appear in the image of $\q^i a$ under a saddle map, Lemma \ref{lem:saddles} implies that $\q^j =\q^{j'}$. Likewise, $ \q^k c= \q^{k'} c'$, and we conclude that the hexagon relation is satisfied. 

\section{From Burnside Functors to Stable Homotopy Types}\label{sec:from Burnside functors to stable homotopy types}

This section describes a general framework for obtaining a spectrum from a Burnside functor. This general construction is then applied to the case of the quantum annular Burnside functor, establishing the main result of the paper, Theorem \ref{equivariant thm}.
In more detail in Section \ref{sec:box maps} we recall box maps and their required properties for the non-equivariant case as in \cite[Section 5]{LLS}.  Then in Section \ref{sec:equivariant box maps} we discuss $G$-equivariant box maps via a slight generalization of the ideas established in \cite{SSS}, ensuring that the required properties are still satisfied. Sections \ref{sec:Burnside to refinements} and \ref{sec:refinements to realizations} describe how to use  box maps to pass from an equivariant Burnside functor $F$ to a $G$-CW complex realizing $F$, following \cite[Section 4]{SSS}, by taking the homotopy colimit (see \cite{Vogt}) of an appropriate diagram.  Finally in Section \ref{sec:defining the homotopy type} we apply this theory to the quantum annular Burnside functor of Section \ref{sec:the quantum annular Burnside functor for a link diagram} to define the  equivariant spectrum $\X_{\A_\q}^r(L)$ and check that it is well-defined, proving Theorem \ref{equivariant thm}.

We emphasize some differences and similarities between this paper and others appearing in the literature. There is no group action on the links considered in this paper, so our box maps and homotopy coherent refinements are different from those in \cite{BPS, Musyt, SZ}. Functors $\2^n \to \B_{\Z/2\Z}$ are considered in \cite{SSS}, and there the authors introduce actions of $\Z/2\Z$ and $\Z/2 \Z \times \Z/ 2 \Z$, which act internally on each box. We are  interested in an external $G$-action which permutes the boxes, so our work is different in this respect.

\subsection{Box Maps}\label{sec:box maps}
We begin by reviewing a key part of the non-equivariant case allowing us to set some notation following \cite[Section 5.1]{LLS}.

A $k$-dimensional box is $\prod_{i=1}^k [a_i, b_i] \subset \R^k$. For two $k$-dimensional boxes $B$ and $B'$, there is a canonical homeomorphism $B\rar{\sim} B'$, obtained by scaling and translating the ambient space $\R^k$. Fix an identification $S^k = [0,1]^k/\d([0,1]^k)$, so that for any $k$-dimensional box $B$, $B/\d B$ is canonically identified with $S^k$.

Suppose we have a correspondence $X\xleftarrow{s} A \xrightarrow{t} Y$. 
Pick disjoint $k$-dimensional boxes $\{B_x\}_{x\in X}$. Following \cite{LLS}, let 
\[
E(\{B_x\}, s)
\]
denote the space of all collections $\{B_a\}_{a\in A}$ of disjoint $k$-dimensional boxes such that $B_a\subset B_{s(a)}$. A point $e=\{B_a\} \in E(\{B_x\},s)$ determines a map 
\begin{equation}\label{eq:box map}
\Phi(e,A): \bigvee_{x\in X} S^k_x \to \bigvee_{y\in Y} S^k_y
\end{equation}
defined as follows. On each wedge summand, $\Phi(e,A)$ is the composition
\begin{equation}\label{eq:box map on each summand}
S^k_x = B_x/\d B_x \to B_x/ ( B_x\setminus \bigcup_{a\in s^{-1}(x)} B_a^{\text{int}} ) = \bigvee_{a\in s^{-1}(x)} B_a/\d B_a \xrightarrow{t} \bigvee_{y\in Y}S^k_y
\end{equation}
where the first map is a quotient and the last map sends the sphere $B_a/\d B_a$  to $B_{t(a)}/ \d B_{t(a)} = S^k_{t(a)}$ via the canonical homeomorphism $B_a\cong B_{t(a)}$. A map of this form is said to \emph{refine the correspondence } $X \xleftarrow{s} A \xrightarrow{t} Y$. 

Suppose we have correspondences $X \lar{s_A} A \rar{t_A} Y$ and $Y\lar{s_B} B \rar{t_B} Z$ with boxes $\{B_x\}_{x\in X}$ and $\{B_y\}_{y\in Y}$. Given $e\in E(\{B_x\}_{x\in X},s_A)$ and $e'\in E(\{B_y\}_{y\in Y},s_B)$, we can consider the preimage of the boxes in $e'$ under the map $\Phi(e,A)$. An important point in the proof of existence and uniqueness of spatial refinements is that $\Phi(e,A)^{-1}(e')$ is a collection of little boxes in $\{B_x\}_{x\in X}$ labelled by the composition $C:=B\times_Y A$; that is,
\[
\Phi(e,A)^{-1}(e') \in E(\{B_x\}_{x\in X}, s_C)
\]
where $s_C: C \to X$ is the source map of the composition. See Figure \ref{fig:pulling back boxes} for an explanation of this.

\begin{figure}
\centering
\includegraphics{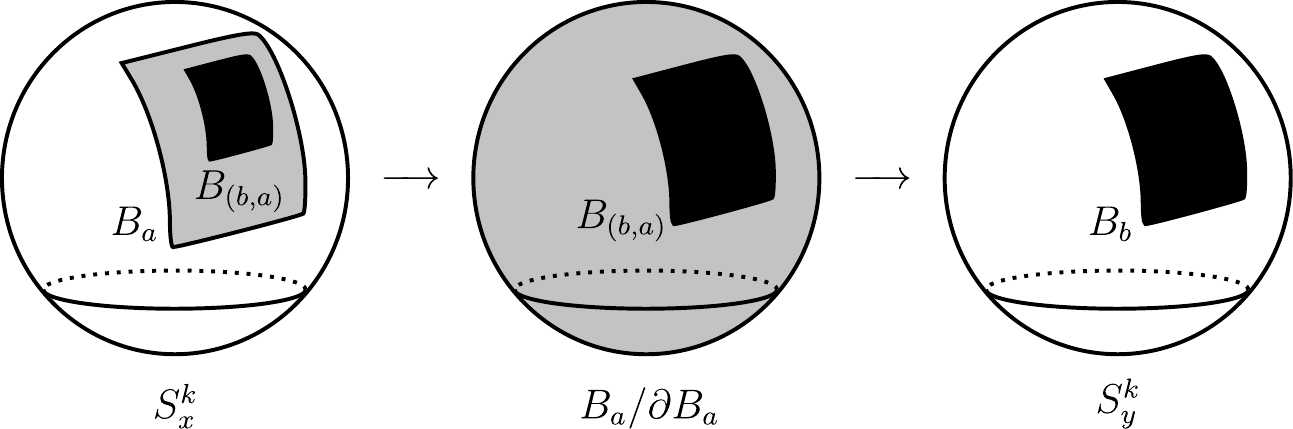}
\caption{
Fix two collections of subboxes $e\in E(\{B_x\}_{x\in X}, s_A)$ $e'\in E(\{B_y\}_{y\in Y}, s_B)$ and an element $(b,a)$ in the composition correspondence $C=B\times_Y A$.  Then by definition $e'$ gives us a box $B_b\subset S^k_y$ (in black on the right) where $y=s_B(b)=t_A(a)$ and $e$ gives us a box $B_a\subset S^k_x$ where $x=s_A(a)$ (in gray on the left).  The maps in the figure are those appearing in \eqref{eq:box map on each summand}, composing to give $\Phi(e,A)$.  If we pull back the box $B_b$ under this map (restricted to $S^k_x$), we obtain a smaller subbox $B_{(b,a)}\subset B_a \subset S^k_x$ for any such pair $(b,a)\in C$.  Taking the collection of such boxes and preimages together, we see that $\Phi(e,A)^{-1}(e') \in E(\{B_x\}_{x\in X}, s_C)$.
}
\label{fig:pulling back boxes}
\end{figure}

\subsection{Equivariant Box Maps}\label{sec:equivariant box maps}
Throughout this section, we will continue to use $G$ to denote a finite cyclic group, but all of the statements below generalize to more general finite groups acting freely on finite sets.

Suppose a correspondence $X \xleftarrow{s} A \xrightarrow{t} Y$ consists of $G$-sets and equivariant maps. Then in particular the spaces $\bigvee_{x\in X} S^k_x$ and $\bigvee_{y\in Y} S^k_y$ inherit a $G$-action via the canonical homeomorphism $B_x\cong B_{gx}$. In terms of the wedge summands, this $G$-action permutes the copies of $S^k$.  However,  to obtain a $G$-equivariant refinement $\bigvee_{x\in X} S^k_x \rightarrow \bigvee_{y\in Y} S^k_y$, a second condition needs to be imposed on the sub-boxes $\{B_a\} \in E(\{B_x\},s)$.

Let 
\[
E_G(\{B_x\},s)
\]
denote the subset of $E(\{B_x\},s)$ consisting of the little boxes $\{B_a\}$ satisfying the following property: for each $g\in G$, $x\in X$, and $a\in s^{-1}(x)$, the canonical homeomorphism $B_x \to B_{gx}$ restricts to a homeomorphism $B_a \to B_{ga}$. With this condition in place, the restricted homeomorphism $B_a \to B_{ga}$ is also the canonical one.  In particular, if we identify all boxes $B_{gx}$ canonically with $[0,1]^k$, then the subboxes $B_{ga}\subset B_{gx}$ for various $g$all have the same image in $[0,1]^k$.

\begin{lemma}\label{lem:box map is equivariant}
Let $X\lar{s} A \rar{t} B$ be an equivariant correspondence. If $e\in E_G(\{B_x\},s)$, then the induced box map $\Phi(e,A)$ is $G$-equivariant.  
\end{lemma}

\begin{proof}
Let $g\in G$ and $x\in X$. We need to verify commutativity of the following diagram
\[
\begin{tikzcd}[row sep = scriptsize]
S^k_x \arrow[r] \arrow[d, "g"'] & \displaystyle\bigvee_{a\in s^{-1}(x)} B_a/\d B_a \arrow[d, "{\cong}"] \arrow[r] & \displaystyle\bigvee_{y\in Y} S^k_y \arrow[d, "g"] \\
S^k_{gx} \arrow[r] & \displaystyle\bigvee_{a\in s^{-1}(x)} B_{ga}/\d B_{ga} \arrow[r] & \displaystyle\bigvee_{y\in Y} S^k_y
\end{tikzcd}
\]
where the horizontal maps are those of \eqref{eq:box map on each summand} and the vertical maps are induced by canonical homeomorphisms between boxes. The left square commutes since $e\in E_G(\{B_x\},s)$, and the right square commutes because all the maps there are induced by canonical homeomorphisms of boxes. 
\end{proof}

In the construction of
stable homotopy refinements, it is crucial that the space of little boxes is highly connected and that boxes pull back to boxes. The remainder of this subsection is devoted to verifying these properties for $E_G(\{B_x\},s)$. For the following statement, recall the notation for the quotient functor, introduced in the last paragraph of Section \ref{sec:the equivariant Burnside category}.

\begin{lemma}\label{lem:E_G vs E(X/G)}
Let $X \xleftarrow{s} A \xrightarrow{t} Y$ be a correspondence of $G$-sets and equivariant maps between them.  Then we have the following homeomorphism of spaces of sub-boxes:
\[
 E_G(\{B_x\}_{x\in X}, s) \cong E(\{B_{[x]}\}_{[x] \in X/G}, s/G).
\]
\end{lemma}

\begin{proof}
For a fixed $x\in X$, a sub-collection of little boxes $\{B_a\}_{a\in s^{-1}(x)}$ in $B_x$ determines the little boxes $\{B_c\}_{c\in s^{-1}(gx)}$ in $B_{gx}$ for all $g\in G$.  Explicitly, for each $a\in s^{-1}(x)$ and $g\in G$, $B_{ga} \subset B_{gx}$ is  the image of $B_a$ under the canonical homeomorphism $B_x\to B_{gx}$, and since $s$ is a $G$-equivariant map, $c\in s^{-1}(gx)$ if and only if $c=ga$ for some $a\in s^{-1}(x)$.

Thus a collection of equivariant boxes in $E_G(\{B_x\},s)$ is equivalent to a (non-equivariant)  choice of boxes in each $G$-orbit of $X$ for each $G$-orbit of $A$, which is the meaning of $E(\{B_{[x]}\}_{[x] \in X/G}, s/G)$.
\end{proof}

\begin{corollary}\label{cor:E_G connected}
$E_G(\{B_x\},s)$ is $(k-2)$-connected.
\end{corollary}
\begin{proof}
This follows from Lemma \ref{lem:E_G vs E(X/G)} together with the connectivity for spaces of (non-equivariant) sub-boxes shown in \cite[Lemma 5.18]{LLS}.
\end{proof}

\begin{lemma}\label{lem:equivariant pullback}
Let $X\lar{s_A} A \rar{t_A} Y$ and $Y \lar{s_B} B \rar{t_B} Z$ be equivariant correspondences. Let $\{B_x\}_{x\in X}$ and $\{B_y\}_{y\in Y}$ be collections of $k$-dimensional boxes, and let $e\in E_G(\{B_x\}, s_A)$, $e'\in E_G(\{B_y\}, s_B)$. Then $\Phi(e,A)^{-1}(e') \in E_G(\{B_x\}_{x\in X}, s_C)$.
\end{lemma}

\begin{proof}
Fix $g\in G$.  Let $B_{(b,a)} \in \Phi(e,A)^{-1}(e')$ and set $y=s_B(b)$ and $x=s_A(a)$ as in Figure \ref{fig:pulling back boxes}. We need to show that the canonical homeomorphism $B_x\to B_{gx}$ sends $B_{(b,a)}$ to $B_{(gb,ga)}$. By construction, $\Phi(e,A)(B_{(b,a)}) = B_b$. Since $e\in E_G(\{B_x\}_{x\in X},s_A)$, we also know that $g B_b = B_{gb}$. Recall from Lemma \ref{lem:box map is equivariant} that $\Phi(e,A)$ is $G$-equivariant. These facts yield
\[
\hskip1em B_{gb} = g \left(\Phi(e,A)(B_{(b,a)}) \right) = \Phi(e,A)\left( g B_{(b,a)}\right).
\]
Since $e'\in E_G(\{B_y\}_{y\in Y},s_B)$ and $B_{(b,a)} \subset B_a$, we have that $g B_{(b,a)} \subset g B_a = B_{ga}$. Then $g B_{(b,a)}$ is a box contained in $B_{ga}$, which is sent to $B_{gb}$ by $\Phi(e,A)$. It follows that $g B_{(b,a)} = B_{(gb,ga)}$, which completes the proof. 

\end{proof}

\subsection{From Burnside Functors to Spatial Refinements}
\label{sec:Burnside to refinements}
In this section we describe how to use box maps to transform a Burnside functor $F:\Cscr\rightarrow \B_G$ into a certain homotopy coherent diagram of spaces, called a spatial refinement of $F$, in an equivariant manner.

Let $\Top_*^G$ denote the category of based $G$-spaces. We refer the reader to  \cite[Section 4.2]{SSS} for the definition and discussion of homotopy coherent diagrams and homotopy colimits in the equivariant setting, 
 parallel to the non-equivariant treatment in \cite[Section 4.2]{LLS}. The following is an equivariant analogue of \cite[Definition 5.21]{LLS} and 
  \cite[Proposition 5.22]{LLS} (and extends  \cite[Definition 4.11]{SSS} and \cite[Proposition 4.12]{SSS} from $\Z/2\Z$ to general finite  groups $G$).

\begin{definition}\label{def:spatial refinement}
Fix a small category $\Cscr$ and a strictly unitary lax 2-functor $F:\Cscr \to \B_G$. A \emph{$k$-dimensional spatial refinement} of $F$ is a homotopy coherent diagram $\til{F}_k:\Cscr \to \emph{Top}^G_*$ satisfying
\begin{enumerate}[label= (\arabic*)]
\item\label{it:vertices are spheres} For any $u\in \Cscr$, $\til{F}_k(u) = \bigvee_{x\in F(u)} S^k$.\\
\item\label{it:edges are box maps} For any sequence $u_0 \xrightarrow{f_1} \cdots \xrightarrow{f_m} u_m$ of morphisms in $\Cscr$ and $t\in I^{m-1}$, the map
\[
\til{F}_k(f_m,\ldots, f_1)(t) : \bigvee_{x\in F(u_0)} S^k \to \bigvee_{y\in F(u_m)} S^k
\]
is a box map which refines the correspondence $F(f_m \circ \cdots \circ f_1)$. Note that, by assumption, the diagram lands in $\emph{Top}^G_*$, so that the map in $(2)$ is also $G$-equivariant.
\end{enumerate}
\end{definition}

\begin{proposition}\label{prop:equivariant spatial refinements}
Let $\Cscr$ be a small category in which every sequence of composable non-identity morphisms has length at most $n$, and let $F:\Cscr \to \B_G$ be a strictly unitary lax 2-functor.\\
\begin{enumerate}[label= \emph{(\arabic*)}]
\item If $k\geq n$, there is a $k$-dimensional spatial refinement of $F$.
\item If $k\geq n+1$, then any two $k$-dimensional spatial refinements of $F$ are equivariantly weakly equivalent (see \cite[Section 4.2]{SSS} for a detailed discussion of weak equivalences between spatial refinements).
\item If $\til{F}_k$ is a $k$-dimensional spatial refinement of $F$, then the (reduced)  suspension of each $\til{F}_k(u)$ and of each $\til{F}_k(f_m,\ldots, f_1)(t)$ gives a $(k+1)$-dimensional spatial refinement of $F$.
\end{enumerate}
\end{proposition}
\begin{proof}
This is completely parallel to the proofs of \cite[Proposition 5.22]{LLS} and \cite[Proposition 4.12]{SSS}. The modification is that the maps 
\[
e_{f_m,\ldots, f_1} : I^{m-1} \to E(\{B_x\}, s_{f_m\circ \cdots \circ f_1}) 
\]
should land in $E_G(\{B_x\}, s_{f_m\circ \cdots \circ f_1})$, which is still highly connected by Corollary \ref{cor:E_G connected}. We have also verified that equivariant boxes pull back to equivariant boxes in Lemma \ref{lem:equivariant pullback}.  Note that suspension respects the group action, which  permutes spheres.
\end{proof}

\subsection{From Spatial Refinements to Realizations}
\label{sec:refinements to realizations}

In this section we discuss how to pass from a $G$-equivariant homotopy coherent spatial refinement $\widetilde{F}_k$ to a realization $\LR{F_k}\in \Top_*^G$.  We also recall some of the cellular properties of such $\LR{F_k}$ and the induced maps between them.

Following \cite[Definition 5.1]{LLS}, we begin by defining a slight enlargement of the cube category $\2^n$, denoted $\2^n_+$. The objects of $\2^n_+$ are $\text{ob}(\2^n) \cup \{*\}$; that is, $\2^n_+$ has an extra object added. Set $\Hom_{\2^n_+}(u,v) = \Hom_{\2^n}(u,v)$ if $u,v\in \2^n$. Otherwise, for $u\in \2^n\setminus \{0\}$, set $\Hom_{\2^n_+}(u,*)$ to consist of a single morphism. Finally, set $\Hom_{\2^n_+}(0,*) = \Hom_{\2^n_+}(*,0) = \Hom_{\2^n_+}(*,u) = \varnothing$. 

Let $F:\2^n \to \B_G$ be a Burnside functor, and let $\til{F}_k : \2^n \to \Top_*^G$ be a $k$-dimensional spatial refinement of $F$.
Extend $\til{F}_k$ to a homotopy coherent diagram $\til{F}_k^+ : \2^n_+ \to \Top_*^G$ by setting $\til{F}_k^+(*)$ to be a single point. 
Following \cite[Definition 4.9]{SSS},
define the space
\begin{equation}\label{eq:realization}
\LR{F}_k := \hocolim\,  \til{F}_k^+,
\end{equation}
called a \textit{realization} of $F$. Since the homotopy coherent diagram $\til{F}_k$ takes values in $\Top_*^G$, the space $\LR{F}_k$ is again a based $G$-space.  

There is a cell structure on $\LR{F}_k$, called the \textit{coarse cell structure} in Section 4.4 of \cite{SSS}, with the cells of $\LR{F}_k$ in bijection with $\coprod_{u\in \2^n} F(u)$. This cell structure is described in \cite[Section 6]{LLS}. 

\begin{lemma}\label{lem:cell structure}
With the above notation,
\begin{enumerate}[label= \emph{(\arabic*)}]
\item The $G$-action on $\LR{F}_k$ is cellular, and the bijection 
\[
\{ \emph{Cells of }\LR{F}_k\} \longleftrightarrow \coprod_{u\in \2^n} F(u)
\]
is $G$-equivariant. 

\item The $G$-action on $\LR{F}_k$ is free away from the basepoint.

\item The weak equivalences of Proposition \ref{prop:equivariant spatial refinements} induce equivariant homotopy equivalences on realizations.  Thus $\LR{F_k}$ is well-defined and $\Sigma \LR{F}_k \simeq \LR{F}_{k+1}$.

\end{enumerate}
\end{lemma}

\begin{proof}
Statement (1) follows from inspecting the cell decomposition in \cite[Proposition 6.1]{LLS} (see also the discussion in \cite[Section 4.4]{SSS}). Statement (2) follows from (1) and the fact that $G$ acts freely on the set $\coprod_{u\in \2^n} F(u)$. 

For (3), we first note that (1) and (2) imply our realizations are $G$-CW complexes (the action of $G$ is cellular and its fixed set is the basepoint, which is trivially a subcomplex).  Then as described in \cite[Section 4.5]{SSS}, equivariant weak equivalences of homotopy coherent diagrams induce equivariant weak equivalences on their homotopy colimits (which can be taken to be cellular), which induce equivariant homotopy equivalences for large enough $k$ by the equivariant Whitehead theorem.
\end{proof}

Now recall from Section \ref{sec:totalizations of Burnside functors} that the totalization $\Tot(F)$ of an equivariant Burnside functor $F :\2^n \to \B_G$ is a complex of $\Z[G]$-modules, and that any natural transformation $\eta : F_1 \to F_0$ between two such functors induces a chain map over $\Z[G]$, $\Tot(\eta): \Tot(F_1) \to \Tot(F_0)$.  Also recall \cite[Lemma 4.15]{SSS}, which in particular says that if $F_1, F_0 : \2^n \to \B$ are Burnside functors and $\eta :F_1\to F_0$ is a natural transformation, then there is an induced map $\eta : \LR{F_1}_k \to \LR{F_0}_k$ for any realization.  We record \cite[Proposition 4.16]{SSS} below relating these notions, where the notation $[k]$ denotes a homological shift by $k$ (that is, $C^i[k] := C^{i-k}$).
\begin{proposition}\emph{(\cite[Proposition 4.16]{SSS})}\label{prop:identifying the cellular chain complex} 
If $F : \2^n \to \B$, then its reduced shifted cellular chain complex $\til{C}^{cell}_*(\LR{F}_k)[-k]$ is isomorphic to $\emph{Tot}(F)$, with the cells mapping to the corresponding generators. If $\eta: F_1\to F_0$ is a natural transformation of Burnside functors, then the map $\LR{F_1}_k \to \LR{F_0}_k$ is cellular, and the induced cellular chain map agrees with $\emph{\Tot}(\eta)$. 
\end{proposition}
When $F$ takes values in $\B_G$, the above discussion shows that $\til{C}_{*}^{cell}(\LR{F}_k)[-k]$ is a $\Z[G]$-module, and the isomorphism
\[
\til{C}_{*}^{cell}(\LR{F}_k)[-k] \cong \Tot(F)
\]
of Proposition \ref{prop:identifying the cellular chain complex} is an isomorphism of complexes of $\Z[G]$-modules (see \cite[Proposition 4.23]{SSS}).  Likewise, a natural transformation $\eta$ between two equivariant Burnside functors induces an equivariant cellular map on their realizations that recovers $\Tot(\eta)$ on the chain complex level.
\begin{lemma}\emph{(\cite[Lemma 4.17]{SSS})}\label{lem:equivariant homotopy equivalence} Let $F_1, F_0: \2^n \to \B_G$ be equivariant Burnside functors, and let $\LR{F_1}_k$, $\LR{F_0}_k$ be $k$-dimensional spatial refinements. If $\eta: F_1\to F_0$ is a natural transformation such that the induced map $\emph{\Tot}(\eta) : \emph{\Tot}(F_1) \to \emph{\Tot}(F_0)$ is a chain homotopy equivalence, then the induced map $\eta :\LR{F_1}_k \to \LR{F_0}_k$ is an equivariant homotopy equivalence. 
\end{lemma}

\begin{proof}
The argument is the same as in \cite[Lemma 4.17]{SSS}. The previous discussion shows that the induced map on spaces $\eta :\LR{F_1}_k \to \LR{F_0}_k$ is $G$-equivariant and cellular. We may take $k$ big enough, so that the realizations are simply connected. Note that the $G$-action on the realization is free away from the basepoint. Since $\Tot(\eta)$ is a $\Z[G]$-linear isomorphism on homology, the equivariant Whitehead theorem implies that $\eta :\LR{F_1}_k \to \LR{F_0}_k$ is an equivariant homotopy equivalence. 
\end{proof}

\begin{remark}
Although the discussion in \cite[Section 4.4]{SSS} actually makes use of spatial refinements out of the category $(\2_+)^n$ rather than $\2^n_+$, the cellular structures described on the realizations there induce equivalent cellular structures on our realizations via a simple quotient within each cell.  This is implicit in their reference to \cite[Section 6]{LLS} which builds the coarse structure for realizations using $\2^n_+$ as we are here; the alternative use of $(\2_+)^n$ in \cite{LLS} is denoted by $\2_\dagger^n$ there.
\end{remark}

\subsection{The Quantum Annular Spectrum $\X_{\A_\q}^r(L)$}\label{sec:defining the homotopy type}

Finally, we turn to defining our quantum annular spectrum and proving Theorem \ref{equivariant thm}.  Let $D$ be a diagram for an annular link with $n_-$ negative crossings, and fix $G=\Z/r\Z$ for some $r\geq 2$. Let $F_\q : \2^n\to \B_G$ be a quantum annular Burnside functor for $D$ as provided by Theorem \ref{thm:the quantum annular Burnside functor}.  Let $\til{F}_{\q,k}^+ : \2^n_+ \to \Top_*^G$ be a $k$-dimensional spatial refinement extended to the enlarged cube category $\2^n_+$, as described in Section \ref{sec:Burnside to refinements} and the beginning of Section \ref{sec:refinements to realizations}.  Let $\LR{F_\q}_k$ be the realization, as defined in \eqref{eq:realization}. 
\begin{definition}\label{def:quantum annular Khovanov spectrum}
Define the \emph{quantum annular Khovanov spectrum} $\X_{\A_\q}^r(D)$ to be the suspension spectrum of $\LR{F_\q}_k$, desuspended $k+n_-$ times; that is,
\[
\X^r_{\A_\q}(D):=\Sigma^{-k-n_-} \left(\Sigma^\infty \LR{F_\q}_k \right).
\]
\end{definition}
By dualizing the isomorphism in Proposition \ref{prop:identifying the cellular chain complex}, we obtain the following isomorphism of $\Z[G]$-modules:
\begin{equation}\label{eq:C(X)=KC}
C^*(\X_{\A_\q}^{r}(D)) \cong CKh_{\A_\q}^{*}(D)\otimes_\k \k_r.
\end{equation}
\begin{theorem}\label{thm:X(D) well-defined}
For a fixed annular link diagram $D$ and $r\geq 2$, the naive $G$-spectrum $\X_{\A_\q}^r(D)$ is well-defined; that is, different choices during the construction yield equivariantly stably homotopy equivalent spectra.
\end{theorem}
\begin{proof}
The construction of $\X^r_{\A_\q}(D)$ requires a choice of generators at each vertex of the cube to build $F_\q$, together with a choice of spatial refinement of $F_\q$.  Any two choices of generators give naturally isomorphic Burnside functors (see Proposition \ref{prop:changing generators}), which in turn yield equivariantly stably homotopy equivalent spectra by Lemma \ref{lem:equivariant homotopy equivalence}.  Meanwhile, any two spatial refinements yield homotopy equivalent realizations  by Lemma \ref{lem:cell structure} (so long as $k$ is large enough), and thus equivariantly stably homotopy equivalent spectra.
\end{proof}

Finally, we address the independence of choice of diagram with the following theorem.

\begin{theorem}\label{thm:invariance of htpy type}
Let $D$ and $D'$ be two annular link diagrams for the same annular link $L\subset \A\times I$.  Then $\X_{\A_\q}^r(D)$ is equivariantly stably homotopy equivalent to $\X_{\A_\q}^r(D')$, and as such we may use the notation $\X_{\A_\q}^r(L)$ to denote the quantum annular $G$-equivariant stable homotopy type of $L$.
\end{theorem}
\begin{proof}
The diagrams $D$ and $D'$ are connected by a series of moves corresponding either to the annular isotopies of Section \ref{sec:isotoping the link diagram} or Reidemeister moves.  Isotopies were shown to induce natural isomorphisms of Burnside functors (Proposition \ref{prop:isotoping the link diagram}), which therefore induce equivariant stable equivalences by Lemma \ref{lem:equivariant homotopy equivalence}.

With such planar equivalences available, we can assume that any Reidemeister move takes place in a disk disjoint from the seam $\mu$.  Such moves then induce homotopy equivalences in precisely the same fashion as they do for the classical Khovanov homotopy type \cite[Section 6]{LS}.  That is to say, any Reidemeister move corresponds to finding subfaces of the relevant cube corresponding to acyclic subcomplexes (or quotient complexes) in the totalization.  (These subfaces are referred to as upwards- and downwards-closed subcategories in the original treatment of \cite{LS}.)  The complements of these acyclic faces can then be included into the large cube; the inclusion induces a map on stable homotopy types that gives an isomorphism on homology, and therefore is a stable equivalence by Whitehead's theorem.  All of this continues to hold in the equivariant setting so long as the group $G$ is finite  - face inclusions induce equivariant maps by definition.
\end{proof}

As in \cite[Section 4.7]{LLS} and \cite[Section 3.9]{SSS}, we also have a splitting of the functor $F_\q$ into a coproduct over the two gradings $\qdeg$ and $\adeg$ (see \eqref{qdeg:simul1} and \eqref{adeg:simul2}).  Thus the spectrum also splits as a wedge sum
\[
\X^r_{\A_\q}(D) = \bigvee_{j,k} \X_{\A_\q}^{r;j,k}(D)
\]
where $j$ corresponds to $\qdeg$ and $k$ corresponds to $\adeg$.  As in \cite[Theorem 1.1]{LS} (see also \cite[Theorem 1]{LLS2}), Theorems \ref{thm:X(D) well-defined} and \ref{thm:invariance of htpy type} respect this splitting, as does Equation \eqref{eq:C(X)=KC} as indicated below:
\[
C^*(\X_{\A_\q}^{r;j,k}(D)) \cong CKh_{\A_\q}^{*,j,k}(D)\otimes_\k \k_r.
\]

We end this section with some remarks about the spectrum $\X_{\A_\q}^r(L)$. Although the construction above was aimed at building a naive $G$-spectrum, one could also construct a genuine $G$-spectrum in a similar manner by applying the functor $\Sigma_G^\infty$, rather than $\Sigma^\infty$, to the realization $\LR{F_\q}_k$.  This functor produces genuine $G$-spectra using smash products with all $G$-representation spheres, rather than using only spheres with trivial $G$-action as $\Sigma^\infty$ does.

We also note that, in general, $G$ acts on $\X_{\A_\q}^r(L)$ in a nontrivial way. Precisely, $\X_{\A_\q}^r(L)$ does not decompose into a wedge product which is simply permuted by $G$. This is already evident on homology due to the calculation in \cite[Proposition 6.9]{BPW} for the annular closure of $(2,n)$ torus links. For an appropriate $n$, homological degree $i$, and $q$-degree $j$, the quantum annular homology is of the form $Kh_{\A_\q}^{i,j}(T_{2,n}) = \k_r/(\q^2+1)$.

\section{Maps on spectra induced by annular link cobordisms} \label{cobordisms section}
In \cite[Section 3]{LS2} the authors show that an embedded cobordism $W\subset S^3\times[0,1]$ between two links $L_0\subset S^3\times\{0\}$ and $L_1\subset S^3\times \{1\}$ induces a map on spectra 
\[\varphi_W:\X(L_1)\rightarrow\X(L_0)\]
such that the induced map on cohomology
\[\varphi_W^*: H^*(\X(L_0))\rightarrow H^*(\X(L_1))\]
recovers the corresponding link cobordism maps $W_*$ in Khovanov homology as studied in \cite{Khovanov, Jac, BN2}. 
The map $\varphi_W$ is constructed by first decomposing $W$ into elementary cobordisms whose planar projections correspond to either Reidemeister moves or Morse moves (births/cups, saddles, or deaths/caps) and assigning maps to each elementary cobordism.  A generically embedded $W$ determines such a decomposition. It is conjectured in \cite{LS2} that isotopic cobordisms induce stably homotopic maps, but this conjecture has not yet been verified.

Now consider a cobordism $W\subset (\A\times I)\times[0,1]$ between two annular links $L_0\subset \A\times I \times\{0\}$ and $L_1\subset \A\times I \times \{1\}$ that is transverse to the $3$-dimensional membrane $\mu\times I \times [0,1]$.  In \cite{BPW} the authors show that there is an induced map 
\[
W_* : Kh_{\A_\q}(L_0) \to Kh_{\A_\q}(L_1)
\]
on the quantum annular homology, defined using the general theory of twisted horizontal traces and shadows established in \cite[Section 3]{BPW}, as well as the functoriality of Chen-Khovanov bimodules under tangle cobordisms (\cite[Proposition 6]{CK}).
The map induced by $W$ on the chain complex level can be determined by the sequence of maps given in \cite[Equation (7.2)]{BPW}.  In the Appendix we compute $W_*$ explicitly for certain elementary cobordisms; this computation is used in the proof of Theorem \ref{thm:cobordism maps on spectra} below.

We note that an isotopy of $W$ can alter $W_*$ by a sign change and a power of $\q$ by \cite[Theorem B]{BPW}; the sign ambiguity is inherited from the similar statement in the usual Khovanov homology (see \cite{Jac}, \cite{BN2}), while the power of $\q$ comes from the ability to isotope parts of $W$ through the membrane. If one instead demands that isotopies fix the membrane, then $W_*$ is well-defined up to a sign.

\begin{theorem}\label{thm:cobordism maps on spectra}
Fix $r\in\N$.  A generically embedded cobordism $W\subset \A\times I\times [0,1]$ between two annular links $L_0$ and $L_1$ induces a map
\[\varphi^r_W:\X^r_{\A_\q}(L_1)\rightarrow \X^r_{\A_\q}(L_0)\]
whose induced map on cohomology
\[(\varphi^r_W)^*:H^*(\X^r_{\A_\q}(L_0))\rightarrow H^*(\X^r_{\A_\q}(L_1))\]
equals the map $W_*$ on quantum annular Khovanov homology over the ring $\k_r$.
\end{theorem}
\begin{proof}
A generic annular cobordism determines a sequence of elementary cobordisms (called \textit{elementary string interactions} in \cite{BN2, GLW}), which are either Reidemeister moves or Morse moves.  When accounting for the presence of the membrane $\mu\times I \times [0,1]$, there are certain additional elementary isotopies of a link through the seam which must be considered: we have the $P$ and $N$ moves of Figure \ref{fig:PN moves}, as well as pushing  a crossing through the seam, as in \eqref{fig:isotoping crossing}.  Meanwhile, the genericity of $W$ here implies that all Reidemeister moves and Morse moves occur away from the seam.

For all elementary isotopies of the link, we already have stable homotopy equivalences via Theorem \ref{thm:invariance of htpy type}.  As in the case for $S^3$, we wish to use the inverses of these maps. 
It is clear that such inverses induce the maps $W_\bullet$ described in the Appendix, and so according to Lemma \ref{lem:elementary cobordisms}, any such map will recover its corresponding $W_*$ up to some power of $\q$.  We may thus compose  any such stable homotopy equivalence with some iterate of the group action on $\X_{\A_\q}^r(L_0)$ to define $\varphi_W^r$ which induces precisely  $W_*$.

Meanwhile, Morse moves induce natural transformations of Burnside functors in the same manner as they do in $S^3$: births induce correspondences which involve a $w_+$ label on the new (trivial, disjoint from the seam) circle; deaths induce correspondences which place a $w_-$ label on the dying (trivial, disjoint from the seam) circle; and saddles utilize the higher dimensional cube which would be built if the diagram had a crossing placed at the point of the saddle.  Carrying out these constructions equivariantly does not present any new issues, leading to constructions of maps $\varphi_W^r$ via Proposition \ref{prop:identifying the cellular chain complex} which again induce  the maps $W_\bullet$ of the Appendix.  For these moves (Type II in Lemma \ref{lem:elementary cobordisms}) we have $W_*=W_\bullet$, concluding the proof of the Theorem.
\end{proof}

\begin{remark}
We stress that Theorem \ref{thm:cobordism maps on spectra} assigns a map to cobordisms $W$ that come with a particular decomposition into a sequence of elementary cobordisms in the thickened annulus with membrane.
\end{remark}

\begin{remark}\label{rmk:isotopic cobs}
Let $W,W'$ be two isotopic annular link cobordisms with corresponding maps on spectra $\varphi^r_W,\varphi^r_{W'}$ via Theorem \ref{thm:cobordism maps on spectra}.  Let $\tau_\q$ denote the map on spectra determined by the action of the distinguished generator of the group $G=\Z/r\Z$.  Then it is reasonable to conjecture that there exists some $m\in\Z$ such that the maps $\varphi^r_W$ and $\tau_\q^m \circ \varphi_{W'}^r$ are stably homotopic.

This is based on the similar conjecture in \cite{LS2} for cobordisms in $S^3$.  Notice that the composition with the map $\tau_\q^m$ recovers the ambiguity in the power of $\q$ which is known to exist for the corresponding maps on the quantum annular homology.
\end{remark}

 Let $\mathbb{S}$ denote the sphere spectrum and define
\begin{equation}\label{eq:spectrum for empty link}
\X^r_{\A_\q}(\varnothing):=\bigvee_G \mathbb{S}
\end{equation}
where $G$ acts by permuting the  wedge summands as usual.  Then we have the following corollary for closed surfaces in $\A\times D^2$ formed by sweeping out a link in the $S^1$ direction.

\begin{corollary}\label{cor:cobordism map for sweep recovers Jones}
Let $L$ be a link in the 3-ball $B^3$, and  consider the surface $\h{W}=S^1\times L$ in $\A\times D^2\cong S^1\times B^3$. Let  $W$ denote a copy of $\h{W}$ perturbed to be generic, viewed as a cobordism from $\varnothing$ to itself.  
Then the map
\[\varphi_W^r:\X_{\A_\q}(\varnothing) \longrightarrow \X_{\A_\q}(\varnothing)\]
induces the map on quantum annular homology
\[(\varphi_W^r)^*:Kh_{\A_\q}(\varnothing)=\k_r \longrightarrow \k_r = Kh_{\A_\q}(\varnothing)\]
which is given by multiplication by the Jones polynomial of $L$, considered as an element of $\k_r$, up to a sign and a power of $\q$ (where the standard basis of the groups $\k_r\cong\bigoplus_G \Z$ is written as $\{1,\q,\ldots \q^{r-1}\}$).
\end{corollary}

\begin{remark}
In order to make sense of assigning a wedge of sphere spectra to the empty diagram $\varnothing$ in terms of Burnside functors, we assign to $\varnothing$ the functor $F_\q : \2^0=\{*\} \to \B_G$ defined by setting $F_\q(*) := G \times \{1\}$, where $1\in \F_{\A_\q}(\varnothing) \otimes_\k \k_r = \k_r$ is the chosen generator. The spatial refinement is then a wedge of $r$ spheres with no box maps, and in the homotopy colimit there is nothing to identify except the basepoint of $\bigvee_G S^k$ with the new basepoint in $\2^0_+$. Thus the final space is just a wedge of $r$ spheres with the natural action, desuspended $k$ times, and its reduced cohomology is isomorphic to $\k_r$ as a $\k_r$-module.
\end{remark}

\begin{proof}[Proof of Corollary \ref{cor:cobordism map for sweep recovers Jones}]
Theorem \ref{thm:cobordism maps on spectra} implies that $(\varphi_W^r)^* = W_*$, and by \cite[Proposition 6.8]{BPW}, we also have $W_* = \pm \q^k \h{W}_*$ for some $k\in \Z$. Finally, \cite[Theorem D]{BPW} states that $\h{W}_*$ is multiplication by the Jones polynomial of $L$.  In fact, the more general statement about Lefschetz traces in \cite[Theorem D]{BPW}  also applies here.
\end{proof}

\section{Taking the Quotient}\label{sec:taking the quotient}

Our goal in this section is to prove Theorem \ref{quotient theorem}, stating that the quotient $\X_{\A_\q}^r(D)/G$ of the quantum annular homotopy type is stably homotopy equivalent to the classical annular homotopy type $\X_\A(D)$.  This is accomplished in two stages.

First we show that the quotient of $F_q$ is naturally isomorphic to the classical annular Burnside functor $F_1$ defined in Section \ref{sec:case r=1,2}.  This will follow from Corollary \ref{cor:1+q^2 splitting makes left choice} which establishes  that the matching forced by powers of $\q$ in the quantum theory agrees with the ladybug matching made with the left pair in the classical theory.

Next we show that taking the quotient of a spatial refinement for $F_q$ yields a spatial refinement for $F_1$. The result will then follow from the property that homotopy colimits commute with taking quotients. 

Let $D$ be a diagram for an annular link with $n$ crossings. Let $F_1 : \2^n\to \B$ denote the classical annular Khovanov Burnside functor, where the ladybug matching is made with the left choice. Recall the quotient functor $(-)/G : \B_G \to \B$ from Section \ref{sec:the equivariant Burnside category}. We can compose $F_\q: \2^n \to \B_G$ with the quotient functor to obtain a Burnside functor $F_\q/G : \2^n \to \B$. We will also use $(-)/G: \Top_*^G \to \Top_*$ to denote the quotient functor on $G$-spaces. It will be clear from context which functor is used.

\begin{proposition}\label{prop:quotient of QABF}
The functors $F_{\q}/G :\2^n \to \B$ and $F_1: \2^n\to \B$ are naturally isomorphic. 
\end{proposition}
\begin{proof}
We will use the strategy of Section \ref{sec:a strategy for constructing natural isomorphisms} to build a natural isomorphism $\eta: F_\q/G \to F_1$. For $u\in \2^n$, there is a natural identification 
\[
F_\q(u)/G = (G\times \Gamma(u))/G \cong \Gamma(u) = F_1(u).
\]
Let $\psi_u : F_\q(u)/G \to F_1(u)$ be the above bijection. For $u\geq_1 v$, let $A_{u,v}$ denote the correspondence assigned by $F_q$ to the edge $u\to v$, and let $A'_{u,v}$ denote the correspondence assigned by $F_1$. There is an injection 
\[
A_{u,v}/G \hookrightarrow F_\q(u)/G \times F_\q(v)/G
\]
given by $[\q^k y, \q^\l x] \mapsto ([x],[y])$. We will identify $A_{u,v}/G$ with its image in $F_\q(u)/G \times F_\q(v)/G$. By Lemma \ref{lem:saddles}, the map 
\[
\psi_u \times \psi_v :F_\q(u)/G \times F_\q(v)/G \to F_1(u) \times F_1(v)
\]
restricts to a bijection
\[
\psi_u \times \psi_v :A_{u,v}/G \to A'_{u,v}.
\]
Thus conditions \ref{nat iso condition1}, \ref{nat iso condition2}, and \ref{nat iso condition3} of Section \ref{sec:a strategy for constructing natural isomorphisms} are satisfied. It remains to verify that the diagram \eqref{eq:simplified hexagon relation} commutes. Recall that we have used the ladybug matching made with the left pair for both $F_\q$ and $F_1$. Then commutativity of \eqref{eq:simplified hexagon relation} follows from Corollary \ref{cor:1+q^2 splitting makes left choice}. 
\end{proof}

Note that any homotopy coherent diagram in $\Top^G_*$ can be composed with the quotient functor $(-)/G:\Top^G_*\rightarrow\Top_*$ to give a homotopy coherent diagram in $\Top_*$ as in \cite[Section 4.2]{SSS}.

\begin{proposition}\label{prop:quotient of homotopy coherent diagram}
Let $\til{F_{\q}} :\2^n \to \emph{\Top}_*^G$ be a $d$-dimensional spatial refinement of $F_{\q}$. Then the homotopy coherent diagram $\til{F_{\q}}/G$, obtained by applying $(-)/G$ to each $\til{F_{\q}}(v)$ and each $\til{F_{\q}}(f_m,\ldots, f_1)$, is a $d$-dimensional spatial refinement of $F_{\q}/G$. 
\end{proposition}

\begin{proof}

On a vertex $u\in \2^n$, since $F_{\q}(u)=G\times F_1(u)$, it is again clear that the quotient
\[
(\til{F_{\q}}/G)(u) = \til{F_{\q}}(u)/G = \bigg(\bigvee_{\q^k x\in F_{\q}(u)} S^d \bigg)/G
\]
is canonically identified with 
\[
\til{F_{\q}/G}(u)=\bigvee_{x\in F_{\q}(u)/G} S^d.
\]

The key point is to recognize that, for any correspondence $A=F_{\q}(f)$ assigned to some morphism $f:u\rightarrow v$ in $\2^n$ (with source and target maps $s$ and $t$, respectively), the quotient of a box map refining $A$ is itself a box map which refines the quotient of $A$.  That is to say, given a choice of equivariant little boxes
\[
e = \{B_a\}_{a\in A} \in E_G(\{B_{\q^k x}\}_{\q^k x\in F_{\q}(u)}, s)
\]
which induces a map
\[
\bigg(\bigvee_{\q^k x\in F_{\q}(u)} S^d \bigg) \xrightarrow{\Phi(e,A)} \bigg(\bigvee_{\q^\l y\in F_{\q}(v)} S^d \bigg),
\]
the image of the boxes $e$ in the quotient gives a new collection of little boxes
\[
e/G := \{B_a/G\}_{a\in A} \cong \{B_{[a]}\}_{[a]\in A/G} \in E(\{B_x\}_{x\in F_{\q}(u)/G}, s/G)
\]
such that the following diagram commutes:
\[
\begin{tikzcd}[column sep= large]
\displaystyle\bigg(\bigvee_{\q^k x\in F_{\q}(u)} S^d \bigg)/G
\arrow[r, "{\Phi(e,A)/G}"]  
\arrow[d, "\cong"'] &
\displaystyle\bigg(\bigvee_{\q^\l y\in F_{\q}(v)} S^d \bigg)/G \arrow[d, "\cong"] \\
\displaystyle\bigvee_{x\in F_{\q}(u)/G} S^d
\arrow[r, "{\Phi(e/G, A/G)}"] & 
\displaystyle\bigvee_{y\in F_{\q}(v)/G} S^d
\end{tikzcd}.
\]

Note that the group is simply permuting equivalent boxes, and recall that all correspondences coming from edges of the cube are subsets of the products of their source and target.  Furthermore, all boxes remain distinct after taking the quotient.  Thus the quotient boxes $e/G$ fit into the commuting diagram as required, and the homotopy coherent diagram $\til{F}_\q/G$ can be identified with a spatial refinement of $F_\q/G$.
\end{proof}

Recall the enlarged cube category $\2^n_+$ from \ref{sec:refinements to realizations}. Any homotopy coherent diagram $D: \2^n \to \Top_*$ (resp. $D: \2^n \to \Top_*^G$) can be extended to $D^+ : \2^n_+ \to \Top_*$ (resp. $D^+ : \2^n_+ \to \Top_*^G$) by setting $D^+(x)$ to be a point for any $x\in \2^n_+ \setminus \2^n$. Take $k$-dimensional spatial refinements of $F_{\q}$, $F_1$, and $F_{\q}/G$, denoted $\til{F_{\q}}$, $\til{F_1}$, and $\til{F_{\q}/G}$ respectively (suppressing the subscript $k$). Extend each of them to diagrams $\til{F_{\q}}^+$, $\til{F_1}^+$, and $\til{F_{\q}/G}^+$ out of $\2^n_+$, and take the corresponding homotopy colimits $\lVert F_{\q}\rVert_k$, $\lVert F_1\rVert_k$, and $\lVert F_{\q}/G \rVert_k$. We also have the homotopy coherent diagram $\til{F_{\q}}/G$; its  two extensions $\til{F_{\q}}^+ /G$ and $(\til{F_{\q}}/G)^+$  are equal. 

\begin{corollary}

$ \lVert F_1 \rVert_k \simeq (\lVert F_{\q} \rVert_k) /G$.
\end{corollary}

\begin{proof}

By Proposition \ref{prop:quotient of QABF}, Proposition \ref{prop:equivariant spatial refinements}, and Lemma \ref{lem:equivariant homotopy equivalence}, there is a homotopy equivalence $\lVert F_1 \rVert_k \simeq \lVert F_{\q}/G \rVert_k $.  By Proposition \ref{prop:quotient of homotopy coherent diagram} and Proposition \ref{prop:equivariant spatial refinements}, there is also a  homotopy equivalence $\lVert F_{\q}/G \rVert_k \simeq \hocolim (\widetilde{F_{\q}}/G)^+$. Since $(\til{F_{\q}}/G)^+ = \til{F_{\q}}^+/G$, we obtain 
\[
\hocolim (\til{F_{\q}}/G)^+ 
=
\hocolim \left( \til{F_{\q}}^+ / G \right).
\]
Finally, homotopy colimits commute with the quotient functor $(-)/G$. This is clear from the definition of homotopy colimit, but is also stated explicitly as property (ho-4) in \cite[Section 4.2]{SSS}. Therefore
\[
\hocolim\left(\til{F_{\q}}^+/G\right) \cong \hocolim(\til{F_{\q}}) /G = \left(\lVert F_{\q} \rVert_k \right)/G.
\]
\end{proof}

\section{Towards lifting the $\U$ action}\label{sec:the K map}

This section concerns Conjecture \ref{Uq conjecture} on lifting the $\U$ action on quantum annular homology, constructed in \cite[Theorem D]{BPW}, to the level of spectra. We start by briefly summarizing the relevant background material; see \cite[Appendix A.1]{BPW} for more details.
Let $\U$ be the $\k$-algebra generated by $E,F,K$ and $K^{-1}$ subject to the relations 
\begin{equation}\label{eq:Uq relations}
\begin{split}
&KE = \q^2 EK \\
&KF = \q^{-2} FK 
\end{split}
\hskip2em
\begin{split} & KK^{-1} = 1 = K^{-1}K \\
& K-K^{-1} = (\q-\q^{-1})(EF - FE)
\end{split}
\end{equation}
Let $\Cs$ be a configuration consisting of $e$ essential circles and $t$ trivial circles, with corresponding standard configuration $\Cs^\c$.  Recall from Section \ref{sec:overview of quantum annular homology} that $\F_{\A_\q}(\Cs^\c) \cong V^{\otimes e} \otimes W^{\otimes t}$ carries an action of $\U$ via an identification 
\[
V^{\otimes e} \cong V_1 \otimes V_1^* \otimes V_1\otimes \cdots 
\]
where $V_1$ is the fundamental representation of $\U$, and $W$ is the trivial $2$-dimensional representation.  Fix an isotopy from $\Cs^\c$ to $\Cs$. Then $\F_{\A_\q}(\Cs)$ inherits a $\U$-action via the isomorphism $\F_{\A_\q}(\Cs^\c) \cong \F_{\A_\q}(\Cs)$.

The stable homotopy type in this paper is constructed for the modified quantum annular functor $\F_{\A_\q}^r$; see the discussion in the beginning of Section \ref{Quantum Annular Burnside section}. In what follows we will also denote by $\U$ the result of applying $(-)\otimes_\k \k_r$ to the algebra defined above. It has the same generators and relations, with the additional relation that $\q^r=1$. 

Let $F_1,F_0 : \2^n \to \B_G$ be Burnside functors. Recall from Proposition \ref{prop:identifying the cellular chain complex} (and the discussion following it) that a natural transformation $\eta : F_1\to F_0$ induces a cellular map $\LR{F_1}_k \to \LR{F_0}_k$ which agrees with the map $\Tot(\eta): \Tot(F_1) \to \Tot(F_0)$. 

Each of $E, F, K$ and $K^{-1}$ can be viewed as $\k$-linear endomorphisms of $\F_{\A_\q}^r(\Cs)$. It is natural to ask whether the generators $E, F, K^{\pm{1}}$ lift to natural endomorphisms of the quantum Burnside functor $F_\q$ constructed in Section \ref{Quantum Annular Burnside section}.  Let $J$ denote one of $E, F, K$ or $K^{-1}$. For a generator $x\in F_{\A_\q}(\Cs)$, 
\[
Jx = \sum_{y} \varepsilon_y y
\]
where the sum is over generators and each $\varepsilon_y$ is either $0$ or of the form $\pm \q^k$. Note that the appearance of negatively signed coefficients in odd Khovanov homology was dealt with by using signed correspondences (\cite[Section 3.2]{SSS}) and signed box maps (\cite[Section 4.1]{SSS}). 

Let $D$ be a diagram for an annular link with $n$ crossings, and fix a corresponding Burnside functor $F_\q:\2^n\to \B_G$ for $D$. For $u\in \2^n$, one can define the signed correspondence
\begin{equation}\label{eq:correspondences for J map}
J_u := \{(\q^k y, \q^\l x) \in F_q(u) \times F_\q(v) \mid \pm{\q^k} y \text{ appears in } J(\q^\l x) \} 
\end{equation}
from $F_\q(u)$ to $F_\q(u)$, with the obvious source and target maps. The sign map $\sigma: J_u \to \Z_2 = \{-1,1\}$ returns the sign of $\q^k y$.  Such a correspondence $J_u$ is equivariant since $J$ is $\k$-linear.
In the case when $J=K^{\pm 1}$, the signs are not needed and we have the following lifts.

\begin{proposition}\label{prop:K map}
Let $D$ be a diagram for an annular link with $n$ crossings, and let $F_\q:\2^n\to \B_G$ be a Burnside functor for $D$.  Then there is a natural isomorphism $\mathcal{K}^{\pm{1}}: F_\q \to F_\q$ which extends the correspondences $K_u^{\pm{1}}$ of \eqref{eq:correspondences for J map}.
\end{proposition}

\begin{proof}
We will use the strategy of Section \ref{sec:a strategy for constructing natural isomorphisms}.  For each $u\in\2^n$ we define the required equivariant bijection $\psi^\pm_u: F_\q(u)\to F_\q(u)$ by
\[\psi^\pm_u(\q^k x) = \q^{k \mp \adeg(x)} x\]
for all generators $x\in D_u$.  Now let $u\geq_1 v$.  The conditions \ref{nat iso condition1} and \ref{nat iso condition3} of Section \ref{sec:a strategy for constructing natural isomorphisms} have already been checked on correspondences assigned to edges $\varphi_{u,v}:u\to v$ by $F_\q$.  In order to check condition \ref{nat iso condition2}, we let $\q^k y \in F_\q(u)$, $\q^\l x \in F_\q(v)$ be elements such that  $\q^k y$  appears in $d_{v,u} (q^\l x)$. Since $d_{v,u}$ preserves annular degree, we have $\adeg(x) = \adeg(y)$, so  $\q^{k\mp \adeg(y)}y$ appears in $d_{v,u} ( \q^{\l -\mp \adeg(x)} x)$. This implies condition \ref{nat iso condition2}.

Thus we can build a natural transformation $\eta^\pm$ as in Section \ref{sec:a strategy for constructing natural isomorphisms}; the diagram \ref{eq:simplified hexagon relation} commutes, since $\psi_u$ simply multiplies generators by powers of $\q$. Finally, note that  $K^{\pm{1}}x = \q^{\pm \adeg(x)} x$, so 
\[
K^{\pm}_u = \{(\q^k x, \q^{k\mp \adeg(x)} x) \mid \q^k x \in F_\q(u)\}
\]
is naturally identified with $\eta^\pm(\edge_u)$ via $(\q^k x, \q^{k\mp\adeg(x)} x)\mapsto \q^k x$.
\end{proof}

We note that when $J=E$ or $J=F$, this overall strategy does not produce such a lift.  
Consider the saddle $S$ from Example \ref{ex4}, thought of as the cube of resolutions for a link diagram with one crossing. Let $u=1$ and $v = 0$ denote the vertices of the cube $\2$, and let $d : \F_{\A_\q}^r(D_v) \to \F_{\A_\q}^r(D_u)$ denote the differential. Let $A$ denote the correspondence $F_\q(\varphi_{u,v})$ from $F_\q(u)$ to $F_\q(v)$ assigned by the Burnside functor $F_\q$. 
\begin{center}
\includegraphics{QAH45.pdf}
\end{center}
The surgery formulas are 
\begin{align*}
d(w_-) = 0 && d(w_+) = v_+\otimes v_- + \q^{-1} v_-\otimes v_+,
\end{align*}
and actions of $E$ and $F$ are given by
\begin{align*}
& Ew_+ = 0 && E (v_+ \otimes v_-) = - v_+\otimes v_+ && E (v_- \otimes v_+) = \q v_+ \otimes v_+ \\
& Fw_+ = 0 && F (v_+ \otimes v_-) = v_-\otimes v_+ &&  F (v_- \otimes v_+) = -\q v_- \otimes v_-
\end{align*}
The correspondence $A\times_{F_\q(u)} J_u$ is empty, whereas the correspondence $J_v\times_{F_\q(v)} A$ is non-empty, containing two oppositely signed elements.

\begin{remark}
It may be possible to overcome these difficulties using a
suitably refined $G$-equivariant Burnside category. 
In the event that one has natural transformations lifting each of the $E,F$ and $K$ maps, one might also ask for some notion of a lift of the relations \eqref{eq:Uq relations}, perhaps in terms of the cones on the corresponding maps of spectra.  The authors hope to continue investigating these topics in the future.
\end{remark}

\appendix

\section{Elementary Cobordisms }

Here we compare two ways of constructing a map on quantum annular chain complexes for certain elementary annular link cobordisms $W$.  On the one hand, $W$ induces a chain map $W_*$ as defined in \cite[Equation (7.2)]{BPW}. On the other hand, for each type of elementary annular link cobordism $W$, we can define a second map $W_\bullet$ tailored towards the maps on spectra corresponding to  our constructions in this paper (mainly those in Sections \ref{sec:saddle maps} and \ref{sec:isotoping the link diagram}).  Our goal will be to show that these two maps $W_*,W_\bullet$ differ at most by some power of $\q$ in all cases.

We begin by describing the general construction of the map $W_*$.  Let $W\subset \A\times I \times [0,1]$ be a corbordism between annular links $L$ and $L'$ which intersect the membrane in $k$ and $\l$ points respectively, and let $T$, $T'$ denote the tangles obtained by cutting $L$ and $L'$ along the membrane.  Then $W$ intersects the $3$-dimensional membrane $\mu\times I \times [0,1]$ in a $(k,\l)$-tangle $P$.
As in \cite[Section 7.1]{BPW}, we represent $W$ by a tangle cobordism $\til{W}: PT \to T' P$:
\begin{equation}\label{eq:W as diagonal cob NEW}
\begin{tikzcd}
k \arrow[r, "T"] \arrow[d, "P"'] & k \arrow[d, "P"] \arrow[dl, Rightarrow, "\til{W}"']\\
\l \arrow[r, "T'"'] & \l
\end{tikzcd}
\end{equation}
The chain map $W_* : CKh_{\A_\q}(L) \to CKh_{\A_\q}(L')$ is then given in each homological grading and quantum grading by the formula (7.2) in \cite{BPW}.

To describe the map $W_\bullet$, we distinguish four types of elementary cobordisms.
\begin{enumerate}[label=\Roman*.]
\item Reidemeister moves away from the seam.
\item Morse moves away from the seam.
\item Moving an arc across the seam as in the $P^{\pm{1}}$ and $N^{\pm{1}}$ moves of Figure \ref{fig:PN moves}.
\item Moving a crossing through the seam as in Figure \ref{fig:moving crossings NEW}.
\end{enumerate}

For Type I moves, we define $W_\bullet$ as in \cite{Khovanov} and \cite{BN2}; that is, a Reidemeister move is assigned its chain homotopy equivalence.

For the remaining types of moves, let $D$ and $D'$ denote the diagrams for $L$ and $L'$ differing locally as indicated by the elementary cobordism $W$, and let $n$ be the number of crossings. For each $u\in \{0,1\}^n$, $W$ induces an annular cobordism $R_u : D_u \to D'_u$ in $\A\times[0,1]$.  The quantum annular TQFT $\F_{\A_\q}$ assigns a map to this cobordism; after tensoring with $\k_r$, we write this as
\[\F_{\A_\q}^r(R_u) : \F_{\A_\q}^r(D_u) \rightarrow \F_{\A_\q}^r(D_u').
\]
If we pick out generators for $\F_{\A_\q}^r(D_u)$ and $\F_{\A_\q}^r(D_u')$ via cobordisms from standard configurations as in Section \ref{sec:fixing generators}, then the image of these generators under this map can be computed by composing cobordisms.  We will use shifted copies of this map to define the components $W_{\bullet u} : \F^r_{\A_\q}(D_u) \to \F^r_{\A_\q}(D'_u)$ of $W_\bullet$ on each smoothing individually.  Note that this is precisely how the natural isomorphisms of the quantum annular Burnside functors are determined in Section \ref{sec:isotoping the link diagram} (although there we omitted the functor $\F_{\A_\q}^r$ from the notation).

If $W$ is of Type II or Type III, we define $W_{\bullet u}$ on each smoothing to be $\F_{\A_\q}^r(R_u)$.  Notice that for Type II moves, this is equivalent to defining $W_\bullet$ as in \cite{Khovanov} and \cite{BN2} where Morse moves are assigned either the unit, saddle map, or counit on each smoothing, corresponding to $0$-handle, $1$-handle, and $2$-handle attachments respectively.

Finally, if $W$ is of Type IV, then there are four cases to consider depending on the type of crossing and the direction of movement across the seam.
\begin{figure}
\includegraphics{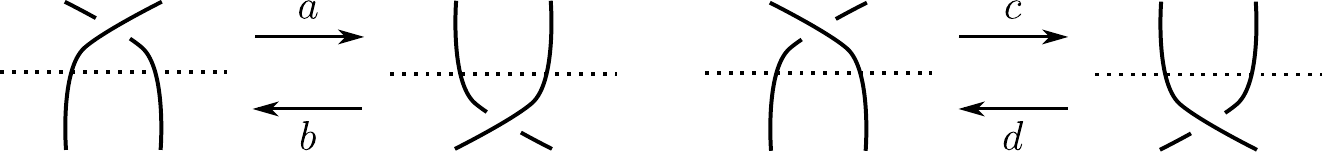}
\caption{}\label{fig:moving crossings NEW}
\end{figure}
In all of these cases, we will define \begin{equation}\label{eq:Wbullet NEW}    
W_{\bullet u} := \q^a \F^r_{\A_\q}(R_u),
\end{equation}
for some power $a\in \Z$ which is determined by the resolution of the crossing near the seam.  If the smoothing corresponding to $u$ resolves this crossing into two parallel lines each intersecting $\mu$ once, we set $a:=0$ in the formula \eqref{eq:Wbullet NEW} for $W_{\bullet u}$, so that $W_{\bullet u}=\F_{\A_\q}^r(R_u)$ once again.  Note that $R_u$ is just the identity cobordism in this case, so $W_{\bullet u}$ is the identity map.  Otherwise, we set $a:=1$ for the moves $(a)$ and $(d)$, and $a:=-1$ for moves $(b)$ and $(c)$.

Note that in all cases, $W_\bullet$ is a chain map; for Type I and II moves this follows from the definitions, while for Type III and IV moves this follows from the trace relations in Figure \ref{fig:trace moves}.   Note also that, for Type III and IV moves, the maps $W_\bullet$ defined here are precisely the totalizations of the natural isomorphisms built in Proposition \ref{prop:isotoping the link diagram}.

\begin{lemma}\label{lem:elementary cobordisms}
Let $L$ and $L'$ be annular links and let $W: L \to L'$ be an elementary cobordism. 
\begin{enumerate}[label=\emph{(\arabic*)}]
\item If $W$ is of Type I, II, or III, then $W_* = W_\bullet$. 
\item If $W$ is Type IV, then $W_* = \q^m W_\bullet$ for some $m$ which depends only on the sign of the crossing involved in the move. 
\end{enumerate}
\end{lemma}
\begin{proof}
Throughout the proof, we will use $C(-)$ to denote the Chen-Khovanov complex of bimodules, which is denoted by $C_{CK}(-)$ in \cite[Section 5.5]{BPW}.

The first thing to notice is that, in any case where the intersection tangle $P=W\cap \mu\times I \times [0,1]$ has no crossings, the formula \cite[(7.2)]{BPW} for $W_*$ simplifies drastically.  There is no summation over indices $i'$ since $C(P)$ has only one term.  In all such cases (which include Types I, II, and III here), one shows that $W_*$ is equal to $W_\bullet$ by direct comparison.

Finally, for elementary cobordisms $W$ of Type IV, we focus on the case (a) from Figure \ref{fig:moving crossings NEW}.  Observe that the tangle $P$ in this case has a single crossing, that $T$ and $T'$ can be written as $T =  T''P$ and $T' = PT'' $, and that $\til{W}: PT \to T'P$ is the identity cobordism.  All of this is illustrated in Figure \ref{fig:tangles NEW}.

\begin{figure}
\centering
\includegraphics{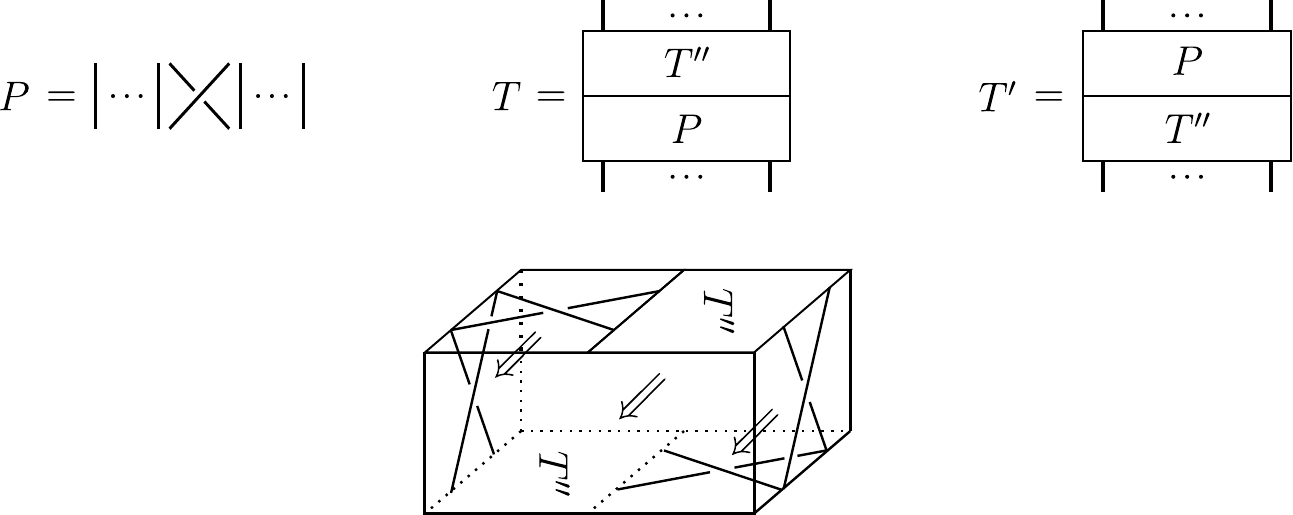}
\caption{}\label{fig:tangles NEW}
\end{figure}

Using $\otimes$ to denote the tensor product over the relevant Chen-Khovanov arc algebra, the map $\til{W}_*: C^i(T) \otimes C^{i'}(P) \to C^{i'}(P) \otimes C^{i}(T')$ is the identity on the summand $C^{i'}(P) \otimes C^{i-i'}(T'') \otimes C^{i'}(P)$ which appears in both $C^i(T) \otimes C^{i'}(P)$ and $ C^{i'}(P) \otimes C^{i}(T')$, and $\til{W}_*$ is $0$ on the other summands.  In particular, there is no need for summing over various $i'$ in the formula \cite[(7.2)]{BPW} for $W_*$, and the only possible difference between $W_*$ and $W_\bullet$ acting on any given generator is in the use of the third map $\theta$ in \cite[(7.2)]{BPW} which permutes the tensor factors and multiplies generators $x\otimes y \otimes \alpha$ by a power of $\q$ according to the grading of $x$ as defined in \cite[Section 5.5]{BPW}.

To analyze this potential difference, let $P_0,P_1$ denote the 0- and 1-resolutions of $P$.  When constructing $W_\bullet$ on each resolution, we view $x$ as living in either $\F_{CK}(P_0)$ or $\F_{CK}(P_1)$. If $x\in\F_{CK}(P_1)$, there is an extra factor of $\q^1$ in  our map (recall that we are considering case (a) amongst the Type IV moves; see the paragraph following \eqref{eq:Wbullet NEW}).  However, in the definition of $W_*$, we view $x$ as living in 
\[\F_{CK}(P) = \big[ \F_{CK}(P_0)\{m\} \rightarrow \F_{CK}(P_1)\{m+1\} \big],\]
where the grading shift $m$ depends on the sign of the crossing.  And so the map $\theta$ multiplies generators by an extra overall factor of $\q^m$ when defining $W_*$ as compared to how it would act when defining $W_\bullet$, as desired.

The proof for case $(c)$ of Figure \ref{fig:moving crossings NEW} is similar. For the cases $(b)$ and $(d)$, note that the map $W_\bullet$ is precisely the inverse of the map defined in $(a)$ and $(c)$, respectively. Moreover, the cobordisms $W$ of moves $(b)$ and $(d)$ are inverses (in the category $Links_\q(\A)$, see \cite[Proposition 6.8]{BPW}) to the cobordisms of $(a)$ and $(c)$, respectively. This completes the proof.

\end{proof}

\end{document}